\documentclass[12pt]{amsart}

\usepackage[utf8]{inputenc}
\usepackage{amsmath}
\usepackage{amsfonts}
\usepackage{amssymb}
\usepackage{amsthm}
\usepackage{float}
\usepackage{mathtools}
\usepackage{csquotes}
\usepackage[numbers]{natbib}
\usepackage{pdfpages}
\usepackage{enumitem}
\usepackage{graphicx,subcaption}

\usepackage[normalem]{ulem}

\usepackage{hyperref}
\usepackage{nameref,cleveref}

\RequirePackage{doi}

\usepackage[margin=3cm]{geometry}

\usepackage{tikz-cd}
\usetikzlibrary{arrows,matrix,shapes,decorations.text, mindmap,shapes.misc}
\usepackage{metalogo}

\newcommand{\N}{\mathbb{N}}

\renewcommand{\S}{\mathbb{S}}

\newcommand{\R}{\mathbb{R}}

\newcommand{\ip}[2]{\left\langle #1,#2\right\rangle}

\newcommand{\cG}{\mathcal{G}}
\newcommand{\sgn}{\mathrm{sgn}}
\newcommand{\er}{\mathcal{O}}
\newcommand{\calL}{{\mathcal{L}}}
\newcommand{\calO}{{\mathcal{O}}}

\newcommand{\Ldot}{\dot{L}}
\newcommand{\Ptilde}{\tilde{P}}
\newcommand{\calP}{\mathcal{P}}

\definecolor{babypink}{rgb}{0.96, 0.76, 0.76}
\definecolor{brightlavender}{rgb}{0.75, 0.58, 0.89}
\definecolor{lightfuchsiapink}{rgb}{0.98, 0.52, 0.9}
\definecolor{turquoise}{rgb}{0.19, 0.84, 0.78}
\definecolor{royalpurple}{rgb}{0.47, 0.32, 0.66}
\definecolor{seagreen}{rgb}{0.18, 0.55, 0.34}

\newtheorem{Thm}{Theorem} 
\newtheorem*{WT}{Winding Theorem}
\newtheorem*{FT}{Focussing Theorem}

\newtheorem{Quest}{Question}
\newtheorem{Lem}[Thm]{Lemma}
\newtheorem{Cor}[Thm]{Corollary}
\newtheorem{Prop}[Thm]{Proposition}

\theoremstyle{definition}

\newtheorem{Rem}[Thm]{Remark}

\newcommand{\calV}{{\mathcal{V}}}
\newcommand{\ymin}{{y_{\min}}}

\newcommand{\vtilde}{\tilde v}

\usepackage[toc,page]{appendix}
\usepackage{dsfont}

\newcommand{\calH}{\mathcal{H}}

\newcommand{\const}{{\operatorname{const}}}

\newcommand{\vh}{v_{\mathrm{hor}}}

\newcommand{\etadot}{\dot{\eta}}

\newcommand{\ydot}{\dot{y}}
\newcommand{\zdot}{\dot{z}}

\newcommand{\zbar}{\overline{z}}
\newcommand{\ybar}{\overline{y}}

\newcommand{\Xtilde}{\tilde X}

\newcommand{\Rbar}{\overline{\R}}

\newcommand{\eps}{\varepsilon}

\newcommand{\trans}{{\mathrm{trans}}}

\newcommand{\ff}{{\mathrm{ff}}} 

\newcommand{\angl}{\operatorname{length}_Y}  

\newcommand{\pz}{\partial_z}

\newcommand{\Span}{\operatorname{span}}
\newcommand{\Rplus}{\R_+}
\newcommand{\id}{\operatorname{id}}

\makeatletter
\@namedef{subjclassname@2020}{%
  \textup{2020} Mathematics Subject Classification}
\makeatother

\makeatletter
\newcommand{\myref}[1]{\cref{#1}\mynameref{#1}{\csname r@#1\endcsname}}
\newcommand{\Myref}[1]{\Cref{#1}\mynameref{#1}{\csname r@#1\endcsname}}

\def\mynameref#1#2{%
  \begingroup
    \edef\@mytxt{#2}%
    \edef\@mytst{\expandafter\@thirdoffive\@mytxt}%
    \ifx\@mytst\empty\else
    \space(\nameref{#1})\fi
  \endgroup
}
\makeatother

\title{Winding and focussing for geodesics passing a thin cuspidal neck}


\author{Daniel Grieser}
\address{Institut f\"ur Mathematik, Carl von Ossietzky Universit\"at Oldenburg}
\email{daniel.grieser@uni-oldenburg.de}
\author{J\o rgen Olsen Lye}
\address{Institut für Differentialgeometrie, Leibniz Universit\"at Hannover}
\email{joergen.lye@math.uni-hannover.de}

\date{\today}
\begin{document}

\pagestyle{plain}

\begin{abstract}
We study geodesics on a family $(M_\eps)$ of manifolds that have a thin neck,  which degenerate to a space with an incomplete cuspidal singularity as $\eps\to0$. There are essentially two classes of geodesics passing the waist, i.e.\ the cross section where the neck is thinnest: 1. Those hitting the waist almost vertically. We find that these exhibit a surprising focussing phenomenon as $\eps\to0$: certain exit directions will be preferred, for a generic limiting singularity.
2. Those hitting the waist obliquely at a uniformly non-vertical angle. They wind around the neck more and more as $\eps\to0$. We give a precise quantitative description of this winding. We illustrate both phenomena by numerical solutions.
Our results rest on a detailed analysis at the two relevant scales: The points whose distance to the waist is of order $\eps$, and those much farther away. This multiscale analysis is efficiently expressed in terms of blow-up.
\end{abstract}

\maketitle

\tableofcontents

\phantomsection

\keywords

\section{Introduction}
\label{Section:Intro}
 In \cite{GrGr15} the first author and Grandjean consider manifolds with a cuspidal singularity and study the geodesics that start at the singularity. In this paper we use and extend these results to  families of smooth Riemannian manifolds $M_\eps$, depending on a parameter $\eps>0$, that develop a thin neck and collapse to a manifold with a cuspidal singularity as $\eps\to0$. A typical example is the family of surfaces with elliptic cross section of eccentricity $\delta$,
\begin{equation}
\label{eqn:example}
M_\eps=\left\{u^2 + \frac{v^2}{1-\delta^2} = z^{2k}+\eps^{2k}\right\}\subset \R^3
\end{equation}
with fixed $k\geq 2$ and $\delta\in [0,1)$. The limit space $M_0$ has a cuspidal singularity at $u=v=z=0$.
Figure \ref{fig:SurfaceExamples} shows $M_\eps$ for $k=2$ and $\eps=1$ resp.\ $\eps=0.2$. In both cases $\delta=0.8$.

\begin{figure}[ht]
\includegraphics[width=.48\textwidth]{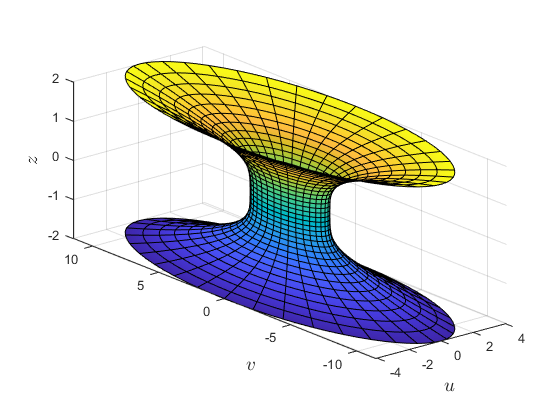}\hfill
    \includegraphics[width=.48\textwidth]{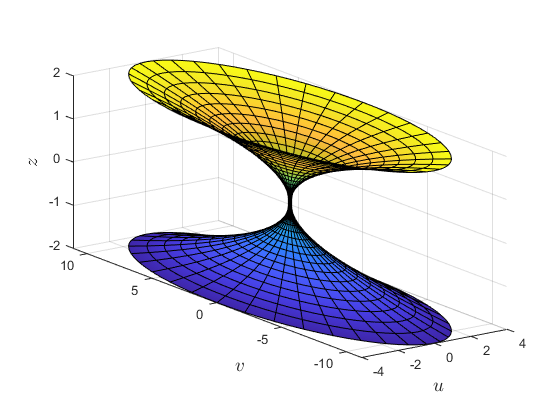}
    \caption{Examples of surface \eqref{eqn:example} with $k=2$ and $\eps=1$ and $\eps=0.2$ respectively. Here $\delta=0.8$.}
    \label{fig:SurfaceExamples}
\end{figure}

We study the geodesics that pass through the 'waist'\footnote{In the physics literature, an object like $M_\eps$ would probably be called a worm hole, in which case the waist is usually called a 'throat'.}
 of $M_\eps$ at $z=0$, uniformly as $\eps\to0$. 
The \textbf{angle of impact} is defined as \label{phiDef1}
$$ \varphi = \text{angle of the geodesic with the waist }\ \  \in[0,\tfrac\pi2]\,. $$
We first explain our results in the case of the family \eqref{eqn:example}. They apply to other than elliptic cross sections and to higher dimensions, see below for the precise setting.

There are two notable types of behaviour that geodesics exhibit on such a family of spaces (here, a 'fixed' quantity is one that is independent of $\eps$, as $\eps\to0$):

\begin{itemize}
 \item
 Winding:
 A geodesic starting at the waist $z=0$ at a fixed angle of impact $\varphi<\pi/2$ will wind around often before reaching a fixed positive $z$. See Figure \ref{fig:WindingExamples}.
The Winding Theorem below 
states that it winds $\sim C\eps^{-(k-1)}$ times, with an explicit constant $C$ depending on $\varphi$. We also give a uniform description of the number of revolutions as this angle tends to $0$ or $\frac{\pi}{2}$.
\item
Focussing:
Consider geodesics starting at the waist vertically, and where they arrive when they reach the cross sectional ellipse at $z=1$.
If $\delta=0$ then their arrival points are clearly evenly distributed, by rotational symmetry. However, if $\delta\neq0$ then a surprising focussing phenomenon occurs: most of these geodesics will veer strongly towards two distinguished points on that ellipse. These points are the minima of the distance to the $z$-axis on the curve.
See the Focussing Theorem below and Figure \ref{fig:EllipsePlots}. Focussing also happens for geodesics leaving the waist almost vertically, i.e.\ with $\varphi\to\pi/2$ sufficiently fast as $\eps\to0$.
\end{itemize}
The winding result relates to \cite{GriLye}. 
The focussing result yields an interpretation of some of the results of \cite{GrGr15} in terms of smooth approximations of the cuspidal singularity.



\subsection{Main results}
\subsubsection{Setup}
\label{CuspDef}
To motivate our general setup, consider the standard cuspidal singularity \eqref{eqn:example} with $\delta=0$ and $\eps=0$. Parametrizing it by $\R\times \S^1\to M_0$, $(z,\phi)\mapsto (z^k\cos\phi,z^k\sin\phi,z)$ we get, after a short calculation, that the Euclidean metric restricted to $M_0$ is
\begin{equation}
\label{eqn:standard cusp}
 A\, dz^2 + z^{2k}\,d\phi^2\,,\quad A = 1 + k^2 z^{2k-2}\,.
\end{equation}
We show in Appendix \ref{Section:GeneratingExamples} that 
for $\eps>0$ the factor $z^{2k}$ turns, not surprisingly, into ${z^{2k}+\eps^{2k}}$, and that for $\delta>0$ additional mixed terms arise. The result motivates the following setup, see also Remark \ref{Rem:motivation setup} below.

Let $Y$ be a smooth compact $(n-1)$-dimensional manifold and $I$ some open interval containing $0$.  Let \label{Mdef}
\[M = I\times Y.\]
We denote the $I$-variable by $z$ and points \label{zyDef} on $Y$ by $y$. We think of $I$ as vertical and of $Y$ as horizontal.

Fix $k\geq2$. We equip $M$ with a family of Riemannian metrics $g_\eps$, depending on a parameter $\eps\in(0,1)$ and extending continuously to $\eps=0$, where $g_0$ is only semi-positive definite at $z=0$:
\begin{equation}
g=g_{\eps}=(1-w^{\kappa}S)\,dz^2+2w^{2k}\, dz\cdot b +w^{2k}\, h\,.
\label{eq:MetricAnsatz}
\end{equation}
Here, $\kappa\geq k$ is an integer, and $S$, $b$ and $h$ are a function, one-form and metric on $Y$, respectively, which may depend on  $\eps, z$ as parameters.
The main feature of $g$ is the \lq degeneration \label{wdef1} factor\rq\ $w$, which is a continuous function of $(\eps,z)$ that is smooth and positive for $(\eps,z)\neq(0,0)$ and vanishes at $(0,0)$.
We invite the reader to keep $w=\sqrt[2k]{z^{2k}+\eps^{2k}}$ in mind as a guiding example. 

The last term of $g$ indicates that the horizontal slice $Y$ at height $z$ and parameter value $\eps$ is scaled by the factor $w(\eps,z)^k$, so this may be thought of roughly as the \textbf{diameter} of this slice.
We assume
\begin{equation}
\label{eqn: w monotone}
\mathrm{sgn}\left(\frac{\partial w}{\partial z}\right)=\mathrm{sgn}(z) 
,\quad  w(\eps,0)=\eps,
\end{equation}
which in particular means that $(M,g_\eps)$ is narrowest at $z=0$, which we therefore call the \textbf{waist}. This waist has diameter $\eps^k$.  Furthermore, we assume
\begin{equation}
\label{eqn:w norm eps=0} 
w(0,z)=\vert z\vert,
\end{equation}
which means that $(M,g_0)$ has a cuspidal singularity of order $k$ at $z=0$, compare \eqref{eqn:standard cusp}.
Finally, we express the uniformity of the degeneration by requiring that $w$ is 1-homogeneous in $(\eps,z)$, i.e.,
\begin{equation}
\label{eqn:w homogeneous}
w(\lambda \eps,\lambda z) = \lambda\, w(\eps,z)\quad \text{for all $(\eps,z)$ and $\lambda>0$}\,.
\end{equation}

See Section \ref{sec:setting} for more precise regularity assumptions on $S$, $b$ and $h$.
For our results stated below the interval $I$ and $\eps$ will have to be chosen sufficiently small.

An important special case of \eqref{eq:MetricAnsatz} are families of metrics of \textbf{warped product} type, i.e.,
\begin{equation}
\label{eqn:def warped product new}
g = (1-w^\kappa S) dz^2 + w^{2k}\,h\quad\text{with $h$ independent of $z$ and $S$ independent of $y$,}
\end{equation}
compare \eqref{eqn:standard cusp}.
The main feature distinguishing warped products from the general case \eqref{eq:MetricAnsatz} is the $y$-dependence of the function $S$ in the coefficient of $dz^2$. In comparison, it turns out that the mixed term involving $b$ is essentially negligible due to its coefficient $w^{2k}$.
The distinction between warped products and the general case is essential for the focussing theorem but not for the winding theorem. 

Note that in the warped product case the factor $1-w^\kappa S$ could be removed by redefining the $z$-variable, but this would create  complications in the $\eps$-dependence, so we keep the form \eqref{eqn:def warped product new}. Similarly, in the general case the mixed term may be removed by a change of variables, but this would again lead to a reduced regularity of the other coefficients at $\eps=z=0$, so we decided to work with the general form \eqref{eq:MetricAnsatz}. 

\begin{Rem}[Motivation for our setup]
\label{Rem:motivation setup} 
In Appendix \ref{Section:GeneratingExamples} we show that the spaces \eqref{eqn:example} are an example of our general setting, with $\kappa=2k-2$:  the metric on $M_\eps$ has the form \eqref{eq:MetricAnsatz} in a suitable parametrization (in particular, after making a change of the $z$-variable). For $\delta=0$ it is of warped product type.

Another motivation for considering metrics of the form \eqref{eq:MetricAnsatz} (in particular, the restriction on $\kappa$) is that, for a space with a cuspidal singularity of order $k$, embedded in a Riemannian manifold, the induced metric has this form in suitable coordinates (where $\eps=0$). See \cite{BeyGri:IGIC} (where it is proven that one can always choose $\kappa\geq k$, and that this is optimal; in \cite[Proposition 7.3]{GrGr15} this was claimed with $\kappa=2k-2$, but this is erroneous in general if $k\geq3$).

Note that if $\kappa$ is larger than $2k-2$, the additional $w$ factors may be absorbed into $S$, making $S$ vanish at $w=0$. 
\end{Rem}

We investigate the geodesics $\gamma$ (always assumed to have \textbf{unit speed}) on $(M,g_\eps)$ starting at a point of the waist $\{z=0\}=\{0\}\times Y$. We write 
\begin{equation}
\label{eqn:intro gamma}
\gamma(t)=(z(t),y(t))\,,\quad z(0)=0\,.
\end{equation}
W.l.o.g. we consider upward moving geodesics, i.e., we assume $\zdot(0)\geq0$. Also, recall that we denote by $\varphi$ the angle of $\gamma$ with the waist $\{z=0\}$.

Motivated by the case of $Y=\S^1$ we call $\ydot$ (or sometimes $|\ydot|$, where $|\cdot|$ denotes length with respect to the metric $h_{\gamma(t)}$) the \textbf{angular velocity} of $\gamma$. Note that the \textbf{actual (or metric) horizontal velocity} is $w^k\ydot$.

\subsubsection{Winding}
Our first main theorem concerns the total length of the $Y$-component of geodesics $\gamma$. We call this the \textbf{angular length} of $\gamma$ and denote it by
\[ \angl (\gamma) \coloneqq \int_0^{T} |\ydot(t)|\, dt \]
where $T$ is the time for which $z(T)=z_0$, for some small\footnote{depending on $S,b,h$} 
but fixed $z_0>0$. 
If $(Y,h)=(\S^1,d\phi^2)$ then the angular length is essentially $2\pi$ times the number of times that $\gamma$ winds around $\S^1$.

\begin{WT}[Rough version of Theorem \ref{Thm:WindingPhi}]
\crefname{}{Winding Theorem}{Winding Theorems}
\Crefname{}{Winding Theorem}{Winding Theorems}
\label{thm:intro winding}
A geodesic on $(M,g_\eps)$ starting at the waist upward at an angle $\varphi\in(0,\frac\pi2)$ to  the horizontal has angular length
\begin{equation}
\label{eqn:intro winding asymp}
\angl (\gamma) \sim \frac{\mathcal{C}_{\varphi}\cos\varphi}{\eps^{k-1}}. 
\end{equation}

The constant $\mathcal{C}_{\varphi}$ is decreasing in $\varphi$, $\mathcal{C}_{\varphi}\geq \mathcal{C}_{\pi/2}>0$, and depends only on $\varphi$ and the function $w$. The asymptotics \eqref{eqn:intro winding asymp} are uniform in $\varphi$ for  compact subsets of $(0,\pi/2)$.
\end{WT}

Here the asymptotics $\sim$ means that the quotient of left hand side and right hand side tends to one as $\eps\to0$, uniformly as functions on the 
set of all geodesics (for all metrics $g_\eps$, $\eps>0$) satisfying the stated conditions. We also give an estimate for the error term in Theorem \ref{Thm:WindingPhi}, uniformly to $\varphi=\frac\pi2$, see also \eqref{eqn:ang length warped phi}. 

There is a detailed description of $\mathcal{C}_{\varphi}$
as $\varphi\to0$, see \eqref{eqn:Cphi asymp}.

\begin{figure}

\begin{subfigure}[t]{0.48\textwidth}
\includegraphics[width=\textwidth]{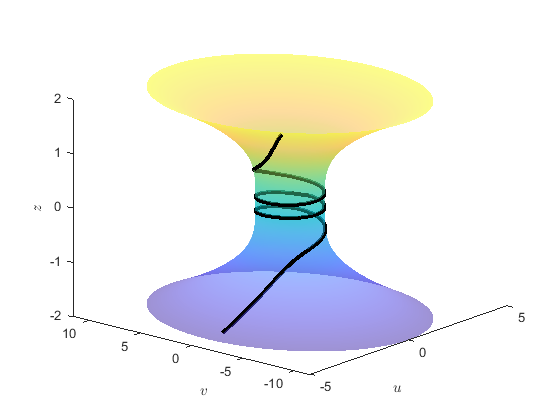}\hfill
    \caption{A geodesic passing the waist $z=0$ on the surface \eqref{eqn:example} with $k=2$, $\delta=0.8$, $\eps=1$.}
\end{subfigure}   
\hfill
\begin{subfigure}[t]{0.48\textwidth}
\includegraphics[width=\textwidth]{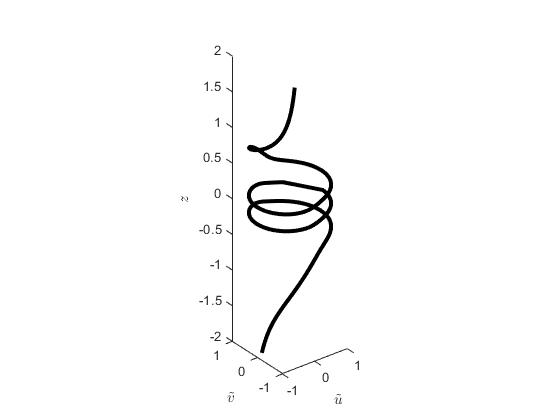}
    \caption{The geodesic of the left hand picture up close by showing $(\tilde{u},\tilde{v},z)=\left(\frac{u}{\vert (u,v)\vert},\frac{v}{\vert (u,v)\vert},z\right)$.}
\end{subfigure}

\medskip

    \begin{subfigure}[t]{0.48\textwidth}
\includegraphics[width=\textwidth]{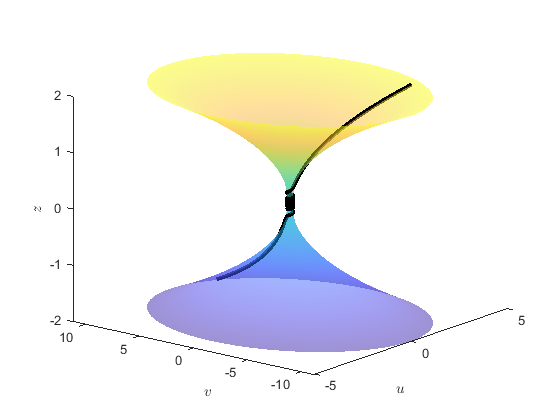}
    \caption{A geodesic passing the waist $z=0$ on the surface \eqref{eqn:example} with $k=2$, $\delta=0.8$, $\eps=0.3$.}
\end{subfigure}    
   \hfill     
    \begin{subfigure}[t]{0.48\textwidth}
\includegraphics[width=\textwidth]{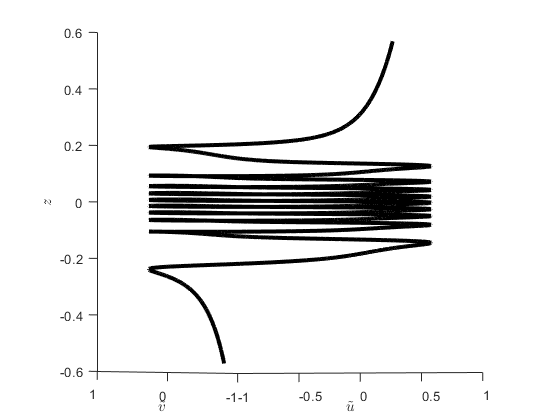}
    \caption{The geodesic of the left hand picture up close by showing $(\tilde{u},\tilde{v},z)=\left(\frac{u}{\vert (u,v)\vert},\frac{v}{\vert (u,v)\vert},z\right)$.}
\end{subfigure}

\caption{Winding: The surface \eqref{eqn:example} with $k=2$, $\delta=0.8$, and $\eps=1$ and $\eps=0.3$ respectively are shown on the left. Each picture shows a single unit speed geodesic passing the waist $z=0$ at an angle $\phi=\arccos 0.95$. The right hand pictures show the geodesics up close by stretching $(u,v)$ to lie on the unit circle. Numerical solutions.}
    
 \label{fig:WindingExamples}
\end{figure}

\subsubsection{Focussing}
The Winding Theorem says that the horizontal length of geodesics hitting the waist too obliquely diverges as $\eps\to 0$. We now turn to the complementary situation of almost vertical geodesics, more precisely
\[\tfrac\pi2 -\varphi =\er(\eps^{k-1}).\]

We need two more assumptions\footnote{We do not know if  these are necessary for the result, but we need them for our method to work. See Subsection \ref{subsec:further remarks}.}. First, we need to assume $\kappa\geq 2k-2$. This is in general a restriction when $k\geq 3$ (not for $k=2$), but all the examples in Appendix \ref{Section:GeneratingExamples}, in particular \eqref{eqn:example}, satisfy this for all $k$. Secondly, we need to assume that $S$ restricts to a Morse function on $Y$ when\footnote{Recall that, by definition, a function is Morse if its Hessian is non-degenerate at all critical points of the function. In particular, critical points must be isolated.}\footnote{\label{footnote:Spm}This restriction is well-defined in our example \eqref{eqn:example}. In general, $S$ need not be continuous at $(\eps,z)=(0,0)$, but only its pull-back to a blown-up space is, see Section \ref{subsec:b-up spaces} and Figure \ref{Fig:Blowup}. 
In this general case the assumption is that $S^\pm=S_{\vert \ff\cap M_{\pm}}$, which are the limits of $S_{|\eps=0}$ from $z>0$ and $z<0$, respectively, should be Morse. In principle, these two restrictions can be different functions on $Y$, but we ignore this for the introduction.}   $(\eps,z)=(0,0)$.
We show in Appendix \ref{Section:GeneratingExamples} that this is the case in the example  \eqref{eqn:example} if $\delta>0$.
 We  write 
 $$S_0=S_{\vert (\eps,z)=(0,0)}\,.$$

The focussing relates to certain geodesics on the space $(M,g_0)$, which has a cuspidal degeneration at $z=0$.  
It is proved in \cite{GrGr15} that for each minimum $y_{\min}$ of $S_0$ there is a unique continuous curve \eqref{eqn:intro gamma} with $t\in[0,t_0)\mapsto z(t)$ strictly increasing, which is a geodesic for $t>0$ and starts at $(0,y_{\min})$, i.e.\ 
$y(0)=y_{\min}$. We denote this geodesic by $\gamma_{y_{\min}}$.


\begin{FT}[Rough version of Theorem \ref{Thm:focussing} and Corollary \ref{Cor:Focussing}]
\label{Focussing Theorem} 
\label{Thm:IntroAlmostVertical}
Assume that $S_0$ is a Morse function and that $\kappa=2k-2$.
Consider a family of unit speed geodesics $\gamma_{\eps}=(z_{\eps},y_{\eps})$ passing the waist $z=0$ at $t=0$, with $(y_{\eps}(0),\eps \dot{y}_{\eps}(0))$ converging to some $(y_0,\vtilde_0)\in TY$ as $\eps\to0$.
    Then for generic limit $(y_0,\vtilde_0)$, the point
$\gamma_\eps(t)$  will approach $\gamma_{y_{\min}}(t)$ as $\eps\to0$, for some minimum $\ymin$ of $S_0$, for any fixed $t>0$.
\end{FT}
Corollary \ref{Cor:Focussing} is formulated in terms of a kind of $z$-parametrisation, where $\gamma^{z_1}$ means the point where $\gamma$ intersects $z=z_1$.  
Note that $\eps \dot{y}_{\eps}(0)$ equals $\eps^{-(k-1)}$ times the initial horizontal velocity $\eps^k \dot{y}_{\eps}(0)$ (see the remark after \eqref{eqn:intro gamma}), so its convergence implies that the horizontal velocity and hence 
$\frac\pi2-\varphi_{\eps}$ are of order $\eps^{k-1}$. Thus these are (initially) almost vertical geodesics.

This focussing effect is illustrated in Figure \ref{fig:EllipsePlots}. In the Example \eqref{eqn:example}, the function $S_0$ is a multiple of the distance to the $z$-axis, so is a Morse function if and only if $\delta\neq 0$. The curves $\gamma_{y_{\min}}$ are (upper parts of) the intersections of $M_0$ with the $v=0$ plane.
For a circular cross section the focussing effect disappears, see Figure \ref{fig:CirclePlot}.

In Remark \ref{Rem:explan eps to 0} we explain how the Focussing Theorem yields an explanation of some of the results of \cite{GrGr15}.

\begin{figure}

\begin{subfigure}[t]{0.48\textwidth}
\includegraphics[width=\textwidth]{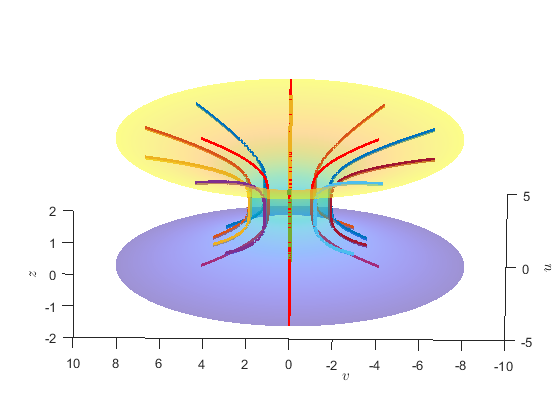}\hfill
    \caption{$\eps=1$.}
\end{subfigure}   
\hfill
\begin{subfigure}[t]{0.48\textwidth}
\includegraphics[width=\textwidth]{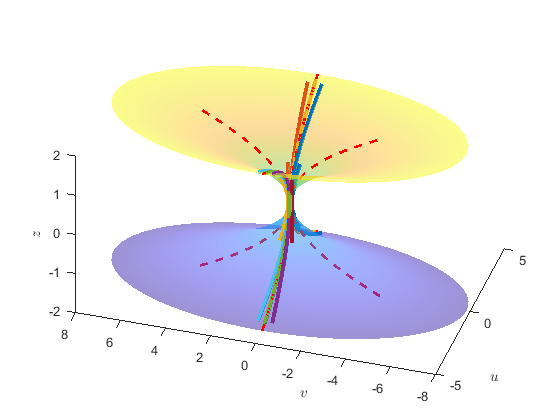}
    \caption{$\eps=0.2$.}
\end{subfigure}

\medskip

    \begin{subfigure}[t]{0.48\textwidth}
\includegraphics[width=\textwidth]{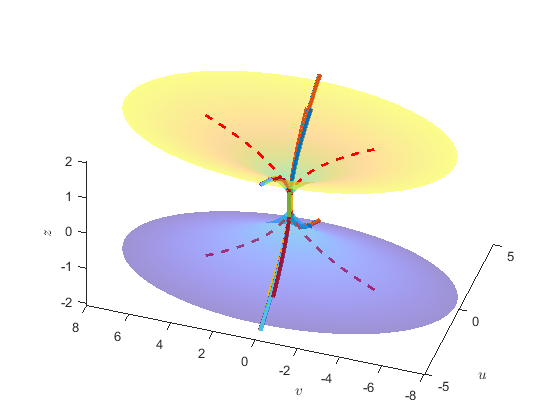}
    \caption{$\eps=0.1$.}
\end{subfigure}    
   \hfill     
    \begin{subfigure}[t]{0.48\textwidth}
\includegraphics[width=\textwidth]{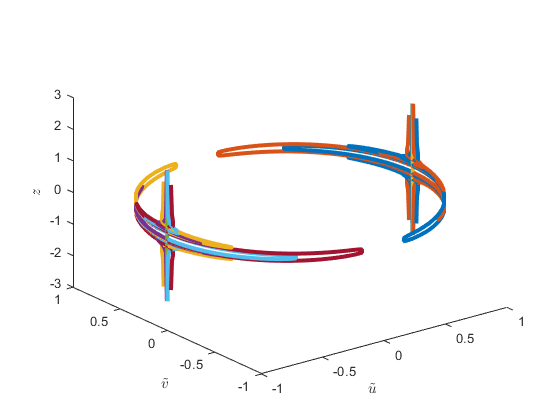}
    \caption{The behaviour near the waist of the left hand picture by showing $(\tilde{u},\tilde{v},z)=\left(\frac{u}{\vert (u,v)\vert},\frac{v}{\vert (u,v)\vert},z\right)$.}
\end{subfigure}

\caption{Focussing: The surface \eqref{eqn:example} with $k=2$, eccentricity $\delta=0.7$, and $\eps=1$, $\eps=0.2$, $\eps=0.1$ respectively are shown. Each has ten  unit speed geodesics passing the waist $z=0$ vertically, $\varphi =\pi/2$. The points of intersection  with the waist are uniformly distributed. The lower right hand picture shows the $\eps=0.1$ plot up close by stretching $(u,v)$ to lie on the unit circle. The dashed red lines are the lines $u=0$ and $v=0$, corresponding to the critical points of $S_0$. The points   $u=0$ are maxima, whereas $v=0$ are minima. Numerical solutions.}
  \label{fig:EllipsePlots}  
    
\end{figure}

\begin{figure}[ht]
\includegraphics[width=.6\textwidth]{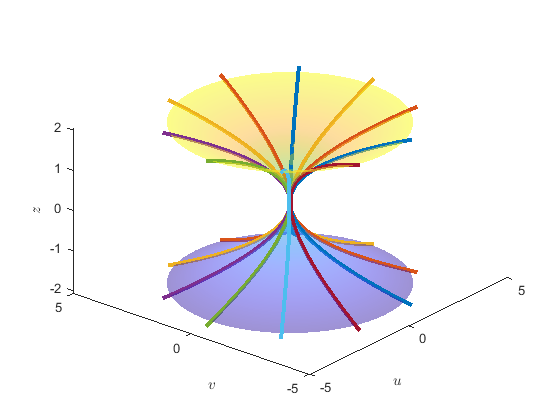}
\caption{Geodesics passing the waist $z=0$ on $M_{\eps}$ for $k=2$ with  $\delta=0$ and $\eps=0.1$. Here ten geodesics hitting the waist $z=0$ vertically, $\cos\varphi=0$, are shown. The points of intersection with the waist are uniformly distributed.}
    \label{fig:CirclePlot}
\end{figure}

\subsection{Main ideas}
\label{Section:MainIdeas}

We explain the main ideas in the case of the family $(M_\eps)$ in \ref{eqn:example}. To analyse the behaviour of geometric objects (like geodesics) on $M_\eps$ uniformly as $\eps\to0$ it is useful to think of $M_\eps$ as composed of two \lq regions\rq. These are not well-defined subsets but are described by orders of magnitude, so we call them \textit{regimes}:
\footnote{We use the following notation: $z\lesssim\eps$ means $z\leq C\eps$  for a positive constant $C$, independent of $\eps$, which may change from occurrence to occurrence, depending on previously introduced constants (in words: $z$ is at most of the order of $\eps$). $w\approx \eps$ means $w\lesssim \eps$ and $\eps \lesssim w$ (so $w,\eps$ are of the same order). Finally, $z\gg \eps$ means that $\neg(z\lesssim\eps)$ (so $z$ is large compared to $\eps$). As before, $\sim$ means that the quotient of left-hand side and right hand side tends to one as $\eps\to0$.
 }
\begin{enumerate}
\item For $z \lesssim \eps$ we have $w\approx \eps$, so the diameter of the cross section at height $z$ is $\approx \eps^k$.
We call this the \textbf{neck regime}.
\item
For $z\gg \eps$ we have $w\sim z$, so the diameter is $\sim z^k \gg \eps^k$. We call this the \textbf{bulk regime}.
\end{enumerate}
The idea of regimes is made precise by the concept of blow-up, as explained in Section \ref{subsec:b-up spaces}, see Remark \ref{Rem:bhs regimes} in particular\footnote{See also \cite{GriQuasimodes} for more explanations concerning regimes, scales, and blow-up.}. 
This will be needed for the Focussing Theorem, but not for the Winding Theorem.\smallskip

The \textbf{Winding Theorem} may be understood in terms of the neck regime:
consider a geodesic starting at $z=0$ at an angle $\varphi_0\in(0,\frac\pi2)$ to the waist. 
\begin{enumerate}
\item[(a)] 
 After one revolution $\gamma$ has travelled a distance $\approx\eps^k$, so its $z$-coordinate will have increased by $\approx \eps^k$.
 This remains true as long as $z\lesssim \eps$, so in order to reach $z=\eps$ the geodesic has to wind $\approx \eps/\eps^k = \frac1{\eps^{k-1}}$ times.
\item[(b)]
To analyse the total number of revolutions, consider the segment $z\in[N\eps,(N+1)\eps]$ for some $N\in\N$. The diameter at height $z$ is $\approx (N\eps)^k$, so $\gamma$ needs at most (see (c) below) $\approx \eps/(N\eps)^k = N^{-k}\frac1{\eps^{k-1}}$ revolutions to traverse this segment.

Since the sum $\sum_{N=1}^\infty N^{-k}$ converges, the total number of revolutions to reach a fixed $z=z_0>0$ is $\approx \frac1{\eps^{k-1}}$.
\item[(c)]
These considerations do not take into account that the angle $\varphi$ changes with increasing $z$. This change is easily analysed in the warped product case \eqref{eqn:def warped product new}, where the Clairaut relation states that $w^k\cos\varphi$ is constant along $\gamma$. This implies that $\varphi$ is approximately constant for $z\lesssim\eps$, and increases towards $\frac\pi2$ as $z$ increases. So the number of revolutions for $z\in[0,\eps]$ is $\approx\frac1{\eps^{k-1}}$ and for 
 $z\gg\eps$ is even less than estimated in (b).
\item[(d)]
This rough explanation of the exponent $k-1$ appearing in \eqref{eqn:intro winding asymp} is made precise by a simple calculation in the warped product case, see Proposition \ref{Prop:WarpedWinding}. For the general case the errors can be estimated and yield a weaker remainder term, see Theorem \ref{Thm:Winding}.
\end{enumerate}
Summarizing: Winding happens mostly in the neck regime, and the number $\frac1{\eps^{k-1}}$ arises as the quotient of its length $\approx\eps$ and the diameter $\approx\eps^k$.

\medskip

The \textbf{Focussing Theorem}
results from a combination of a scaling consideration and a \lq torque effect\rq:
\begin{enumerate}
\item (Scaling)
Consider a unit speed geodesic starting at the waist with a horizontal velocity $\vh \lesssim\eps^{k-1}$. After time $\eps$ it has moved vertically a distance $\sim\eps$ and horizontally a distance  $\lesssim\eps^k$. Recall that $\eps^k$ is the diameter of the waist, so the angle\footnote{We denote the angular variable by $y$ instead of $\phi$ for consistency with the rest of the paper.} $y$ has changed by $\lesssim1$.
Similarly, if it starts at height $z$ then the same consideration applies with $\eps$ replaced by $w=w(\eps,z)$.

This suggests that a limiting dynamics might exist as $\eps\to0$ if we write the dynamics in terms of $y$ and the rescaled horizontal velocity
$\theta=w^{-(k-1)}\vh$, and with time slowed down by the factor $w$.\footnote{Imagine the geodesic carries a clock along. Then this means that the clock is slowed down by the factor $w(\eps,z)$ when the geodesic is at height $z$.}
\item (Torque)
The next issue is to understand the effect of $y$-dependence of the function $S$ in the metric \eqref{eq:MetricAnsatz}. To focus on essentials, consider an $\eps$-independent metric on $I\times \S^1$ of the form
\begin{equation}
\label{eqn:doubly warped g}
 g = A(z,y) \,dz^2 + W(z)^2\,dy^2
\end{equation}
(no mixed terms).\footnote{This might be called a generalized doubly warped product: A doubly warped product has the form $A(y)\,dz^2 + W(z)^2\,dy^2$. The concept of a doubly warped product seems to have been introduced in \cite{DWarped}. It is a special case of the more general concept of \textbf{conformally separable} metrics defined in \cite{Yan40}, \cite{Won43}, where both warping functions are allowed to depend on all variables. What we call "generalized doubly warped" is a compromise, where one warping function is allowed to depend on all variables.}
If $A$ is independent of $y$ then this is a warped product metric, 
and rotational symmetry implies that, along a geodesic $\gamma(t)=(z(t),y(t))$, the (signed) angular momentum
\begin{equation}
\label{eqn:def L intro}
L := W(z)^2\, \ydot
\end{equation}
is constant (this is the Clairaut relation again). We will show in \eqref{eqn:doubly warped} that $y$-dependence of $A$ creates \textbf{torque}, i.e.\ variation in $L$. More precisely, we have
$$\Ldot =  \tfrac12 \partial_y A\,\zdot^2$$
Together with \eqref{eqn:def L intro} this shows that the angular motion is slowed down when $A$ decreases in $y$.
\item (Combining scaling and torque)
We now consider the metric \eqref{eqn:doubly warped g} with $A = 1- w^\kappa S(y)$ and $W=w^k$ as in \eqref{eq:MetricAnsatz}. We get
\begin{equation}
\label{eqn:Ldot intro}
 \Ldot = - \tfrac12 w^\kappa S'(y) \,\zdot^2
\end{equation}
So the angular motion is slowed down at points where $S'(y)>0$, and sped up where $S'(y)<0$. This already indicates a move towards minima of $S$.

Combining \eqref{eqn:def L intro} and \eqref{eqn:Ldot intro}, and simplifying things by setting $\zdot=1$ and taking $S(y)=y^2$ (locally near a minimum $y=0$), we get that the geodesic dynamics is approximately governed by the equations
\begin{equation}
\label{eqn:simple flow eqn 1}
\ydot = w^{-2k}L\,,\quad \Ldot = - w^{2k-2}y
\end{equation}
where we now take $\kappa=2k-2$. If $w$ was a fixed constant then this is easily solved: the $y$-component of any solution would oscillate around $y=0$, with constant amplitude. 

Using the rescaling suggested above one can now see that the dependence of $w$ on $z$ (and hence time) has a focussing effect, pushing $y$ towards $0$:
Since the horizontal velocity is $\vh=w^k \ydot$, its rescaling is $\theta=w^{-(k-1)}\vh = w\ydot = \frac{L}{w^{2k-1}}$. Rewriting  \eqref{eqn:simple flow eqn 1} in terms of $y$ and $\theta$ we obtain approximately the linear system
\[ \ydot = \frac1w \theta\,,\quad \dot\theta = \frac1w \left(-c\theta - y\right) \]
where $c=2k-1$. Here we have simplified the equations by setting $\zdot=1$ again and by replacing $w_z$ by $1$, which is approximately true for $z>\eps$.
The factors $\frac1w$ disappear when rewriting this in the time variable rescaled by $w$. Indicating the derivative with respect to the rescaled time by $'$ we arrive at the system
\[ y' = \theta\,,\quad \theta' = -c\theta - y\,. \]
All solutions of this linear system approach $y=\theta=0$, and this explains the focussing.
It turns out that the various simplifications that we made in this argument (e.g., taking $b=0$ and setting $\zdot=1$) are justified in that they don't affect this basic mechanism, and that the argument can be generalized to higher dimensions.
\end{enumerate}
The argument shows that the focussing happens mostly in the neck regime, just like the winding. Once the geodesic is close to the minimum (here $\ymin=0$) it follows the geodesic $\gamma_\ymin$ in the bulk regime.

\subsection{Outline of the paper}
In Section \ref{sec:setting} we introduce the blow-up construction which provides a rigorous foundation for the idea of scaling regimes. It
also allows us to formulate succinctly the precise regularity assumptions on the quantities $S,b,h$ occurring in the metric. We also explain the meaning of different choices of the scaling function $w$.

In Section \ref{sec:warped product} we consider geodesics in the warped product case (and without the parameter $\eps$). The Clairaut integral, which expresses conservation of angular momentum, allows a fairly direct treatment of the geodesics in this case. 
For the rest of the paper we find it convenient to view the geodesic flow as a Hamiltonian flow on the cotangent bundle. In Section \ref{Section:Hamiltonian} we recall this formalism for the reader's convenience.   

In Section \ref{Section:Dynamics} we derive the Hamilton equations for general metrics of the form \eqref{eq:MetricAnsatz}. We use these to estimate the variation of the angular momentum $L$ in case the metric is not of warped product type.
This is then used to prove the Winding Theorem in Section \ref{Section:winding}: In the warped product case it follows (with a very good error estimate) from the formulas in Section \ref{sec:warped product}, which are based on constancy of $L$, and the general case then follows in a similar way (but with weaker error estimates) using the previously derived estimates on $\Ldot$.

Sections \ref{Section:GeneralRescalings} and \ref{Section:alpha=2k-1} are devoted to the proof of the Focussing Theorem. In Section \ref{Section:GeneralRescalings} we lay the groundwork by introducing and studying systematically various rescalings of 
the cotangent bundle (corresponding to the rescaled horizontal velocity $\theta$ in the outline of main ideas above), and calculate the Hamilton vector field transferred to these. We find that there is a unique rescaling for which this vector field  (when multiplied by
an additional factor, corresponding to the time rescaling above) is both smooth up to the boundary at $\eps=0$ (as subset of the blown-up space) and sufficiently non-degenerate, meaning that its critical points are hyperbolic (under the assumption that $S_0$ is Morse). The study of these critical points then allows us to prove the Focussing Theorem in Section \ref{Section:alpha=2k-1}. A fundamental tool here is the  (un-)stable manifold theorem.

Note that our proof of the Focussing Theorem yields much finer results than stated. For example, some focussing already occurs for time $t$ on the order $C\eps$ for large $C$, rather than for fixed positive $t$ as stated in the Focussing Theorem (or correspondingly $z_1$ in its precise version Theorem \ref{Thm:focussing}).

\subsection{Further remarks, open problems}
\label{subsec:further remarks}
\hfill\\[2mm]
\noindent\textbf{Conical degeneration:}
If the limit space has a conical rather than a cuspidal singularity (this would correspond to $k=1$) then there is no winding, i.e.\ the horizontal length of geodesics stays bounded as $\eps\to0$. See \cite[Theorem C]{GriLye} and compare \cite[Definition 1.4, Lemma 1.5]{MeWu:Geo}. Note that this is consistent with formula \eqref{eqn:intro winding asymp}. Also, we expect that there is no focussing effect in this case. The essential difference between the conical and the cuspidal case is that in the conical case the radius of the waist and the neck length are comparable, while in the cuspidal case the neck length is much larger.

\medskip

\noindent\textbf{Some open problems:}
\begin{itemize}
\item (General profile functions)
 For the Winding Theorem, one can probably replace $z^k$ by more general profile functions $W$ as in our previous paper \cite{GriLye}. We then suspect that the term $\eps^{k-1}$ in the denominator of the winding asymptotics gets replaced by $W'(\eps)$. It would be interesting, but probably a lot harder, to extend the focussing results to more general profile functions.
\item (The role of $\kappa$)
As mentioned before (see Remark \ref{Rem:motivation setup}), the number $\kappa$ in \eqref{eq:MetricAnsatz} equals $2k-2$ for the example \eqref{eqn:example} if the Euclidean metric is used on $\R^3$, but for other ambient metrics it may not possible to choose it larger than $k$ (even by choosing different coordinates).
The Winding Theorem is true if $\kappa=k$, but for the Focussing Theorem we need to assume $\kappa\geq 2k-2$ for all the interesting statements. This is also the case when $\epsilon=0$: the results in \cite{GrGr15} hold only if $\kappa\geq 2k-2$. We expect that new phenomena occur when $k\leq \kappa<2k-2$.
\item (Properties of $S_0$ in case $\kappa=2k-2$)
The Focussing Theorem and several other statements of Section \ref{Section:alpha=2k-1} require $S_0$ to be a Morse function. It would be interesting to see if one can relax this assumption, for example to Morse-Bott or to finite order vanishing of $S_0'$ in case $n=2$. The dynamical systems methods of Section \ref{Section:alpha=2k-1} would need to be refined to handle this case.

In \cite{GrGr15} also the case of constant $S_0$ is analysed for the cuspidal limit space, and it is shown that there is a smooth exponential map based at the singularity then. We expect that, correspondingly, in our setup there is no focussing if $S_0$ is constant. Compare Remark \ref{Rem:explan eps to 0}, and Theorem \ref{Thm:SconstAsymptotics} for a partial result.
\item
(The angle of transmission)
Rather than studying geodesics that start at the waist
one can ask which geodesics starting at a point $p=(z_0,y)$ with $z_0>0$ pass the waist. For fixed $p$ there will be a cone of directions for which this is the case.  In the warped product case this cone is easily determined from the Clairaut relation, see Theorem \ref{Thm:warped product}. In the case of a Morse function $S_0$ and $n=2$ we expect that the cone is convex and that our methods can be used to determine the asymptotic behaviour of its opening angle as $\eps\to0$. For example, the Focussing Theorem suggests that for $p$ lying on a geodesic $\gamma_\ymin$ this angle should be much larger than for most other points at the same level $z_0$.
\item (Non-isolated singularities)
Both the present work and our previous work \cite{GriLye} deal with isolated (or 
zero-dimensional) singularities. It would be interesting to know what one can say about  singularities of positive dimension, like a wedge singularity. This would be relevant in complex geometry, where one often encounters singularities along a divisor, which has real codimension two.
\end{itemize}

\subsection{Related works}
We do not know of other works dealing with geodesics under metric degeneration. There are several works dealing with the behaviour of other geometric quantities under a controlled metric degeneration (often to conical metrics). For instance, there is work on the Laplace spectrum \cite{Gromov}, \cite{hyperbolic}, \cite{Yoshikawa}, \cite{MazMel}, \cite{GriJer}, \cite{GriQuasimodes}, the analytic torsion \cite{DaiMel}, \cite{ARS}, and the eta-invariant \cite{CheegerEta}, \cite{SeeleyConic}, \cite{DaiEta}) under metric degeneration. This list is of course very far from complete. 

There have been other studies of geodesics on a singular space (as opposed to a smooth family degenerating to something singular).
 The conical case ($k=1$) was treated by Melrose and Wunsch in \cite[Definition 1.4, Lemma 1.5]{MeWu:Geo} in the context of wave propagation, after early work by Stone \cite{Sto:EMISP}.
The first author discusses conical metrics in detail in \cite{Gri:NDOCS}. Previously, Bernig and Lytchak \cite{BerLyt:TSGHLSS} obtained first order information for geodesics on general real algebraic sets $X\subset\R^n$, by showing that any geodesic reaching a singular point $p$ of $X$ in finite time must have a limit direction at $p$.

Still in the context of $\eps=0$, geodesics hitting the singularity were analysed by Grandjean and the first author \cite{GrGr15}. The starting point of the present article was the question "Can one obtain the results of \cite{GrGr15} by letting family of metrics degenerate to a cuspidal metric? Do any new phenomena appear?" The answer, in so many words, is "yes and yes".

\subsection{Notation}
We have quite a lot of notation. Some notation which gets used repeatedly is listed in the table below. Symbols which are only used in the section they get introduced will generally not be listed. This is for instance the case for most notation introduced and used only in Section \ref{Section:alpha=2k-1}.
\begin{table}[H]
\begin{center}
  \begin{tabular}{| l | c | c |}
    \hline
    \textbf{Symbol} &  \textbf{Meaning} &  \textbf{Introduced}\\ \hline
    $M=I\times Y$ & manifold & p. \pageref{Mdef} \\ \hline
    $(z,y)\in M$ & point in $M$ &   p. \pageref{zyDef}\\ \hline
    $g,S,b,h$& metric on $M$ and components thereof & \eqref{eq:MetricAnsatz}, \eqref{eq:MetricAnsatz-general1} \\ \hline 
    $w$& singular scaling function & pp. \pageref{wdef1}, \pageref{wdef2} \\ \hline
    $k\geq2$& order of the cusp & pp. \pageref{CuspDef}, \pageref{wdef1} \\ \hline
     $\kappa\in [k,2k-2]$& order to which $g$ differs from warped product, roughly & \eqref{eq:MetricAnsatz}, \eqref{eq:MetricAnsatz-general1} \\ \hline     
    $\varphi$   & angle of impact  & pp. \pageref{phiDef1}, \pageref{phiDef2} \\ \hline        
    $(z,\xi)\in T ^*I$ & cotangent variables got vertical direction & p. \pageref{T*IDef} \\ \hline
	$(y,\eta)\in T^*Y$ & cotangent variables for horizontal direction & p. \pageref{T*YDef} \\ \hline
	$X_0$ & blowup of $[0,1)_\eps \times I$ & p. \pageref{X0Def} \\ \hline    
	$X=X_0\times Y$& blowup of $[0,1)_\eps \times M$ & p. \pageref{XDef}, Fig. \ref{Fig:Blowup} \\ \hline
	$X^*$& pullback of $T^*M$ to $X$ & p. \pageref{X*Def} \\ \hline    
	$Z=\frac{z}{\eps}, E=\frac{\eps}{z}>0$& projective coordinates on (parts of) $X$ (or $X_0$ or $X^*$) & pp. \pageref{ZDef}, \pageref{EDef}, Fig. \ref{Fig:Blowup}\\ \hline 
	$f(Z)=\frac{w}{\eps}$, $F(E)=\frac{w}{z}$ & profile functions & \eqref{eqn:def f}, \eqref{eqn:def F}, \eqref{eqn:props f}\\ \hline   
	$\ff$, $M_{\pm}$& boundary hypersurfaces of the blowup $X$ & p. \pageref{ffMDef}, Fig. \ref{Fig:Blowup}\\ \hline       
	$\calH\colon X^*\to \R$& (usual) Hamiltonian  & pp. \pageref{HamDef1}, \pageref{HamDef2} \\ \hline 
	$L=\vert \eta\vert$  & angular momentum & p. \pageref{LDef2} \\ \hline 
	$V$& geodesic vector field & pp. \pageref{VDef1}, \pageref{VDef2}\\ \hline    
	$\theta=\frac{\eta}{w^{2k-1}}$& rescaled momentum & p. \pageref{thetaDef}\\ \hline    
	$\tau$, $'=\frac{d}{d\tau}=w\frac{d}{dt}$& rescaled time and derivative & p. \pageref{tauDef} \\ \hline    
	$\mathcal{V}=wV$ & rescaled geodesic vector field & p. \pageref{calVDef} \\ \hline      
  \end{tabular}
\end{center}
\label{tab:Symbols}
\end{table}

\section{Fundamentals on the scaling function, blow-up, and the metric}
\label{sec:setting}

\subsection{The singular scaling function $w$}\label{subsec:w}
\label{wdef2}
We first discuss the geometric meaning of the singular scaling function $w$.
Recall that we assume \eqref{eqn: w monotone}, \eqref{eqn:w norm eps=0}, \eqref{eqn:w homogeneous}.
A family of examples is\footnote{The main results of this paper hold also for $w=w_p$ with $p\geq2$ real. The reason being that then $w_p$ is at least a $C^{2,\alpha}$- function, so the geodesic vector field is $C^{1,\alpha}$, hence the geodesic flow is well-defined with $C^1$ dependence on initial data. To simplify the presentation we assume $w$ to be smooth.}
\begin{equation}
\label{eqn:choices w}
w_p=\sqrt[p]{\eps^{p}+\vert z\vert^{p}},\ p\in [2,\infty).
\end{equation}
Homogeneity is a natural assumption since $\eps$, $z$ and $w$ should all have the same physical unit of length. Then \eqref{eqn: w monotone} (second part), \eqref{eqn:w norm eps=0} are normalizations: They say that  $g_0=dz^2 + z^{2k} h$, the standard cuspidal metric of order $k$, and that $(g_{\eps})_{\vert \{0\}\times Y} = \eps^{2k} h$, which means that the neck at $z=0$ has a \lq diameter\rq\ of order $\eps^k$.

Different choices of $w$ correspond to different geometric 'shapes' of the spaces $(M,g_\eps)$ near their waist $z=0$ in the following sense.
Consider a fixed $\eps>0$. Write $Z=\frac{z}{\eps}$.  Then for $p=2$ in \eqref{eqn:choices w} 
we get
$$ w_2^{2k}/\eps^{2k} = (Z^2+1)^k = 1 + kZ^2 + \er(Z^4) $$
while for $p=2k$ we get
$$w_{2k}^{2k}/\eps^{2k} = 1+Z^{2k}.$$
Thus, for $w_2$ the approximating spaces have a 'non-degenerate' waist at $z=0$, while for $w_{2k}$ these spaces are flat in $Z$ to order $2k$ at $z=0$. 

The function $f(Z)\coloneqq w(Z,1)=w/\eps$ will play a role in the fine structure of our statements. It is hidden in the constant $C_\varphi$ in the Winding Theorem and controls the strength of the focussing in the Focussing Theorem.

\subsection{The blown-up spaces $X$ and $X^*$}
\label{subsec:b-up spaces}
To make the idea of scaling, explained in the introduction, precise we introduce certain blown-up spaces (although we only really use this for the Focussing Theorem, not for the Winding Theorem). This also allows to formulate succinctly the precise smoothness assumptions on the objects $S,b,h$ occurring in the metric. (Note that the function $w$ is not smooth at $(\eps,z)=(0,0)$, and in Appendix \ref{Section:Elliptic} we show for the example in \eqref{eqn:example} that $S,b,h$ depend on $w$, so they are not smooth there either.)
Recall that $S,b,h$ depend on $(\eps,z,y)$, so it is natural and useful to express them as objects on the space  
\begin{equation}
\label{eqn:total space}
[0,1)_\eps \times M = [0,1)_\eps \times I_z \times Y \,,
\end{equation}
the \textbf{total space} for our problem.

The non-smoothness only occurs in the variables $(\eps,z)$.  
Observe that any 1-homogeneous function of $(\eps,z)$ which is smooth outside $(0,0)$ becomes smooth down to $r=0$ when written in polar coordinates $r=\sqrt{\eps^2+z^2}$, $\phi=\arctan \frac z\eps$, because $w(\eps,z) = w(r\cos\phi,r\sin\phi) = r\,w(\cos\phi,\sin\phi)$. 
For our investigations, in particular for a proper understanding of the proof of the Focussing Theorem, it is useful to restate the idea of 'expressing things in polar coordinates' geometrically, in terms of blow-up.
We introduce the basic notions that we need and refer the reader to \cite{Mel:DAMWC}, \cite{Gri:BBC}, \cite{Mel:RBIASS}  for details.
\subsubsection{Definition of the blown-up spaces}
By definition, the \textbf{blow-up} of $\Rplus\times \R$ (where $\Rplus=[0,\infty)$) at the origin is the domain of polar coordinates, 
\begin{equation}
\label{eqn:blow-up1}
[\Rplus\times\R;(0,0)] \coloneqq 
\Rplus\times[-\tfrac\pi2,\tfrac\pi2] 
\end{equation}
This is related to the unblown-up space $\Rplus\times\R$ by the \textbf{blow-down map}, which by definition is simply the polar coordinates map
$\beta_0: [\Rplus\times\R;(0,0)]\to \Rplus\times\R\,,\quad (r,\phi) \mapsto (r\cos\phi, r\sin\phi)$. The \textbf{front face} is the subset $\ff_0\coloneqq\beta_0^{-1}((0,0)) = \{r=0\}$; it is parametrized by $\phi$. Note that $\beta_0$ collapses $\ff_0$ to the point $(0,0)$ but is a diffeomorphism between the complements of these sets.\footnote{This corresponds to the well-known fact that polar coordinates provide a coordinate system only in $r>0$, not at the origin.}
A suggestive graphical representation is Figure \ref{Fig:Blowup0}, where some of the rays $\phi=\const$ are drawn. The key thing to note in Figure \ref{Fig:Blowup0} is how the direction of approaching $(0,0)$ in the right hand picture corresponds to points of the front face in the left hand picture. 
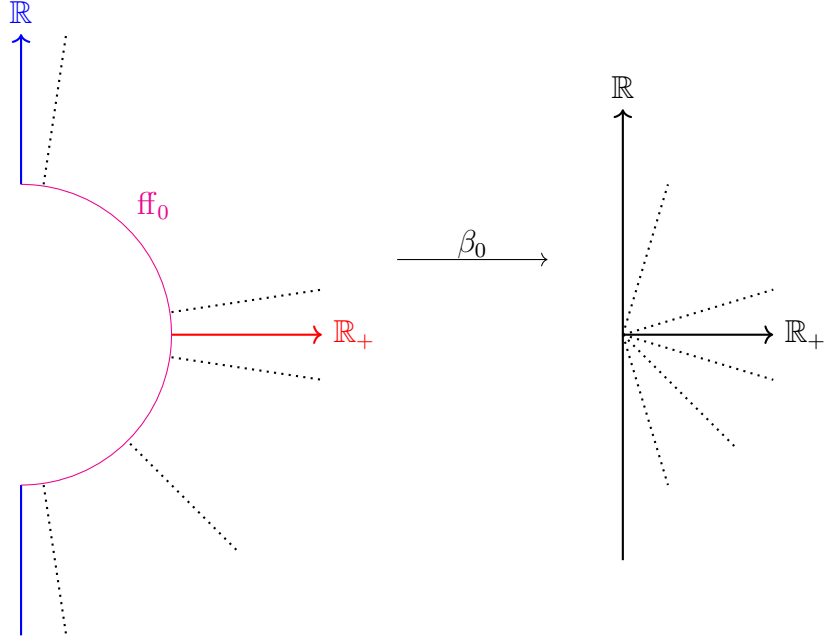
\begin{figure}
\begin{tikzpicture}
\draw[red,thick,->] (2,0) -- (4,0) node[anchor=west] {$\R_+$};
\draw[blue,thick,->] (0,2) -- (0,4) node[anchor=south] {$\R$};
\draw[blue,thick,-] (0,-2)--(0,-4);
\draw[magenta] (0,-2) arc (-90:90:2cm);
\draw[magenta] (1.4,1.4) node[anchor=south west] {$\mathrm{ff}_0$};
\draw[dotted,thick] (2,0.3)--(4,0.6);
\draw[dotted,thick] (0.3,2)--(0.6,4);
\draw[dotted,thick] (2,-0.3)--(4,-0.6);
\draw[dotted,thick] (0.3,-2)--(0.6,-4);
\draw[dotted,thick] (1.45,-1.45)--(2.9,-2.9);
\draw[dotted,thick] (8,0)--(8.6,2);
\draw[dotted,thick] (8,0)--(10,0.6);
\draw[dotted,thick] (8,0)--(8.6,-2);
\draw[dotted,thick] (8,0)--(10,-0.6);
\draw[dotted,thick] (8,0)--(9.5,-1.5);
\draw[->] (5,1)--(7,1);
\draw node at (6,1.2) {$\beta_0$};
\draw[thick,->] (8,0) -- (10,0) node[anchor=west] {$\R_+$};
\draw[thick,->] (8,-3)--(8,3) node[anchor=south] {$\R$};

\end{tikzpicture}
\caption{The blow-up $[\R_+\times \R;(0,0)]$ is shown on the left. The blown-down space $\R_+\times \R$  is shown on the right. The dotted lines represent $\phi=\const.$ lines.} 
\label{Fig:Blowup0}
\end{figure}

We define \label{X0Def}
$$X_0=[\,[0,1)_\eps\times I;(0,0)]$$ 
to be the part of \eqref{eqn:blow-up1} corresponding to $\eps\in[0,1)$, $z\in I$, i.e.\ $X_0\coloneqq\beta_0^{-1}([0,\eps)\times I)$.

Similarly, we blow up the total space \eqref{eqn:total space} at $\{\eps=z=0\}=\{(0,0)\}\times Y$, by simply taking the product with $Y$: we define \label{XDef}
\begin{equation}
\label{eqn:def X}
X  = [\,[0,1)_\eps \times M;\{\eps=z=0\}] = [\,[0,1)_\eps\times I\times Y;\{(0,0)\}\times Y] \coloneqq X_0\times Y
\end{equation}
and call $\beta = \beta_0 \times \id_Y: X \to [0,1)_\eps \times M$ the blow-down map and $\ff\coloneqq \beta^{-1}(\{\eps=z=0\})$ the front face.

We also need a version of this for the cotangent bundle, on which the geodesic vector field lives. Here the total space is,  
\begin{equation}
\label{eqn:total cot space}
[0,1)_\eps \times T^*M = [0,1)_\eps \times I_z\times\R_\xi \times T^*Y 
\end{equation}
where we identified\label{T*IDef} $T^*I=I\times\R$ with variables denoted $z,\xi$. Note that there is no codirection for $\eps$ since this is a parameter.
Correspondingly the blown-up total cotangent space is defined as\label{X*Def}
\begin{equation}
\label{eqn:def X*}
X^* = [\,[0,1)_\eps\times I_z\times \R_\xi\times T^*Y;\{(0,0)\}\times \R_\xi\times T^*Y] \coloneqq X_0\times \R_\xi\times T^*Y.
\end{equation}
This comes with its own blow-down map 
\begin{align*}
&\tilde{\beta}\colon X^*\to [0,1)_{\eps} \times T^*M,\quad \tilde{\beta}=\beta_0\times \mathrm{id}_{\R\times T^*Y}
\end{align*}
and is a vector bundle of rank $n$ with projection
\begin{equation}
\label{eqn:X* bundle}
\begin{aligned}
&\tilde{\pi}\colon X^*\to X\\
&\tilde{\pi}(x_0,\xi,y,\eta)=(x_0,y)
\end{aligned}
\end{equation}
where $x_0\in X_0$ and\label{T*YDef} $(y,\eta)\in T^*Y$.
This means $X^*$ sits in the following commutative diagram:
\[
\begin{tikzcd}
X^* \arrow{r}{\tilde{\beta}} \arrow[swap]{d}{\tilde{\pi}} & 
{[0,1)}_{\eps} \times T^*M \arrow{d}{\mathrm{id\times\pi}} \\ 
X\arrow{r}{\beta} & 
{[0,1)}_{\eps}\times M.
\end{tikzcd}
\]
We will get back to a rescaled version of the blown up cotangent bundle $X^*$ in Sections \ref{Section:GeneralRescalings} and
\ref{Section:alpha=2k-1}. 

\smallskip

The spaces $X_0$, $X$ and $X^*$ are manifolds with corners having three boundary hypersurfaces.\footnote{We ignore the boundary hypersurfaces at $z\in \partial I$ since they play no role in this discussion.} For $X$ these are
the front face\label{ffMDef} $\ff$ and the upper and lower limit faces $M_\pm$, defined by
$M_\pm=\{\phi=\pm\frac\pi2\}$ in polar coordinates, or 
$$ M_\pm \coloneqq \beta^*(\{\pm z \geq 0, \eps=0\}) $$
where $\beta^*(S) = \overline{\beta^{-1}(S\setminus \{z=\eps=0\})}$ denotes the \textbf{lift} of a set $S$  not contained in the center of blow-up.
The faces $M_\pm$ are canonically identified, via $\beta$, with the upper and lower half of $M$, i.e.,
\begin{equation}
\label{eqn:Mpm diffeo}
M_\pm \cong I_\pm \times Y
\end{equation}
naturally, where $I_\pm= \{z\in I: \pm z \geq 0\}$.
There are corresponding faces of $X^*$,  $\tilde{\pi}^{-1}(M_\pm)$.
There are also two corners of codimension two: $M_\pm\cap\ff$ in $X$ and $\tilde{\pi}^{-1}(M_\pm\cap\ff)$ in $X^*$.

See Figure \ref{Fig:Blowup} for an illustration.
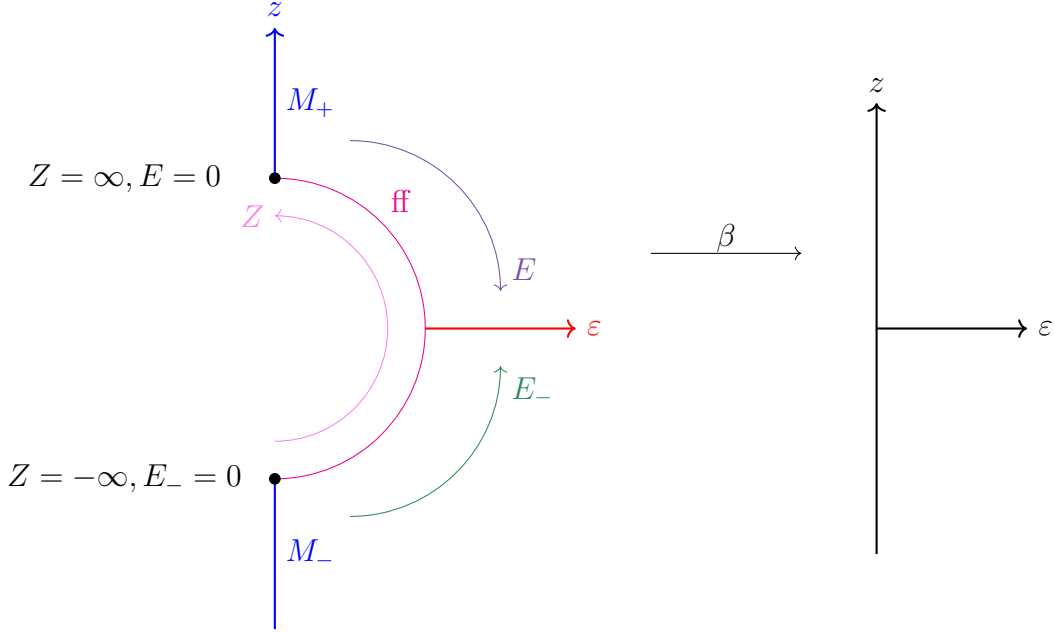
\begin{figure}
\begin{tikzpicture}
\draw[red,thick,->] (2,0) -- (4,0) node[anchor=west] {$\eps$};
\draw[blue,thick,->] (0,2) -- (0,4) node[anchor=south] {$z$};
\draw[blue,thick,-] (0,-2)--(0,-4);
\draw[magenta] (0,-2) arc (-90:90:2cm);
\draw[lightfuchsiapink,->] (0,-1.5) arc (-90:90:1.5cm) node[anchor=east] {$Z$};
\draw[blue] (0,3) node[anchor=west] {$M_+$};
\draw[blue] (0,-3) node[anchor=west] {$M_-$};
\draw[magenta] (1.4,1.4) node[anchor=south west] {$\mathrm{ff}$};
\filldraw[black] (0,-2) circle (2pt);
\draw node at (-2,-2) {${Z=-\infty, E_-=0}$};
\filldraw[black] (0,2) circle (2pt);
\draw node at (-2,2) {${Z=\infty, E=0}$};
\draw[seagreen,->] (1,-2.5) arc (-90:0:2cm) node[anchor=north west] {$E_-$};
\draw[royalpurple,->](1,2.5) arc (90:0:2cm) node[anchor=south west] {$E$};
\draw[->] (5,1)--(7,1);
\draw node at (6,1.2) {$\beta$};
\draw[thick,->] (8,0) -- (10,0) node[anchor=west] {$\eps$};
\draw[thick,->] (8,-3)--(8,3) node[anchor=south] {$z$};

\end{tikzpicture}
\caption{The blown-up space $X$, with the $Y$-directions suppressed, is shown on the left. The blown-down space without $Y$, $[0,1)_{\eps}\times I_z$, is shown on the right. The curved arrows indicate the three projective coordinates $Z=\frac{z}{\eps}$, $E=\frac{\eps}{z}$, $E_-=-\frac{\eps}{z}$.} 
\label{Fig:Blowup}
\end{figure}
\smallskip

The function $w$, as well as many other functions appearing in the investigation, are  smooth on the blow-up, by which we mean the following. A function $f\colon [0,1)_{\eps}\times M\to \R$ of is \textbf{smooth on the blow-up} if $\beta^* f\colon X\to \R$ is smooth. In other words, $f(\eps,z,y)$ is smooth when written in terms of $r,\phi,y$.  We use the same letter $w$ for the function on the unblown-up spaces and for the blown-up spaces $X_0$, $X$ and $X^*$. It is a boundary defining function for the respective front faces, i.e.\ vanishes on them to precisely first order and is positive otherwise.

\subsubsection{Projective coordinates}
For calculations it is easier to use \textbf{projective coordinates} on $X_0$ instead of polar coordinates. They have the advantage of simple formulas (no transcendental functions appear), but the disadvantage that one needs three such coordinate systems to cover the whole space. 

The first projective coordinate system is \label{ZDef}
$$ Z= \frac z\eps \in\R\,,\quad \eps \geq 0 $$
and is defined on the subset $\{|\phi|<\frac\pi2\}$, i.e.\ on the complement of the upper and lower boundary hypersurface $I_\pm$, with $\eps=0$ corresponding to the front face and $Z\to\pm\infty$ when approaching $I_\pm$.
This is to be understood as follows: Points $(\eps,z)$ in $\{\eps>0\}$ can be specified by $Z=\frac z\eps$ (slope) and $\eps$. Now 
the set $\{\eps>0\}$ corresponds to $X_0\setminus\{I_+\cup I_- \cup \ff_0\}$, so the pull-backs of $Z,\eps$ to this set are smooth coordinates on it because $\beta_0$ is a diffeomorphism outside $\ff_0$. We denote the pull-backs by $Z,\eps$ again. 
They extend to a smooth coordinate system on $X_0\setminus\{I_+\cup I_-\}$ since
$Z=\tan\phi$, $\eps=r\cos\phi$ are smooth down to $r=0$.

The second projective coordinate system is \label{EDef}
$$ z\geq0\,,\quad E = \frac\eps z \geq0\,, $$
defined \lq near the upper corner\rq, more precisely in $\{\phi>0\}$. Here $z=0$ corresponds to the front face and $E\to\infty$ when approaching the $\eps$-axis $\phi=0$. There is a third system, $z\leq0$, $E_-=-\frac \eps z\geq0$, covering the bottom part $\{\phi<0\}$.

We also use these coordinates on (subsets of)  $X$ and $X^*$. In terms of $Z$, we may identify the front face $\ff_0$ of $X_0$ with $\Rbar=\R\cup\{\pm\infty\}$, and then
\begin{equation}
\label{eqn:ff diffeo}
\ff \cong \Rbar_Z \times Y
\end{equation}
naturally, where the index $Z$ serves as a reminder that the variable on $\Rbar$ is $Z$. 

In terms of the $(Z,\eps)$  system we can express $w$ as 
\begin{equation}
\label{eqn:def f}
w = \eps f(Z),\quad f(Z) \coloneqq w(1,Z)
\end{equation} 
and in terms of the $(z,E)$ system as
\begin{equation}
\label{eqn:def F}
w = z F(E),\quad F(E) = w(E,1).
\end{equation}
Then $f,F$ are smooth on $[0,\infty)$ and are equivalent in the sense that $f(Z)=ZF(\frac1Z)$. The assumptions on $w$ translate into
\begin{equation}
\label{eqn:props f}
f(0)=1,\ \sgn\left(\tfrac{\partial f}{\partial Z}\right) = \sgn\, Z,\ f(Z) \sim |Z| \text{ as } |Z| \to \infty,
\end{equation}
plus the analogous properties for $F$.

\begin{Rem}[Regimes and boundary hypersurfaces]
\label{Rem:bhs regimes} 
The blow-up construction of $X$ allows to make the idea of regimes, mentioned in Section \ref{Section:MainIdeas}, precise.
The boundary hypersurfaces $\ff$ and $M_+$ correspond to the neck regime and the bulk regime, respectively (say in $z\geq0$ for simplicity): The neck regime corresponds to $z\leq C\eps$ for small $\eps$. Now $z\leq C\eps$ means $Z\leq C$, so this corresponds to small neighbourhoods of compact subsets of $\ff\setminus M_+$.
The bulk regime corresponds to $\frac z\eps\to\infty$, so $Z\to\infty$ or equivalently $E\to 0$, which characterizes $M_+$. 

The edge $\ff\cap M_+$ corresponds to the transition from neck to bulk regime. To have a geometric object encoding this transition is fundamental for the analysis: the central  issue for the Focussing Theorem are the critical points of the rescaled Hamilton vector field, and they lie on this edge.
\end{Rem}

\subsection{Smoothness of the metric} \label{ssec:smoothness}
Having introduced suitable blow-ups, we can finally formulate our smoothness assumptions on our family of metrics  
\begin{equation}
\label{eq:MetricAnsatz-general1}
g=g_{\eps}=(1-w^{\kappa}S)dz^2+2w^{2k} dz\cdot b +w^{2k} h
\end{equation}
 in \eqref{eq:MetricAnsatz}. We assume that $S,b,h$ \lq are smooth on the blow-up\rq, i.e.,
\begin{align*} 
 &\beta^*S \colon X\to \R\\
 &\beta^*b \colon X\to T^*Y\\ 
 &\beta^* h \colon X\to T^*Y\otimes T^*Y
 \end{align*}
are all smooth. Here $\beta^* S=S\circ \beta$, $\beta^*b=b\circ \beta$ and so on. We shall in general suppress the blow-down from the notation, and write $S$ for both $S$ and $\beta^*S$ and similarly for $b,h$. Note that since $\beta$ is a diffeomorphism away from the front face $\ff$, this implies that $S,b,h$ are smooth at all $(\eps,z,y)$ with $\eps>0$. Furthermore, we assume that 
\[\partial_zh  \text{ is bounded }\]
(but not necessarily smooth) on the blow-up $X$. We expect that this assumption is not needed.

Finally, we make an assumption on the restriction of $S$ to the front face at the corner $\ff\cap M_\pm$: In terms of projective coordinates $(E,z)$
\begin{equation}
\label{eqn:S assumption}
\partial_E S = 0 \text{ at }\ \ff \cap M_+
\end{equation}
and similarly at $\ff\cap M_-$.
This assumption simplifies some of the computations in Section \ref{Section:alpha=2k-1}. It
is satisfied in examples of type \eqref{eqn:example}, see Lemma \ref{Lem:ExampleS}. We expect that our results also hold without it.

 There is a tacit assumption that $\eps$ and $I$ are taken small enough that $g_{\eps}$ is a metric, i.e. that $1-w^\kappa S>0$ and $\det(g_{\eps})>0$. All our statements concern the behaviour of geodesics near $(\eps,z)=(0,0)$, so this will not be a real restriction.

 For some of our results we will further impose a  restriction\footnote{We should remark that the other bound, $\kappa\leq 2k$, is not really a restriction; if $\kappa>2k$, one can simply replace $S$ by $w^{\kappa-2k} S$. The upper bound $\kappa\leq 2k$ makes some statements a bit cleaner and should be seen as a cosmetic requirement.} on $\kappa$;  $\kappa\geq 2k-2$.
 We remark that we are here using a different normalisation of $S$; our $S$ is $k(k-1)S$ in \cite{GrGr15}. We will show that a large class of examples fit into this framework in Appendix \ref{Section:GeneratingExamples}.

\section{Geodesics on a  warped product} \label{sec:warped product}

In this section we consider a  manifold $M=[-1,1]\times Y$ with a fixed warped product metric
\[ g = dz^2 + W(z)^2 h .\]
Here $h$ is any Riemannian metric on $Y$, and $W>0$. 
This includes surfaces of revolution, for instance.

Our assumptions are that $W$ is $C^2$ and that
\[ \sgn(W'(z))=\sgn(z).\]
 In particular, $W(z)$ has a unique minimum at $z=0$, so that $M$ has a 'waist' at $z=0$. We let $W_{\min}=W(0)$.

Write geodesics on $M$ as $\gamma(t)=(z(t),y(t))$.
We will prove a sharp version of the Winding Theorem  in this simpler setting. The proof relies on the fact that one has (at least) two preserved quantities in this case. The first is the preservation of energy (i.e. the unit speed condition),
\begin{equation}
\label{eqn:pres energy}
\zdot^2 + W(z)^2\, |\ydot|^2 = 1 
\end{equation} 
(where $|\ydot| \coloneqq |\ydot|_h$).
The second, less obvious, is the preservation of\footnote{In Section \ref{Section:Hamiltonian} we offer a proof of $L$ being constant.}  the angular momentum:
\[ L\coloneqq W(z)^2\, |\ydot|\qquad \text{ is constant along $\gamma$}.\]
Note that if $\zdot(t) = \sin\left [\varphi(t)\right]$ with $\varphi(t)\in \left[-\tfrac\pi2,\frac\pi2\right]$, i.e. $\varphi(t)$ is the angle at which $\gamma$ intersects the latitude $z=$ const at time $t$ (with $\varphi$ positive for geodesics running upwards), then $W(z)\,|\ydot|=\cos\varphi$ from \eqref{eqn:pres energy}, so this is simply the Clairaut integral:
\begin{equation}
\label{eqn:clairaut}
L = W(z) \cos\varphi  \qquad \text{ along $\gamma$.}
\end{equation} 

The following theorem shows that a geodesic starting upwards at negative $z$ will pass the waist if and only if it has small angular momentum, $L<W_{\min}$, and gives a formula for its angular length.

\begin{Thm} \label{Thm:warped product}
 Let $(M,g)$ be as above. Consider a geodesic $\gamma(t)=(z(t),y(t))$ starting at $z=-1$ upwards, i.e.\ with $\zdot>0$.
\begin{enumerate}
 \item If $L< W_{\min}$ then $\gamma$ reaches $z=1$ after a finite time, $z$ is strictly increasing and the total angular length of $\gamma$ is
\[ \angl (\gamma) =\int_{-1}^1 \frac{L}{W(z)\sqrt{W(z)^2-L^2}}\, dz \]
 \item If $L=W_{\min}$ then $\gamma(t)$ exists for all $t>0$, $z(t)$ is strictly increasing and $z(t)\to0$ as $t\to\infty$. Furthermore, the angular length is infinite, $\angl (\gamma) = \infty$.
 \item If $L>W_{\min}$ then $\gamma(t)$ turns back and reaches $z=-1$ at some finite time $T$ again, and the angular length of $\gamma$ on $[T_-,T]$ satisfies essentially the same formula as in case (1),
\[\angl (\gamma)=2\int_{-1}^{z_0} \frac{L}{W(z)\sqrt{W(z)^2-L^2}}\, dz\]
where $z_0\in[-1,0)$ is defined by $W(z_0)=L$.
\end{enumerate}
\end{Thm}
Note that in (3) $z_0$ exists and is unique since $W(0)<L\leq W(-1)$ by \eqref{eqn:clairaut}, and $W$ is strictly decreasing on $[-1,0)$.
\begin{proof}
This all follows by combining the energy and angular momentum conservation. We include the details for completeness.

Inserting the angular momentum $L$ into the energy gives
\[1=\dot{z}^2+\frac{L^2}{W(z)^2}\]
which is an ODE for $z$. 
\begin{enumerate}
\item When $L<W_{\min}$, the energy tells us $\dot{z}^2=1-\frac{L^2}{W^2}\geq 1-\frac{L^2}{W_{\min}^2}>0$. So $\dot{z}>0$ for all time and is uniformly bounded away from $0$. Hence $z(t)$ has to reach $z=1$ in finite time.

For the winding formula, we start by writing $\vert \ydot\vert=\frac{L}{W^2}$.
Since $\dot{z}=\sqrt{1-\frac{L^2}{W^2}}>0$ for all time, we can parametrize $y$ with respect to $z$. 
Then
\[ \angl (\gamma) =\int_J \vert \ydot\vert \, dt =\int_{-1}^1 \frac{ \vert\ydot\vert}{\dot{z}}\, dz=\int_{-1}^1 \frac{L}{W(z)\sqrt{W(z)^2-L^2}}\, dz.\]
More geometrically, if $\ell_Y$ denotes angular length parametrized by $z$, then $\frac{d\ell_Y}{dz}=W^{-1}\cot\varphi$ by trigonometry\footnote{The factor $W^{-1}$ appears since $\ell_Y$ is measured with respect to $h$, while the angle is defined with respect to the metric $dz^2 + W^2 h$.}, and $\cos\varphi=\frac LW$ gives
$\cot\varphi=\frac{L/W}{\sqrt{1-L^2/W^2}} = \frac L{\sqrt{W^2-L^2}}$. Then integrate.
\item We offer two related proofs of $z(t)$ never reaching $z=0$. 

\textbf{Analysis proof:} The ODE for $z$ is $\dot{z}=\sqrt{1-\frac{L^2}{W(z)^2}}$. Since $W(z)$ is $C^2$, the function $z\mapsto\sqrt{1-\frac{L^2}{W(z)^2}}$ is Lipschitz. So by the Picard-Lindelöf theorem, the ODE has a unique solution. Now $z(t)=0$ for all $t$ is a solution when $W(0)=L$, hence any solution reaching $z=0$ in finite time would have to satisfy $z(t)=0$ constantly. This contradicts $z(T_-)=-1$.

\textbf{Geometric proof:} Observe that the waist $\{z=0\}$ is totally geodesic since $W(z)$ has a minimum there. Any geodesic reaching the waist at finite time $T$ satisfies $z(T)=0$, $\dot{z}(T)=\sqrt{1-\frac{L^2}{W(0)^2}}=0$. So this geodesics has to stay in the waist, contradicting $z(T_-)=-1$.

So $W(z(t))>W_{\min}=L$ for all time and thus $\dot{z}(t)=\sqrt{1-\frac{L^2}{W^2}}>0$ for all time. The limit $\lim_{t\to \infty} z(t)$ therefore exists and has to be a point where $\dot{z}=0$, i.e. $z=0$. 
Finally, $|\ydot|$ is bounded below away from zero since $L=W(z)^2|\ydot|$ is constant, so the angular length is infinite.

\item When $L>W_{\min}$, the geodesic never reaches $z=0$ due to energy conservation. Since $W'(z)<0$ for $z<0$,  
\[\ddot{z}=\frac{L^2}{W^3}W'(z)<0.\]
 So the geodesic will get turned around at a time $T$ where $W(z(T))=L$. Since $W(z(t))$ is monotonic for $t\in [T_-,T)$, it is invertible. So the turning point $z(T)$ is given by $z(T)=W^{-1}(L)$. 
 The above formula for the winding number holds in this case by the exact same argument as in case (1). 
\end{enumerate}

\end{proof}


One could read off the qualitative trichotomy in Theorem \ref{Thm:warped product} from phase portraits. See Figure \ref{fig:PhasePortrait} for an example with $W(z)=\sqrt{z^4+\eps^4}$ and $\eps=\frac{1}{2}$.

\begin{figure}[ht]
\includegraphics{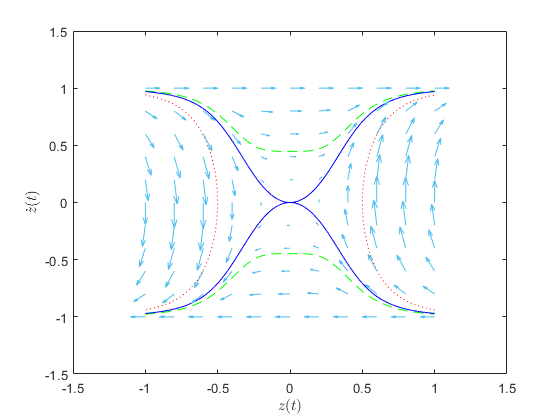}
\caption{The energy hypersurfaces $\dot{z}^2+\frac{L^2}{W(z)^2}=1$ plotted for three values of $L$. The green dashed curves show $L<W_{\min}$, the blue curves are for the separatrix $L=W_{\min}$, and the red dotted curves are for $L>W_{\min}$. The vector field of the ODE for $z$ is also shown.}
\label{fig:PhasePortrait}
\end{figure}

One can write the angular length in the same form as in \cite{GriLye}. We will deduce the formula here, since it is a neat formula and provides a link to our previous study of geodesics on singular spaces. We will not make further use of the formula in this paper.

 Let $\rho$ denote the inverse of $W_{\vert [-1,0]}$. 
\begin{Prop}
\label{Prop:ClairautLength}
Assume case (1) of Theorem \ref{Thm:warped product}. Let $\varphi_0\coloneqq \arcsin(\dot{z}(T_-))$ and $\varphi_1\coloneqq \arccos\left(L\right)$. Restrict the  geodesic $\gamma$ to end when $z(t)=0$.  Then the angular length of this part of $\gamma$ is 
\[ \angl (\gamma)=\int_{\varphi_0}^{\varphi_1} \rho'\left(\frac{L}{\cos \varphi}\right)\, d\varphi.\]
\end{Prop}
\begin{proof}
 Let $\varphi(t)\in \left[-\frac{\pi}{2},\frac{\pi}{2}\right]$ denote the Clairaut angle defined by $\dot{z}(t)=\sin(\varphi(t))$. The energy and angular momentum conservation then say $L=W(z(t))\cos \varphi(t)$. Differentiating this in time and using $\dot{L}=0$ gives 
\[\dot{\varphi}=\frac{W'(z)\cos \varphi}{W(z)}.\]
Per assumption, this is decreasing for $z \in [-1,0)$. So we can parametrize the integral with respect to $\varphi$. Since $\vert \dot{y}\vert = W\cos \varphi$, we find
\[\vert \ydot\vert \, dt=\frac{d\varphi}{W'(z)}.\]
The integral has an upper\footnote{Recall that $W'<0$ since $z<0$. So $\varphi_1<\varphi_0$. Upper and lower refers to the corresponding $z$-coordinate of the geodesic.} limit $\varphi_1$ defined by $W(0)\cos(\varphi_0)=L$ and a lower limit $\varphi_0$ defined by $W(-1)\cos(\varphi_0)=L\iff \varphi_0=\arcsin\left(\dot{z}(T_-)\right)$.
Since $\rho$ is the inverse of $W$, we have $z=\rho(W(z))$. Hence
\[ \angl (\gamma)=\int_{\varphi_0}^{\varphi_1} \frac{ d\varphi}{W'(z)}=\int_{\varphi_0}^{\varphi_1} \rho'(W)d\varphi=\int_{\varphi_0}^{\varphi_1} \rho'\left(\frac{L}{\cos \varphi}\right)\, d\varphi.\]
\end{proof}

\section{Hamiltonian formulation of the geodesic flow}
\label{Section:Hamiltonian}

In this section, we give a quick recapitulation of geodesics from a Hamiltonian point of view. This is all well known, but usually not covered in a course on Riemannian geometry. This section will be formulated for an arbitrary Riemannian manifold $(M,g)$, without any $\eps$ parameter.

Let $(M,g)$ be a Riemannian manifold. Denote points of $M$ by $x$ and points of the cotangent bundle $T^*M\xrightarrow{\pi}M$ by $(x,\chi)$. We also use $x=(x^i)$ to denote a coordinate system for $M$ and $\chi=(\chi_i)$ to denote the corresponding coordinates on the fibres of $T^*M$. In terms of these, the canonical symplectic form on $T^*M$ reads
\[\omega= dx^i \wedge  d\chi_i\]
 The Riemannian metric $g$ induces a metric on $T^*M$. Consider the \textbf{Hamiltonian}\label{HamDef1} $\calH \colon T^*M\to \R$,
\[\calH(x,\chi)=\frac{1}{2}\vert \chi\vert^2_{g_x}=\frac{1}{2}(g_x)^{ij} \chi_i \chi_j,\]
where $(g_x)^{ij}$ are the components of the inverse matrix of $(g_{ij})$ evaluated at the point $x$. 
The associated \textbf{geodesic vector field}\label{VDef1} is the Hamiltonian vector field $V$ on $T^*M$ which satisfies
\[\iota_V \omega=d\calH.\]
The \textbf{geodesic flow} is the local diffeomorphism $\phi^t$ of $T^*M$ given by the flow of $V$;
\[\begin{cases}\frac{d}{dt}\phi^t(x,\chi)=V_{\phi^t(x,\chi)} \\
\phi^0=\mathrm{id}.\end{cases}\]
Writing $(x(t),\chi(t))=\phi^t(x,\chi)$, this can be written out as
\[\dot{x}^i =\frac{\partial \calH}{\partial \chi_i}, \quad \dot{\chi}_i =-\frac{\partial \calH}{\partial x^i}.\]
The curve $x(t)=\pi((x(t),\chi(t))=\pi\circ \phi^t(x_0,\chi_0)$ is a geodesic in $(M,g)$ with initial position $x(0)=x_0$ and initial velocity $\dot{x}(0)=\chi_0^\sharp$, where the musical isomorphism $\sharp$ is the metric-induced isomorphism between the cotangent and tangent space.  The curve $(x(t),\chi(t))$ will be called a \textbf{lifted geodesic}, or simply a geodesic. The quantity $\chi(t)$ is called its \textbf{momentum}.

One could recover the usual Lagrangian equations for $x(t)$ by differentiating $\dot{x}(t)$ once more, but we will not need this.

The Hamiltonian flow always preserves the corresponding Hamiltonian. Hence the geodesic flow $\phi^t$ restricts to a local diffeomorphism of the energy hypersurface 
\[ST^*M\coloneqq \left\{(x,\chi)\, \colon\, \vert \chi\vert^2=1\right\},\] 
also known as the unit cotangent bundle. This restriction will also be called the \textbf{geodesic flow}. The geodesics lifting to the unit cotangent bundle are precisely the unit speed geodesics.

If $\psi\colon T^*M\to (0,\infty)$ is a smooth function, $\psi\cdot V$ is a smooth vector field. The integral curves of $\psi\cdot V$ are related to the integral curves of $V$ as follows. Let $(x(t),\chi(t))$ be an integral curve of $V$. Define a new time variable $\tau$ via
\[\begin{cases} \frac{d\tau}{dt}=\psi(x(t),\chi(t))^{-1}\\ \tau(0)=0.\end{cases}\]
Since this derivative is strictly positive, $t\mapsto \tau(t)$ is invertible.  
Then 
\[(\tilde{x}(\tau),\tilde{\chi}(\tau))\coloneqq (x(t(\tau)),\chi(t(\tau))\] is the integral curve of $\psi \cdot V$ with the same initial condition. In other words, rescaling the Hamiltonian vector field corresponds to a reparametrisation of the lifted geodesics. This well-known observation will be crucial in this work.  

\subsection{The warped case in the Hamiltonian picture}
We briefly sketch Section \ref{sec:warped product} in the Hamiltonian picture. The setting is $M=I\times Y$, $x=(z,y)$, $\chi=(\xi,\eta)$. The Hamiltonian reads
\begin{equation}
2\calH=\xi^2 +W^{-2} \vert \eta\vert^2.
\label{eq:WarpedHamiltonian}
\end{equation}

This is constant along a lifted geodesic. We claim that $\vert \eta\vert$ is a constant as well. A conceptual proof is to notice that the Hamiltonian vector field $V$ has horizontal components
\[\dot{y}=W^{-2} \eta^\sharp,\quad \dot{\eta}=-\frac{1}{2}W^{-2} \partial_y \vert \eta\vert^2.\]
This is the Hamiltonian vector field of $\mathcal{H}_Y=\frac{1}{2}\vert \eta\vert^2$, multiplied by $W^{-2}$. By the rescaling discussion above, we see that the $(y,\eta)$ components of the geodesic are just reparametrised $Y$-geodesics. Their energy $\mathcal{H}_Y=\frac{1}{2}\vert \eta\vert^2$ is therefore constant. 
One can also compute this directly by taking the  $t$-derivative of $1=2\calH$, to find
\[0=2\xi \dot{\xi}+ (W^{-2})_z \dot{z}\vert \eta\vert^2 +W^{-2} \frac{d}{dt}\vert \eta\vert^2=-2\xi \calH_z+2\calH_z \xi +W^{-2} \frac{d}{dt}\vert \eta\vert^2.\]
Hence $L=\vert \eta\vert$ is a constant. 

\begin{Rem}
For a (possibly generalized) \textbf{doubly warped product} the angular momentum is not constant. In the simplest case $Y=\S^1$, $h=dy^2$ this is $\R\times \S^1$ with the metric
$$g = A(z,y)\,dz^2 + W(z)^2 \,dy^2\,. $$
The Hamiltonian is $2\calH = \frac{\xi^2}{A} + \frac{\eta^2}{W^2}$, and from the Hamilton equations $\zdot = A^{-1}\xi$, $\etadot = -\frac12 \partial_y(A^{-1}) \xi^2$ one obtains directly
\begin{equation}
\label{eqn:doubly warped}
\etadot = \tfrac12 \partial_y A\, \zdot^2.
\end{equation}
Note that $\eta$ is the signed angular momentum in this case since $n=2$.
In general, i.e.\ for 
$g = A(z,y)\,dz^2 + W(z)^2 \,h$ with $h$ independent of $z$, a simple computation yields
$\Ldot = \frac12 \ip{\partial_yA}{ \frac{\eta}{|\eta|}} \zdot^2$.
In particular, $L$ is constant also for the more permissive warped product \ref{eqn:def warped product new}.

Note that this formula is the same for a generalized doubly warped product as for a doubly warped product, where $A$ is independent of $z$. 
\end{Rem}

\section{Dynamics and Angular Momentum}
\label{Section:Dynamics}

We now turn to the study of geodesics for the metric \eqref{eq:MetricAnsatz}. We start by deducing the Hamiltonian and the Hamiltonian vector field. This will allow us to study the variation of the\label{LDef2} \textbf{angular momentum}\footnote{We find this term helpful, even though $Y$ is any manifold and $\vert \eta\vert$ need not have anything to do with angles or momenta (in the symplectic geometry sense).} $L=\vert \eta\vert$.

\subsection{The Hamiltonian and the Hamiltonian vector field}
\mbox{}\\
\textit{Notation:}
We will abuse notation slightly and write $b$ for either the 1-form $b=\sum_i b_i dy^i$ or its components $(b_i)_{i}$ (in some local coordinates) and $h$ for either the symmetric two-tensor $h=\sum_{i,j} h_{ij}dy^i dy^j$ or its component matrix $(h_{ij})_{i,j}$. As usual, $h$ induces a dual metric on $T^*Y$, and we write
\[\vert b\vert^2\coloneqq h^{ij}b_i b_j\]
for the squared norm of $b$ in this metric. We will write $\ip{b}{\eta}\coloneqq h(b,\eta)=h^{ij} b_i \eta_j$ for the scalar product of two 1-forms.
Also, we will use the musical notation, turning 1-forms into vectors using the metric $h$. So $b^{\sharp}$ is the vector field with components $(b^\sharp)^i=h^{ij} b_j$. 
\medskip

Recall that we consider the metric \eqref{eq:MetricAnsatz}, 
$$g=(1-w^{\kappa}S)dz^2+2w^{2k} dz\cdot b +w^{2k} h.$$
We stress that this is very much \textit{not} a warped product, since there are mixed terms and all functions, $b,S,h$, depend on both $z$ and $y$ (besides $\eps$). 

Note that if the mixed term involving $b$ is non-zero for some $p\in M$ and $\eps>0$  then the two summands in the decomposition $T_pM= T_z I \oplus T_y Y$ at $p=(z,y)\in M$ are not orthogonal. 
Instead, the orthogonal  decomposition
$$ T_p M= (T_y Y)^\perp \oplus T_y Y $$
is relevant for what follows. 
Write $\calL_{(z,y)}=(T_y Y)^\perp$ for the normal bundle of the inclusion $\iota_z \colon Y \to \{z\} \times Y$. This is a line bundle over $M$.

\begin{Lem}
The line bundle $\calL$ is spanned by
\begin{equation}
\label{eqn:orthogonal to TY}
\partial_z - b^\sharp.
\end{equation}
The Hamiltonian associated to the metric \eqref{eq:MetricAnsatz} reads\label{HamDef2}
\begin{equation}
2\calH=\frac{\vert \eta\vert^2}{w^{2k}} + D^{-1}(\xi -\ip{b}{\eta})^2,\quad D = {1-w^{\kappa}S -w^{2k}\vert b\vert^2}.
\label{eq:ExactHamiltonianEta}
\end{equation}
\end{Lem}

\begin{proof}
This is a pointwise statement, so we consider the following situation. Let $V$ be a real vector space. Then any metric (scalar product) on $\R\times V$ can be written
$$ G = A\, dz^2 + 2 dz\cdot B + H $$
where $z$ is the coordinate on $\R$, $A\in\R_{>0}$, $B\in V^*$ is a one-form and $H\in S^2 V^*$ is a metric on $V$. We claim that the dual metric on $(\R\times V)^*=\R^*\times V^*$ is
\begin{equation}
\label{eqn:Gstar claim}
G^* = D^{-1} (\pz - B^\sharp)^2 + H^\sharp
\end{equation}
where $B^\sharp\in V=(V^*)^*$, $H^\sharp \in S^2 V$ are the duals to $B$, $H$, i.e., $H(B^\sharp,\cdot) = B$ and $H^\sharp(\alpha,\beta)=H(\alpha^\sharp,\beta^\sharp)$ for $\alpha,\beta\in V^*$, and $\partial_z$ is the standard basis vector of $\R$, dual to $dz$, and
$$ D = A - |B|_{H^\sharp}^2 .$$
This yields \eqref{eq:ExactHamiltonianEta} if we take $A=1-w^\kappa S$, $B=w^{2k}b$ and $H=w^{2k}h$, since $w^{2k} h(B^\sharp,\cdot) = w^{2k}b$ implies $B^\sharp = b^\sharp$ and $|B|^2_{H^\sharp} = |w^{2k}b|^2_{w^{-2k}h^\sharp} = w^{2k}|b|^2$
then.\footnote{In \eqref{eq:ExactHamiltonianEta}, $\xi,\eta$ are understood as coordinates, dual to $z,y$, on $T_p^*M=T_z^*I\oplus T_y^Y$, hence as linear isomorphisms $\xi\,dz \in T_z I \mapsto \xi\in\R$, $\eta\,dy \in T_y Y \mapsto \eta\in\R^n$. Using the standard identification of a vector space with its double dual, we can identify $\xi$ with the vector $\pz$ since $\xi\,dz(\pz)=\xi$, and $\ip{b}{\eta}$ 
with $b^\sharp$ since $\eta\,dy (b^\sharp) = \ip{b}{\eta}$.
}

To prove \eqref{eqn:Gstar claim} we first note that 
\begin{equation}
\label{eqn:perp vector}
 \pz - B^\sharp \perp V_0 \coloneqq \{0\} \times V \subset \R\times V
\end{equation}
(which also implies \eqref{eqn:orthogonal to TY})
since $G(\pz - B^\sharp,v) = dz(\pz)\cdot B(v) - H(B^\sharp,v) = B(v)-H(B^\sharp,v)=0$ for all $v\in V$, and
$$ |\pz - B^\sharp|^2 = D $$
since the left side equals $A \left[dz(\pz)\right]^2 - 2 dz(\pz)\cdot B(B^\sharp) + H(B^\sharp,B^\sharp) = A - 2|B|^2 + |B|^2 = D$.
This implies that, with $\calL = \Span\{\pz - B^\sharp\}$, we have the orthogonal decomposition
$$ \R\times V = \calL \oplus V_0\,,$$
and this yields the dual decomposition, orthogonal with respect to $G^*$,
$$ (\R\times V)^* = \calL^* \oplus V_0^* $$
where $\calL^*$ is considered as a subspace of $(\R\times V)^*$ by extending a functional on $\calL$ to a functional on $\R\times V$ by setting it equal to zero on $V_0$, and vice versa. In particular, we have $G^* = (G_{|\calL})^* \oplus (G_{|V_0})^*$. 

Now $\calL^*=\Span\{dz\}$ since $\ker dz = V_0$, and then $(G_{|V_0})^* = H^\sharp$ gives the second term in \eqref{eqn:Gstar claim}. To find the metric on $\calL^*$ first note that the basis $\{\pz-B^\sharp\}$ of $\calL$ has dual basis $\{dz_{|\calL}\}$ in $\calL^*$. The metric on $\calL$ is $D\,dz^2$ since $dz(\pz-B^\sharp)=1$ and $|\pz-B^\sharp|^2=D$,
so its dual metric is $D^{-1}(\pz-B^\sharp)^2$, and this gives the first term in \eqref{eqn:Gstar claim}.

\end{proof}

\begin{Rem}
One can of course find the Hamiltonian by linear algebra without referencing the normal bundle $\calL$, but we feel it helps clarify the situation somewhat.
\end{Rem}

We will specify the components of the associated geodesic vector field\label{VDef2} $V$ by writing down the components. The equations of motion of the Hamiltonian \eqref{eq:ExactHamiltonianEta} read

\begin{align}
\dot{z}&=\partial_\xi \calH=D^{-1}(\xi-\ip{b}{\eta}) \label{eq:zdot}\\
\dot{y}&=\partial_{\eta}\calH=\frac{\eta^\sharp}{w^{2k}} -D^{-1}(\xi-\ip{b}{\eta}) b^\sharp=\frac{\eta^\sharp}{w^{2k}} -\dot{z}b^\sharp \label{eq:ydot}\\
\dot{\xi}&=-\partial_z \calH , \label{eq:xidot} \\
2\dot{\eta}&=-2\partial_y \calH 
=-\frac{\partial_y\vert \eta\vert^2}{w^{2k}}+2\dot{z}\partial_y \ip{b}{\eta}  -\dot{z}^2 \partial_yD . \label{eq:etadot}
\end{align} 
 Here $\partial_y \vert \eta\vert^2= \eta^i \eta_j \partial_y h^{ij}$, $\partial_y \vert b\vert^2= b_i b_j \partial_y h^{ij} + 2 h^{ij} b_i \partial_y b_j$ and similarly for $\ip{b}{\eta}$. In short, $\eta$ does not depend on $y$ whereas $b$ does. We do not write out the $\xi$-equation since we do not need it (we can find $\xi$ up to sign from the energy) and it is generally a mess (since $h$, $S$, $b$, and $w$ all are allowed to depend on $z$).

 One notices here that the Hamiltonian vector field will not extend down to $w=0$ in these coordinates. This motivates the rescaling performed in Section \ref{Section:GeneralRescalings}. But first we study the variation of the angular momentum $\vert \eta\vert$.

\subsection{Angular momentum}

In analogy with the warped product case, we refer to $L\coloneqq \vert \eta\vert$ as the angular momentum of $\gamma$.
This is generally not constant when the metric is not of warped form, but we can estimate its variation for geodesics passing through $z=0$.    The energy \eqref{eq:ExactHamiltonianEta} tells us that
\[L \leq w^k\]
along any unit speed geodesic, but we will deduce a bound involving $L_0\coloneqq \vert \eta\vert(0)$. 
\begin{Prop}
\label{Prop:etaEst}
Assume $\gamma$ is a geodesic passing through $z=0$ at $t=0$. Then, parametrising with respect to $z$,
\begin{equation}
L(z)=(1+\mathcal{O}(z))L_0+w^{\kappa}\mathcal{O}(z).
\label{eq:etaEst}
\end{equation}
\end{Prop}
The statement will follow from a lemma about how $L$ changes along a geodesic.

\begin{Lem}
\label{Lem:etaSol}
Assume $\gamma$ is a geodesic passing through $z=0$ at $t=0$ with non-vanishing horizontal momentum, $\eta(0)\neq 0$. Then there are bounded functions $A,B$ of $z,\eps$ such that
\begin{equation}
L(z)=e^A L_0+e^A\int_0^z e^{-A} w^{\kappa}B\, d\zeta.
\label{eq:etazSolution}
\end{equation}
Furthermore, $A=\mathcal{O}(z)$.
\end{Lem}

\begin{proof}
Combining the above equations of motion, \eqref{eq:zdot}, \eqref{eq:ydot}, \eqref{eq:etadot}, we find
\begin{align}
\frac{d}{dt} L^2&=2\ip{\dot{\eta}}{\eta}+\ip{\partial_yL^2}{\dot{y}}+\partial_z L^2 \dot{z} \notag\\
&=\dot{z}\left(2\ip{\partial_y  \ip{b}{\eta}}{\eta}-\ip{\partial_yL^2}{b}+\partial_z L^2 -w^{\kappa}\dot{z} \ip{\partial_yS+ w^{2k-\kappa}\partial_y \vert b\vert^2}{\eta}\right).
\label{eq:AngularChange} 
\end{align}
One could also deduce this by differentiating the energy \eqref{eq:ExactHamiltonianEta} and using the various equations of motion. An important thing to notice about \eqref{eq:AngularChange} is that the intimidating $w^{-2k}$ factor of the $\dot{\eta}$-equation has dropped out. Indeed, \eqref{eq:AngularChange} makes sense on all of the blow-up $X$ without having to rescale the momentum variable $\eta$.

We now write the evolution equation \eqref{eq:AngularChange} in spherical coordinates,
\[L=\vert \eta\vert, \quad \psi=\frac{\eta}{\vert \eta\vert}.\]
The result is
\begin{align}
\frac{d}{dt} L^2
&=\dot{z}L \left(2L\ip{\partial_y  \ip{b}{\psi}}{\psi}-L\ip{\partial_y\vert \psi\vert^2}{b}+L \partial_z \vert \psi\vert^2 -w^{\kappa}\dot{z} \ip{\partial_yS+ w^{2k-\kappa}\partial_y \vert b\vert^2}{\psi}\right). \notag
\end{align}
 Parametrizing using $z$ instead of $t$ and dividing by $L$, we find
\[2\frac{dL}{dz}
=\left(2\ip{\partial_y  \ip{b}{\psi}}{\psi}-\ip{\partial_y\vert \psi\vert^2}{b}+ \partial_z \vert \psi\vert^2\right) L -w^{\kappa}\dot{z} \ip{\partial_yS+ w^{2k-\kappa}\partial_y \vert b\vert^2}{\psi}.\]
This has the form
\begin{equation}
\frac{dL}{dz}=a L +w^\kappa B
\label{eq:etazevol}
\end{equation}
for some functions $a,B$ of $z$ and $\eps$. 
 Since $S,b,h$ are smooth on the blow-up $X$ and $\partial_zh$ is bounded, we can find $c>0$ such that \[\vert \partial_yS\vert \leq c, \quad \left\vert \partial_y \vert b\vert^2\right\vert\leq c, \quad \left\vert \partial_z \vert \psi\vert^2\right\vert\leq c \]
\[\left\vert \ip{\partial_y \vert \psi\vert^2}{b}\right\vert\leq c , \quad   \left\vert \ip{\partial_y \ip{b}{\psi}}{\psi}\right\vert\leq c .\]
This shows that $a$ and $B$ are both bounded. The equation \eqref{eq:etazevol} is a linear first order ODE, and the solution is gotten by introducing $A=\int a\, dz$ and the integrating factor $e^{-A}$. Clearly $A$ is $\mathcal{O}(z)$.

\end{proof}

\section{Winding}
\label{Section:winding}

In this section we consider geodesics with respect to the metric $g_\eps$ on $M$ which pass the waist and 
analyse how often they 'wind around $Y$' (i.e.\ their horizontal length) uniformly as $\eps\to0$.

\subsection{Winding in the warped product case}
\label{Ssection:winding warped pr}
We first consider the case where our metric \eqref{eq:MetricAnsatz} is a warped product for each $\eps>0$. Thus\footnote{Our slightly more general form \eqref{eqn:def warped product new} can be put into this form by a change of $z$-variable, 
$\zbar = \int_0^z \sqrt{1- (z')^{\kappa}S(\eps,z')}\, dz'$. However, then $w$ is not homogeneous in $(\eps,\zbar)$ but only approximately so. See Remark \ref{Rem:WarpedWindingS} below for the change if one wants to use \eqref{eqn:def warped product new}.}
$$ g = dz^2 + w^{2k} h $$
with a metric $h=h_\eps$ on $Y$, independent of $z$ and smoothly depending on $\eps\geq0$,
and a warping function $w=w(\eps,z)$ satisfying our assumption \eqref{eqn: w monotone} and \eqref{eqn:w norm eps=0}. Recall that angular momentum $L=|\eta|$ is preserved in this case, and we can use Theorem \ref{Thm:warped product}(1) to analyse the winding behaviour of geodesics uniformly as $\eps\to0$.
Note that the assumptions on $w$ imply that $W_{min} = \eps^k$ in the notation of that theorem, so geodesics pass the waist if and only if 
$L<\eps^k$.
\begin{Prop}[Winding in the warped product case]
\label{Prop:WarpedWinding}
 Let $z_0\in I$, $z_0>0$.
 Let $\gamma$ be a geodesic starting at $z=0$ upwards with angular momentum $L<\eps^k$. Then its total angular length on the part $z\in[0,z_0]$ satisfies
 \begin{equation}
 \label{eqn:winding warped prod}
 \angl (\gamma)
 = \frac v{\eps^{k-1}} \left[C_v + \mathcal{O}(\eps^{2k-1})\right]\,,\quad v = \frac L{\eps^k}.
 \end{equation}
 Here $C_v$ is positive and depends smoothly on $v\in[0,1)$, and is independent of $z_0$.
It is  given explicitly by \eqref{eq:Cv} below.
 
The error term $\mathcal{O}(\eps^{2k-1})$ is uniform for $v\in[0,1]$ and for $z_0$ bounded away from zero.
\end{Prop}
\begin{proof}
The proof of  Theorem \ref{Thm:warped product}(1) can be restricted to $z\in[0,z_0]$ and yields
\[ \angl (\gamma) = \int_0^{z_0} \frac{L}{w^k \sqrt{w^{2k}-L^2}}\, dz.\]
We substitute $Z=\frac{z}{\eps}$. Recall that $w=\eps f(Z)$, see \eqref{eqn:def f}. Then
\[ \angl (\gamma) =\frac{v}{\eps^{k-1}}\int_{0}^{z_0/\eps}  \frac{1}{f(Z)^k \sqrt{f(Z)^{2k}-v^2}}\, dZ\]
where $v=\frac{L}{\eps^k}$.  
The properties \eqref{eqn:props f} of $f$ imply $f\geq1$ and $f(Z)\sim |Z|$ as $|Z|\to\infty$, so 
\begin{equation}
C_v\coloneqq \int_{0}^\infty \frac{1}{f(Z)^k \sqrt{f(Z)^{2k}-v^2}}\,dZ
\label{eq:Cv}
\end{equation}
is finite when $v<1$ and has the claimed properties. Now
\[\angl (\gamma)=\frac{v}{\eps^{k-1}}\left(C_v-\int_{z_0/\eps}^{\infty}  \frac{1}{f(Z)^k \sqrt{f(Z)^{2k}-v^2}}\, dZ\right),\]
and the integral over  $Z>z_0/\eps$ is $\er(\eps^{2k-1})$ uniformly.
This implies \eqref{eqn:winding warped prod}.
\end{proof}
\begin{Rem}
\label{Rem:WarpedWindingS}
If one wants the somewhat more permissive warped product \eqref{eqn:def warped product new}, i.e. $g=(1-w^{\kappa} S)dz^2+h$ with $S$ \textit{independent} of $y$, one can proceed as in the above proof. The first step in the proof changes to
\[\angl(\gamma)=\int_0^{z_0} \frac{L\sqrt{1-w^{\kappa}S}}{w^k \sqrt{w^{2k}-L^2}}\, dz.\]
By restricting to small $\eps$ and potentially shrinking the $z$-interval, we can ensure $C^{-1} \leq \sqrt{1-w^{\kappa}S}\leq C$ for some uniform $C$. So the above argument will go through with a slightly modified constant $C_v$.
\end{Rem}

We now investigate the behaviour of $C_v$ as $v\to 1$.

\begin{Lem}
\label{Lem:Cv}
Define $C_v$ for $v<1$ as in \eqref{eq:Cv}. Then    
$C_v\to\infty$ as $v\to 1$.
If $f$ vanishes to finite order $2\ell$ at $Z=0$ then $C_v$  has the asymptotic behaviour as 
$v\to 1$
\begin{equation}
\label{eqn:Cv asympt}
C_v \sim 
\begin{cases}
 C (1-v)^{-\frac{\ell-1}{2\ell}}  & \text{ if }\ell>1 \\
 C \log((1-v)^{-1})  & \text{ if }\ell=1
\end{cases}
\end{equation}
where $C$ is an explicit positive constant depending on $k$, $\ell$ and $f^{(2\ell)}(0)$.
\footnote{Explicitly, $C = 2^{-1/2}(kc)^{-1/2\ell}\int_{0}^\infty \frac1{\sqrt{\mathcal{Z}^{2\ell}+1}}\,d\mathcal{Z}$, which could be expressed in terms of the Gamma function, if $\ell>1$ and
$C = (2kc)^{-1/2}$ if $\ell=1$, with $c$  determined by $f(Z) = 1 + cZ^{2\ell} + \er(Z^{2\ell+1})$ near $Z=0$.}
 
\end{Lem}
\begin{proof}
Since $f$ is smooth and has a minimum at $Z=0$ we have  $1\leq f(Z)^{2k}\leq 1+CZ^2$ near $Z=0$, which implies $C_v\to\infty$ as $v\to1$.

Since $f(Z)>1$ for $Z\neq0$ the integral  over $Z>1$ is bounded uniformly in $v\leq1$. 
It remains to analyse the integral over $[0,1]$. Define $c\neq 0$ via
\[f(Z)=1+c Z^{2\ell}+\er(Z^{2\ell+1})\]
near $Z=0$.  Write 
$$ f(Z)^{2k} = 1 + Z^{2\ell} \mu \digamma(Z) $$
with $\mu=2kc$ and $\digamma\geq0$ smooth, $\digamma(0)=1$. Then
$$ f(Z)^{2k} - v^2 = Z^{2\ell} \mu \digamma(Z) + a,\quad a \coloneqq 1-v^2
$$
Let $\alpha = (\frac a{\mu} )^{1/2\ell}$. 
Substituting $Z=\alpha \mathcal{Z}$ in the integral over $|Z|<1$ we get
$$ C_v = \frac\alpha{\sqrt{a}} \int_{0}^{\alpha^{-1}} \frac1{f(\alpha \mathcal{Z})^{k}}
\frac1 {\sqrt{\mathcal{Z}^{2\ell}\digamma(\alpha \mathcal{Z}) + 1}}\,d\mathcal{Z}  + \text{bounded}$$
When $v\to1$ then $a\to0$, so $\alpha\to0$.
If $\ell>1$ then the dominated convergence theorem (using $\inf_{|Z|<1}\digamma(Z)>0$) yields the result (note $a\sim 2(1-v)$). 

If $\ell=1$ then $\frac\alpha{\sqrt{a}} = \mu^{-1/2}$ is independent of $a$, but the integral will have a logarithmic divergence since the integrand behaves like $\mathcal{Z}^{-1}$.
Consider the integral over $[1,\alpha^{-1}]$ first. Using  $\digamma(Z)=1+\er(Z)$ we expand
\[ \frac{1}{\sqrt{\mathcal{Z}^{2}\digamma(\alpha \mathcal{Z}) + 1}} = \frac{1}{\mathcal{Z}} \frac{1}{\sqrt{1+\er(\alpha \mathcal{Z}) + \mathcal{Z}^{-2}}} =\frac{1}{\mathcal{Z}} + \frac{\er(\alpha \mathcal{Z} + \mathcal{Z}^{-2})}{\mathcal{Z}} \]
and 
$ \frac1{f(\alpha \mathcal{Z})^{k}} = 1 + \er(\alpha \mathcal{Z})$.
The main term in the integral then is 
$$ \int_1^{\alpha^{-1}}1 \cdot \frac{1}{\mathcal{Z}}\,d\mathcal{Z} = \log\alpha^{-1} $$
and all other terms, including the integral over $\mathcal{Z}<1$, are uniformly bounded.
\end{proof}

\subsection{Winding in the general case}
We continue to write $L=\vert \eta\vert$. In the general case this is not constant along a geodesic, but we can control its change. More precisely, parametrizing geodesics by $z$, Proposition \ref{Prop:etaEst} says
\begin{equation}
\label{eqn:L estimate}
L = L_0(1 + \er(z)) + \er(z) w^\kappa,
\end{equation}
where $L_0=L(0)$. This will turn out to be sufficient to get a winding statement like in the warped case.

We write $v =v(z) = \frac L{\eps^k}$, $v_0=v(0)$.

\begin{Thm}
 Let $z_0\in I$, $z_0>0$ be sufficiently small.
 Let $\gamma$ be a geodesic starting at $z=0$ upwards with $v_0\in(0,1)$ and ending in $z_0$. Then its total angular length satisfies
\label{Thm:Winding}
\begin{equation}
 \label{eqn:windingas-variable}
 \angl(\gamma) = \frac{1}{\eps^{k-1}} \left[ v_0C_{v_0} + R\right]
\end{equation}
where the remainder $R$ depends   on $\gamma$ and satisfies\footnote{The estimates are uniform for $v_0$ in any compact subset of $[0,1)$.}
\[R = 
\begin{cases}
 \er(\eps) & \text{ if } k>2 \\
 \er(\eps\log\eps) & \text{ if } k=2
\end{cases}\]
and $C_{v_0}$ is given by \eqref{eq:Cv}.

\end{Thm}
\begin{proof}
 From the equation of motion \eqref{eq:ydot}, 
 \[\vert \dot{y}\vert=\left\vert \frac{\eta}{w^{2k}}-\dot{z} b\right\vert=\frac{L}{w^{2k}}+\er(1).\]
The energy and the equation of motion \eqref{eq:zdot} tell us 
\[\vert \dot{z}\vert=\frac{\sqrt{1-\frac{L^2}{w^{2k}}}}{\sqrt{D}},\]
so parametrising using $z$ instead of $t$ results in
\begin{align}
\vert \dot{y}\vert\, dt &=\left(\frac{L}{\sqrt{D}w^{k}\sqrt{w^{2k}-L^2}}+\er(1)\right)\, dz \notag \\ 
&=\left(\frac{L}{w^{k}\sqrt{w^{2k}-L^2}}(1+\er(w^\kappa))+\er(1)\right)\, dz.\notag
\end{align}
Let us here drop the $\mathcal{O}(1)$-contribution, since the error term will be larger eventually. 
Let
\[\mathcal{A}(x,y)=\frac{x}{y\sqrt{y^2-x^2}},\]
so that our integrand is $\mathcal{A}(L,w^k)$ to leading order. Note
\[\partial_x \mathcal{A}=\frac{y}{(y^2-x^2)^{\frac{3}{2}}}. \] By the mean value theorem, there is a $\lambda=\lambda(z)$ between $L_0$ and $L(z)$ such that 
\[\mathcal{A}(L,w^k)-\mathcal{A}(L_0,w^k)= \frac{w^k(L-L_0)}{(w^{2k}-\lambda^2)^{\frac{3}{2}}}.\]
Analysing the integral of this will give us the leading order of the error term.  We use the coordinate $Z=\frac{z}{\eps}$, $f(Z)=\frac{w}{\eps}$, and $v=\frac{L}{\eps^k}$. What we have to estimate is then
\begin{equation}
\mathcal{R}=\int_0^{z_0} \frac{w^k(L-L_0)}{(w^{2k}-\lambda^2)^{\frac{3}{2}}}\, dz=\eps^{1-k} \int_0^{z_0/\eps} \frac{f(Z)^k(v-v_0)}{ \left(f(Z)^{2k}-(\lambda/\eps^k)^2\right)^{\frac{3}{2}}}\, dZ.
\end{equation}   
As in the proof of Proposition \ref{Prop:WarpedWinding}, we assume $\eps<z_0$ and split the integral.
The piece near $0$ is 
\[I_0\coloneqq \int_0^1  \frac{f(Z)^k(v-v_0)}{ \left(f(Z)^{2k}-(\lambda/\eps^k)^2\right)^{\frac{3}{2}}}\, dZ.\]
By \eqref{eqn:L estimate}, 
\[v-v_0=\er(z)(1+f(Z)^\kappa \eps^{\kappa-k})=\eps \er(Z)\]
for $Z\in [0,1]$. For the denominator, we can again use \eqref{eqn:L estimate} for $\lambda$ to say 
\[\frac{\lambda}{\eps^k }=v_0+ \eps \er(Z).\]
Since $v_0<1$, we can ensure that this is strictly less than $1$ by shrinking $\eps$ (the amount of shrinking depends on $v_0$ and will  not be uniform on all of $[0,1)$). The integral $I_0$ will therefore be
\[I_0=\tilde{C}_{v_0} \er(\eps).\]
We now turn to the integral away from $0$, 
\[I_1\coloneqq \int_1^{z_0/\eps} \frac{f(Z)^k(v-v_0)}{ \left(f(Z)^{2k}-(\lambda/\eps^k)^2\right)^{\frac{3}{2}}}\, dZ.\]
The estimate \eqref{eqn:L estimate} then tells us
\[v-v_0=(v_0+f(Z)^k)\eps \er(Z)\]
and
\[\lambda=f(Z)^k \er(z_0).\] 
By choosing $z_0$ sufficiently smaller than $1$, we can ensure $ \lambda/\eps^k <c(z_0) f(Z)^k$ for some $c(z_0)<1$, so the denominator will simply behave like $f(Z)^{3k} \sim Z^{3k}$. The integral $I_1$ therefore goes as
\[I_1 \sim \int_1^{z_0/\eps} \frac{\eps Z^{2k+1}}{Z^{3k}}\, dZ \sim \begin{cases} \eps & k> 2 \\ \eps \log(\eps) & k=2. \end{cases}.\]
In our analysis of both $I_0$ and $I_1$ we have neglected an $\er(w^\kappa)$-factor in the integrand. This will only contribute lower order terms ($\er(\eps^{\kappa+1})$ and $\er(\eps^k)$ respectively).
\end{proof}

\begin{Rem}
One could probably perform an analysis of the behaviour of $R$ as $v_0\to 1$ similar to Lemma \ref{Lem:Cv}. A full description would be more complicated here, since the combined $\eps$- and $v_0$- behaviour would need to be analysed. 
\end{Rem}

\begin{Rem}
Note how the error term is much larger  than in the warped case, being of order $\eps$ or $\eps \log(\eps)$ versus $\eps^{2k-1}$. This stems from the second term in the estimate for $L$, \eqref{eqn:L estimate}, which in turn comes from both $S$ and $b$, see the proof of  Lemma \ref{Lem:etaSol}. Assuming both $b=0$ and $\partial_yS=0$ to improve the estimate would land us in the almost-warped setting of our previous work \cite{GriLye}.  

\end{Rem}

\begin{Rem}
Another difference between the warped and general case is that \eqref{eqn:winding warped prod} vanishes when $v_0=0$, but this is not claimed in \eqref{eqn:windingas-variable}. Indeed, it would be false, as the numerical solutions in Figure  \ref{fig:EllipsePlots}(D) show. There geodesics with $v_0=0$ are shown to move a bit along $Y$ when the metric is not of warped form. 
\end{Rem}

\begin{Rem}
Note how the horizontal length diverges as $\eps\to 0$ unless $v_0=\er(\eps^{k-1})\iff L_0=\er(\eps^{2k-1})$. This suggests that $\frac{L}{w^{2k-1}}$ could be an interesting quantity to study. In Sections \ref{Section:GeneralRescalings} and \ref{Section:alpha=2k-1} this will turn out to be correct. 
\end{Rem}

\subsection{The angle of impact}
We have been stating the winding results using the initial horizontal velocity $v_0=\frac{L_0}{\eps^k}\stackrel{\eqref{eq:ydot}}{=}\vert \dot{y}+\dot{z} b^\sharp\vert(0)\leq 1$.  Recall that we introduced the \textbf{angle of impact} $\varphi\in [0,\pi/2]$ via \label{phiDef2}
\[\cos\varphi=v_0.\]
Geometrically, this is the angle between the geodesic $\gamma(0)$ and the waist $z=0$ as explained in the introduction.
One can then rephrase Proposition \ref{Prop:WarpedWinding} as 
\begin{equation}
\label{eqn:ang length warped phi}
\angl(\gamma)=\frac{\cos\varphi}{\eps^{k-1}}\left(\mathcal{C}_{\varphi}+\er(\eps^{2k-1})\right) \ \text{ (warped product case)}
\end{equation}
and Lemma \ref{Lem:Cv} as
\begin{equation}
\label{eqn:Cphi asymp}
\mathcal{C}_{\varphi}\coloneqq C_{\cos\varphi}\stackrel{\varphi \to 0}{\sim} \begin{cases} C\varphi^{-\frac{\ell-1}{\ell}} & \text{if } \ell>1\\ C\log\varphi^{-1} & \text{ if } \ell=1.\end{cases}
\end{equation}
for a positive constant $C$.
This follows directly from the statement with $v=v_0$ and $1-v\sim \frac{\varphi^2}{2}$. Similarly for Theorem \ref{Thm:Winding}:
\begin{Thm}[Theorem \ref{Thm:Winding} in terms of $\varphi$]
\label{Thm:WindingPhi}
 Let $z_0\in I$, $z_0>0$ be sufficiently small.
 Let $\gamma$ be a geodesic starting at $z=0$ upwards at an angle $\varphi\in (0,\frac\pi2)$ and ending in $z_0$. Then its total angular length satisfies
\begin{equation}
 \label{eqn:windingas-variablePhi}
 \angl(\gamma) = \frac{1}{\eps^{k-1}} \left[ \mathcal{C}_{\varphi}\cos \varphi  + R\right]
\end{equation}
where the remainder $R$ depends on $\gamma$ and satisfies\footnote{The estimates are uniform for $\varphi$ in any compact subset of $(0,\pi/2]$.}
\[R = 
\begin{cases}
 \er(\eps) & \text{ if } k>2 \\
 \er(\eps\log\eps) & \text{ if } k=2
\end{cases}\]
and $\mathcal{C}_\varphi=C_{\cos\varphi}$ is given by \eqref{eq:Cv}, with asymptotics \eqref{eqn:Cphi asymp}
\end{Thm}

\section{Rescaling the Hamiltonian vector field}
\label{Section:GeneralRescalings}

There are of course many ways to rescale the momentum variable $\eta$ and/or the time, and we will eventually choose a special one. In this section we take a moment to discuss rescaling a bit more systematically and (hopefully) motivate our choice. 

We introduce the family of rescaled momentum variables
\begin{equation}
\label{eqn:theta eta}
\vartheta\coloneqq \frac{\eta}{w^{\alpha}}
\end{equation}
for $\alpha\in \N$. We will eventually choose $\alpha=2k-1$, and this special cases will be denoted by  $\vartheta=\theta$. To not clutter the notation, we will not write $\vartheta_{\alpha}$ to distinguish which $\alpha$ we are considering. We invite the reader to think of the rescaled cotangent bundle as a space where $\vartheta$ is smooth also at $w=0$. A reader who is happy with this as an initial explanation may skip the next subsection on a first reading.

\subsection{The rescaled cotangent bundle}
Recall from Section \ref{sec:setting} that $X\xrightarrow{\beta} [0,1)_{\eps}\times M$ is the blowup at $\eps=z=0$. There we also introduce the blown up cotangent bundle, \eqref{eqn:def X*} as 
\[X^*=X_0\times \R\times T^*Y\]
where the variable on $\R$ is $\xi$. This is a rank $n$ vector bundle over $X$, see \eqref{eqn:X* bundle}.
The fibre at a point $(x_0,y)\in X$ is 
\[X^*_{(x_0,y)}=\R\times T_yY.\]
As such, we may locally represent a section as an element of 
\[\Span_{C^\infty(X)}\{dz ,dy\}\,.\]

The function $w$ plays the role of a boundary defining function for the front face in $X$. In terms of this, we define the \textbf{rescaled cotangent bundle} for $\alpha\in\N$
\[{}^\alpha X^* \xrightarrow{\pi_{\alpha}} X\] 
as the bundle whose sections are locally of the form
 \[\Span_{C^\infty(X)} \{dz ,w^{\alpha}dy\}.\]
 This gives us a bundle by the Serre-Swan theorem. 
 We use fibre coordinates $(\xi,\vartheta)$, that is, we 
write sections of ${}^\alpha X^*$ as $\xi dz + \vartheta \,w^\alpha dy$ with $\xi,\vartheta$ smooth on $X$. Interpreting this as a section of $X^*$, i.e.\ writing it as $\xi dz + \eta\, dy$, defines
 a bundle map 
 \begin{align*}
  {}^\alpha X^* &\to X^*\\
(x_0,y,\xi,\vartheta)&\mapsto(x_0,y,\xi,w^\alpha \vartheta), 
 \end{align*}
which is a bundle isomorphism away from the front face $w=0$. This yields \eqref{eqn:theta eta}.
 
The total space of the bundle ${}^\alpha X^*$ is a manifold with corners like $X^*$, and we denote its boundary hypersurfaces by $\ff^*$, $M^*_\pm$ as for $X^*$.  Therefore,  it makes sense to talk about smoothness of vector fields on ${}^\alpha X^*$. It is on ${}^\alpha X^*$ where the rescaled Hamiltonian vector field will become a smooth vector field.

\subsection{Rescaled dynamics}

 In terms of $\vartheta$, the Hamiltonian becomes

\begin{equation}
2\calH=w^{2(\alpha-k)}\vert \vartheta\vert^2 +\frac{(\xi -w^{\alpha}\ip{b}{\vartheta})^2}{D},
\label{eq:ExactHamiltonianVartheta}
\end{equation}
where $D=1-w^{\kappa}S -w^{2k}\vert b\vert^2$ as before.
 
We denote a partial derivative by a subscript. For instance,
\[\calH_w\coloneqq \frac{\partial \calH}{\partial w}.\]
The equations of motion change to
\[\dot{z}=\calH_{\xi},\]
\[\dot{y}=w^{-\alpha}\calH_{\vartheta},\]
\[\dot{\xi}=-\calH_z\]
and
\[\dot{\vartheta}=-w^{-\alpha}\calH_y-\alpha \vartheta \frac{w_z}{w} \calH_{\xi}.\]

For a $\beta\geq 1$ (to be fixed in a moment to be $\beta=2k-\alpha$), we also rescale the time by introducing $\tau$ via 
\[\begin{cases}\frac{d\tau}{dt}=w^{\beta}(t) & \\ \tau(0)=0. & \end{cases}\] This corresponds to multiplying the Hamiltonian vector field $V$ by $w^{\beta}$. Let us write $\mathcal{V}=w^\beta V$.  Denoting the $\tau$-derivative by $'$, the above equations read
\begin{align*}
z'&=w^{\beta}\calH_{\xi}=\frac{w^{\beta}(\xi-w^{\alpha} \ip{b}{\vartheta})}{D},\\
y'&=w^{\beta-\alpha}\calH_{\vartheta}=w^{\alpha+\beta -2k} \vartheta^\sharp +z' b^\sharp,\\
\xi'&=-w^{\beta}\calH_z,\\
\vartheta'&=-w^{\beta-\alpha}\calH_y-\alpha w^{\beta-1} \vartheta w_z \calH_{\xi}.
\end{align*}

From these it should be clear that choosing $\beta$ large, the vector field $\mathcal{V}=w^{\beta}V$ becomes smooth. But then one risks $\mathcal{V}$ vanishing identically on the lifted front face $\mathrm{ff}^*$. A look at the $y'$-equation should convince one that we want $\alpha+\beta=2k$. This will impose some restrictions coming from the $\vartheta'$-equation. We formulate this as a main theorem. We will refer to the components of the vector field using the associated equation of motion. So $z'$ will refer the $z$-component of $\mathcal{V}$ and so on. We use two\footnote{We are are suppressing the second $E$-coordinate $E_-$. The situation near the lower edge $Z\to -\infty \iff E_-\to 0$ is completely analogous to the upper edge $Z\to \infty\iff E\to 0$.} projective coordinates on the front face,
\[Z=\frac{z}{\eps}, \quad f(Z)=w(Z,1)=\frac{w}{\eps},\]
\[E=\frac{\eps}{z},\quad F(E)=w(1,E)=\frac{w}{z}.\]
\begin{Thm}
\label{Thm:GeneralRescaling}
Let $V$ denote the Hamiltonian vector field of the Hamiltonian \ref{eq:ExactHamiltonianVartheta}.  
Let $\alpha\in [1,2k-1]$ be some integer subject to $\alpha\leq k+\frac{\kappa}{2}$ and let $\beta=2k-\alpha$. Then the rescaled Hamiltonian vector field $\mathcal{V}_{\alpha}\coloneqq w^\beta V$ extends to a smooth vector field on ${}^\alpha X^*$. This extended vector field is tangent to the boundary when $\alpha>1$.

For the special value $\alpha=2k-1$ we write
\[\theta=\frac{\eta}{w^{2k-1}} = \vartheta\,,\]
and then
the restriction of $\mathcal{V}\coloneqq\mathcal{V}_1$ to the front face has the components
\[Z'=f(Z)\xi,\quad \xi'=0,\quad y'=\theta^\sharp,\quad \theta'=-(2k-1)\xi\theta-\frac{1}{2}\partial_y (\vert \theta\vert^2+f(Z)^2 \xi^2 S).\]

\end{Thm}
\begin{proof}
The rescaled Hamiltonian vector field is clearly smooth away from the front face $\{w=0\}$. So we have to investigate the behaviour near the front face.

\noindent\textbf{$z$-component}
\medskip
The radial equations of motion in terms of the projective coordinates become
\[Z'=\frac{f(Z)^\beta \eps^{\beta-1}(\xi-w^{\alpha} \ip{b}{\vartheta})}{D}\]
and
\[E'=-\frac{E z^{\beta-1} F(E)(\xi-w^{\alpha} \ip{b}{\vartheta})}{D}.\]
These are clearly smooth in $Z,\eps$ and $E,z$ respectively. We note that the front face restrictions read
\[Z'=\begin{cases} f(Z)\xi & \beta=1\\ 0 & \beta>1\end{cases}\]
and
\[E'=\begin{cases} -EF(E)\xi & \beta=1\\ 0 & \beta>1\end{cases}.\]

\noindent\textbf{$y$-component}
\medskip
The $y'$-component is smooth since we are choosing $\alpha+\beta=2k$. It reads
\[y'=\vartheta^\sharp + \eps Z' b^\sharp=\vartheta^\sharp + z' b^\sharp\]
in the $(Z,\eps)$- and $(E,z)$-coordinates respectively.

The restriction to the front face is simply $y'=\vartheta^\sharp$.

\noindent\textbf{$\vartheta$-component}
\medskip
To discuss the momentum variables $\vartheta$, we first note that
\[w_z = \partial_z (\eps f(z/\eps))=\partial_z (zF(z/\eps)),\]
hence
\[w_z=f'(Z)=F(E)-EF'(E)\]
in the projective coordinates. 

The $\vartheta'$-equation reads
\[\vartheta'=-\alpha \frac{w_z}{w} z' \vartheta -w^{\beta-\alpha} \partial_y \calH.\]
The first term reads 
\[-\alpha f'(Z)w^{\beta-1}\frac{(\xi-w^{\alpha} \ip{b}{\vartheta})}{D}\]
in $(Z,\eps)$-coordinates and
\[-\alpha (F(E)-EF'(E))w^{\beta-1}\frac{(\xi-w^{\alpha} \ip{b}{\vartheta})}{D}\]
in the $(E,z)$-coordinates. This is smooth and restricts to 
\[-\alpha f'(Z)\xi \delta_{\beta,1}\]
(and similarly in the other coordinate system) on the front face $\{w=0\}$. The second term of $\vartheta'$ reads
\[-w^{\beta-\alpha}\partial_y \calH=-\frac{w^{\alpha+\beta-2k}}{2} \partial_y \vert \vartheta\vert^2+ \frac{w^{\beta}(\xi-w^{\alpha} \ip{b}{\vartheta})}{D} \partial_y \ip{b}{\vartheta}-\frac{w^{\beta-\alpha}(\xi-w^{\alpha} \ip{b}{\vartheta})^2}{D^2} \partial_yD.\]
The first term here is smooth since $\alpha+\beta=2k$. The second term is clearly smooth. For the last term, we observe that 
\[-\partial_yD= w^{\kappa} \partial_yS+ w^{2k} \partial_y \vert b\vert^2.\]
 Using $\beta=2k-\alpha$, the total $w$-exponents therefore read 
\[\beta-\alpha+\kappa=2k-2\alpha+\kappa \]
and
\[\beta-\alpha+2k=4k-2\alpha.\]
These are positive if and only if the assumed bound
\[\alpha\leq k+\frac{\kappa}{2}\]
holds.

The restriction to the front face depends on the value of $\alpha$. When $\alpha<2k-1$, the restriction is simply $\vartheta'=-\frac{1}{2} \partial_y \vert \vartheta\vert^2$. When $\alpha=2k-1$ (and $\kappa= 2k-2$), we get the additional term $(z')^2 \partial_y S$, where $z'$ is to be interpreted in $(Z,\eps)$-coordinates as above.

\noindent\textbf{$\xi$-component}
\medskip
We start by noting
\[2\partial_z \calH=\partial_z (w^{2(\alpha-k)} \vert \vartheta\vert^2)-2 z' \partial_z (w^{\alpha}\ip{b}{\vartheta})+(z')^2\partial_z (1-w^\kappa S-w^{2k}\vert b\vert^2).\]
The last couple of terms extend smoothly to the front face since $z'$ extends. The first term reads
\[\partial_z (w^{2(\alpha-k)}\vert \vartheta\vert^2)=w^{2(\alpha-k)}\left(-2k \frac{w_z}{w}\vert \vartheta \vert^2+(\partial_z h^{ij})\vartheta_i \vartheta_j\right).\]
Here $\partial_z h^{ij}$ is shorthand for $(\partial_z h^{-1})^{ij}=-(h^{-1} (\partial_z h) h^{-1})^{ij}$. 
This will get multiplied by $w^{\beta}$. Using $\alpha+\beta=2k$ again, we find
\[w^{\beta} \partial_z (w^{2(\alpha-k)}\vert \vartheta\vert^2)=-2k w^{\alpha-1} w_z \vert \vartheta\vert^2 +w^\alpha (\partial_z h^{ij})\vartheta^i \vartheta^j.\]
This extends smoothly to all of ${}^{\alpha}X^*$, including the (lifted) front face.

The restriction to the front face vanishes unless $\alpha=1$.
\end{proof}

\begin{Cor}
The rescaling with $\alpha=2k-1$ gives a smooth Hamiltonian vector field if $\kappa\geq 2k-2$. 
\end{Cor}
We will get back to the case  $\alpha=2k-1$ in Section \ref{Section:alpha=2k-1}. For now we just note that $\alpha=k$ is the largest $\alpha$ for which $\vartheta$ is bounded when $\calH$ is restricted to the unit cotangent bundle. The value $\alpha=2k-1$ is the only one for which the $(Z,\xi)$-dynamics does not vanish on the front face.
\begin{Rem}
A common (and often fruitful) approach to rescaling is to choose $\alpha$ in such a way that $\vartheta$  becomes unit length (with respect to $g$) on the front face. This corresponds to $\alpha=k$. For conical singularities, $k=1$, $2k-1=k$, so there is no difference. For cuspidal singularities, $k\geq 2$, the rescaling $\alpha=k$ turns out to be suboptimal since the $Z$-dynamics vanishes on the front face, causing the linearisation to be degenerate. This was also observed in \cite{GrGr15}.  
\end{Rem}

The time rescaling is crucial to ensure a smooth vector field. Dynamically the fact that $\mathcal{V}$ is tangent to the boundary has the effect of stretching out the time a geodesic spends near the waist $z=0$ in the following sense.
\begin{Lem}
\label{Lem:GeneralwBound}
Assume the setup of Theorem \ref{Thm:GeneralRescaling}. Let $\gamma_{\eps}(t)=(z_{\eps}(t),y_{\eps}(t))$ be a family of unit speed geodesics with $z_{\eps}(0)=0$. 
Then, for any $T>0$ there is an $\eps_0>0$ and a constant $C_{\eps_0}$ (both depending on $T$) such that
\[ w(\tau)\coloneqq w(z_{\eps}(t(\tau)),\eps)\leq C_{\eps_0} \eps\]
for all $\eps\leq \eps_0$ and $\tau\in [-T,T]$. In particular, $\lim_{\eps\to 0} w(\tau)=0$ pointwise and uniformly on compacts $K\subset \R$. 
\end{Lem}
\begin{proof}
This follows by the rescaled vector field being smooth all the way up to the lifted front face $\ff^*$ and smooth dependence on parameters.

\end{proof}

\section{Almost vertical dynamics: Focussing}
\label{Section:alpha=2k-1}
This section deals with the complement of the winding regime, i.e. geodesics hitting the waist 'almost vertically'. To this end, we turn to the rescaling of Theorem \ref{Thm:GeneralRescaling} with $\alpha=2k-1, \beta=1$. Thus, we consider the \textbf{rescaled Hamiltonian vector field}\label{calVDef} 
$$\mathcal{V}
=w\cdot V$$ 
on ${}^{2k-1}X^*$. 
To ensure smoothness of $\mathcal{V}$ down to the front face, we have to impose the condition $\kappa=2k-2$, due to the restriction $\alpha\leq k + \frac\kappa2$ in Theorem \ref{Thm:GeneralRescaling}.\footnote{When $k=2$ and $\eps=0$, these cover the metrics on any space with a cuspidal singularity of order $k=2$ as defined in \cite[Section 7]{GrGr15} and \cite[Section 6]{GriLye}. For $k\geq3$, these are not the most general metrics - the $S$-term could be $w^k S$ instead of $w^{2k-2}S$. See \cite{BeyGri:IGIC} for a correction to \cite[Proposition 7.3]{GrGr15}. See also Appendix \ref{Section:GeneratingExamples} for a proof that a large family of examples satisfy the $\kappa=2k-2$ requirement.}

This is a long section, so an overview is in order. We briefly recall what the rescaled Hamilton vector field looks like in Subsection \ref{Subsection:RescaledDyn}. In \ref{Subsection:ffDyn} we determine its restriction to the front face, and also introduce some notation for the spaces needed in this section. In Subsection \ref{Subsection:Crit} 
we determine the critical  points of the vector field and find that they lie on the intersection of the front face and $M^*_{\calH,+}$ (the cotangent versions of the face $M_+$ of $X$).
In \ref{Subsection:ffLin} we describe the linearised dynamics near a critical point on the front face. Subsection \ref{Subsection:Globalff} shows that front face curves flow into critical points, also when they start far away (but still on the front face).  
In Section \ref{Subsection:eps=0Dyn} we analyse the dynamics near the critical points on $M^*_{\calH,+}$ with the help of \cite{GrGr15}. In
Section \ref{ssec:dyn near boundary} we put these pieces together to deduce the behaviour of geodesics near  the boundary and prove the Focussing Theorem. The remaining subsections give a result when $S^\pm$ is constant (as opposed to Morse), and a rewriting of the front face dynamics as a time-dependent Hamiltonian system.

\subsection{The rescaled Hamiltonian vector field in the case $\alpha=2k-1$}
\label{Subsection:RescaledDyn}
 Recall that integral curves of $V$ parametrized by $t$ correspond to integral curves of  $\mathcal{V}=w\cdot V$ 
parametrized by $\tau$, where
the time variable $\tau$ is defined by\label{tauDef}
\[\frac{d\tau}{dt}=w^{-1}, \quad \tau(0)=0,\]
with $w$ evaluated at $z(t)$. We continue to denote $\tau$-derivatives by $'$, so $y'=\frac{dy}{d\tau}$ and so on. 
Also recall the special variable name in this case:
\[\theta=\frac{\eta}{w^{2k-1}} = \vartheta \quad \text{with } \alpha=2k-1.\label{thetaDef}\]
In terms of $\theta$ the Hamiltonian \eqref{eq:ExactHamiltonianVartheta} takes the form
\begin{equation}
\label{eqn:Ham 2k-2}
2\calH=\xi^2+w^{2k-2} \cG,\quad\text{ where } \cG=\cG_0+\cG_1,
\end{equation}
with leading term
\[\cG_0\coloneqq S\xi^2+\vert \theta\vert^2\]
and a remainder
\[\cG_1=\mathcal{O}(w^2)\xi^2+\mathcal{O}(w)\ip{b}{\theta} \xi+\mathcal{O}(w^{2k})\vert \theta\vert^2.\]

The equations of motion in this rescaled time and momentum were found in the proof of Theorem \ref{Thm:GeneralRescaling} and read\footnote{We do not keep a track of $\xi'$-equation anymore, since one can find $\xi$ from the energy, and for the almost vertical geodesics we are studying, $\xi \approx 1$ anyway.}
\begin{align}
z'&=\frac{w\xi - w^{2k}\ip{b}{\theta}}{D}
\label{eq:z'}\\
y'&=\theta^\sharp -z' b^\sharp\label{eq:y'}\\
\theta'&=-\frac{1}{2}\partial_y\vert \theta\vert^2+z'\cdot \left(\partial_y \ip{b}{\theta}-(2k-1)\frac{w_z}{w}\theta-\frac{z'}{2w^2}\left(\partial_yS +w^2 \partial_y \vert b\vert^2\right)\right),  \label{eq:theta'}
\end{align}
where $D=1-w^{2k-2}S-w^{2k}\vert b\vert^2$ as before.
Note that $w_z$ is homogeneous of degree zero in $(\eps,z)$, so it is smooth on the blow-ups $X$, $X^*$.
Since we always consider unit speed geodesics these equations take place on an energy hypersurface $2\calH=1$. Therefore, $\xi$ and $w^{2k-2}\cG$ are bounded along the motion. So $w^{k-1}\vert \theta\vert=w^{-k}\vert \eta\vert$ is bounded (but not necessarily $\theta$ itself). Using in addition that $S,b,h$ are smooth on $X$ we can simplify the equations of motion:
\begin{align}
z'&=w\xi +\mathcal{O}(w^{k+1}) \tag{\ref{eq:z'}'}\label{eq:z'Leading}\\
y'&=\theta^\sharp + \mathcal{O}(w) \tag{\ref{eq:y'}'} \label{eq:y'Leading}\\
\theta'&=-(2k-1)w_z \xi \theta -\frac{1}{2}(\xi^2 \partial_yS +\partial_y \vert \theta\vert^2)+w \xi \partial_y\ip{b}{\theta} + \mathcal{O}(w).  \tag{\ref{eq:theta'}'} \label{eq:theta'Leading}
\end{align} 
where all bounds are uniform on the energy hypersurface.
It is worth stressing that the term $w \xi \partial_y\ip{b}{\theta}$ is not necessarily $\mathcal{O}(w)$ uniformly, since $\theta$ is $\mathcal{O}(w^{1-k})$.
For later purposes, we also record 
\begin{equation}
\frac{d}{d\tau}\vert \theta\vert^2=-2(2k-1)\xi w_z \vert \theta\vert^2-\xi^2\ip{\partial_y S}{\theta}+ (\vert \theta\vert +\vert \theta\vert^2) \er(w)  .
\label{eq:NormthetaEvol}
\end{equation} 
One can either deduce this from the above equations, or use \eqref{eq:AngularChange}.

\subsection{Restriction to the front face}
\label{Subsection:ffDyn}
We recall the two\footnote{We note again that we would strictly speaking need the two $E$-coordinates $E$ and $E_-$, since $E=0$ corresponds to both $Z=\pm \infty$.} projective coordinates 
\[Z\coloneqq \frac{z}{\eps}, \quad E \coloneqq \frac{\eps}{z},\]
valid on $\eps>0$ and $z\neq 0$ respectively. In terms of these, we have
\[w=\eps f(Z)= z F(E),\]
where $f(Z)\coloneqq w(Z,1)$ and $F(E)\coloneqq w(1,E)$,
and
\[w_z=f'(Z)=F(E)-E F'(E).\]
In this section we only consider upward moving unit speed geodesics. Their rescalings live on the space
\[X^*_{\mathcal{H}}\coloneqq {}^{2k-1}X^*\cap \{2\calH=1,\ \xi>0\}.\]
This is a manifold with corners\footnote{It is a p-submanifold of ${}^{2k-1}X^*$ (see \cite{Mel:DAMWC} or \cite{Gri:BBC}) since $\partial_\xi\calH\neq0$ on it.} with  boundary hypersurfaces
\[\ff^*_{\calH} = \ff^*\cap X^*_\calH = ({}^{2k-1} X^*)_{\vert \ff} \cap \{\xi=1\},\]
(note that $2\calH = \xi^2$ at the front face $\{w=0\}$ by \eqref{eqn:Ham 2k-2}) 
and 
\[M_{\pm,\calH}^*=M_\pm^*\cap X^*_\calH = ({}^{2k-1} X^*)_{\vert M_{\pm}} \cap \{2\calH=1,\ \xi>0\}.\]
Extending \eqref{eqn:ff diffeo} we have a natural diffeomorphism and vector bundle structure
\begin{equation}
\label{eqn:ffstar diffeo}
\ff^*_{\calH} \cong \Rbar_Z\times T^*Y \to \Rbar_Z\times Y \cong \ff
\end{equation}
where $(y,\theta)$ is considered as an element of $T^*Y$. 
On the other hand, extending \eqref{eqn:Mpm diffeo} we have a fibre (not vector) bundle
\begin{equation}
\label{eqn:Mstar bundle}
M_{+,\calH}^* \to M_+ \cong I_+ \times Y
\end{equation}
where the fibre over $(z,y)$ is $\{(\xi,\theta):\,2\calH=1, \xi>0\}\subset \R_\xi \times T_y^*Y$, with the Hamiltonian in \eqref{eqn:Ham 2k-2} taken with $\eps=0$, so $w=z$.\footnote{For example, in the warped product case $2\calH=\xi^2 + z^{2k-2}|\theta|^2$, so the fibre over a point $z>0$ is a half ellipsoid.}

Here is an application of Theorem \ref{Thm:GeneralRescaling}. It describes the limit of the exponential map when restricted to a sufficiently narrow cone around $\eta=0$.
\begin{Prop}
\label{Prop:LimitingDynamicbeta=1}
Fix $y_0\in Y$ and a sequence of directions $\eta_{0,\eps}\in T^*_{y_0}Y$ such that $\theta_0\coloneqq \lim_{\eps \to 0}  \eta_{0,\eps}/\eps^{2k-1}$ exists. Let $\gamma_{\eps}=(z_{\eps},y_{\eps},\xi_{\eps},\eta_{\eps})$ be the family of upwards moving unit speed geodesics starting in $(z=0,y_0)$ with $Y$-momentum $\eta_{0,\eps}$.  

Then the curves $\tau\mapsto(Z_\eps(\tau),y_{\eps}(\tau),\xi_\eps(\tau),\theta_\eps(\tau))$, where $Z_\eps = \frac{z_\eps}\eps$ and $\theta_\eps = \frac{\eta_\eps}{w(\eps,z_\eps)^{2k-1}}$, converge to a smooth curve in  $\ff^*_{\calH}$, uniformly for bounded $\tau$. The limiting curve satisfies 
\begin{align}
Z'&=f(Z), \label{eq:Z'ff} \\
 y'&=\theta^\sharp, \label{eq:y'ff}\\
\theta'&=-(2k-1)f'(Z) \theta-\frac{1}{2}\partial_y\left (S+\vert \theta\vert^2\right).\label{eq:theta'ff}
\end{align}
\end{Prop}

\begin{proof}

 Due to Lemma \ref{Lem:GeneralwBound}, $\lim_{\eps\to 0} w(\tau)=0$ locally uniformly in $\tau$. The convergence then follows by smooth dependence on initial data since the rescaled vector field is smooth up to the boundary. The limiting equations were deduced as part of Theorem \ref{Thm:GeneralRescaling}.
 
 \end{proof}

Note that \eqref{eq:Z'ff}-\eqref{eq:theta'ff}  describe the components of the rescaled vector field $\mathcal{V}$ restricted to the front face $\ff^*_{\calH}$.

For later use we note that the front face equations of motion make sense also in the $E$-coordinates. The $E$- and $\theta$-equations then read
\begin{align}
E'&=-E F(E) , \label{eq:E'ff} \\ 
  \theta'&=-(2k-1)(F(E)-E F'(E))\theta -\frac{1}{2}\partial_y\left (S+\vert \theta\vert^2\right). \label{eq:Etheta'ff} 
\end{align}

On the front face, \eqref{eq:NormthetaEvol} reads
\begin{equation}
\frac{d}{d\tau}\vert \theta\vert^2=-2(2k-1) f'(Z)\vert \theta\vert^2-\ip{\partial_y S}{\theta}
\label{eq:ffNormthetaEvol}
\end{equation} 
in the $Z$-coordinate.

\begin{Lem}
\label{Lem:thetaBounded}
The limiting momentum $\theta$ of Proposition \ref{Prop:LimitingDynamicbeta=1} is uniformly bounded in $\tau$.
\end{Lem}

\begin{proof}
We offer two proofs. 
Note that \eqref{eq:Z'ff} says $Z\to \infty$ as $\tau \to \infty$, since $f\geq 1$.  
The first proof is to note that Proposition \ref{Prop:etaEst} with $\kappa=2k-2$ says 
\[\vert \theta\vert=\er(z)w^{-1}=\er(Z)f(Z)^{-1}=\er(1),\]
as $\vert Z\vert \to \infty$.

For the second proof, observe that $f'(Z)\to 1$, so we can assume $f'(Z)\geq 1/2$ for all large $\tau$. Looking at \eqref{eq:ffNormthetaEvol}, one notices that the negative first term will dominate when $\theta$ is large enough, causing $\vert \theta\vert^2$ to decrease. This makes $\vert \theta\vert^2$ bounded.
\end{proof}

\subsection{Critical points of the rescaled geodesic vector field}
\label{Subsection:Crit}  
The blown-up space $X$ has two corners of codimension 2, given by $\ff\cap M_\pm$.
In the $E$-coordinate system $\ff\cap M_+$ is given by $\{z=E=0\}\times Y$. Also, $\ff\cap M_+$ can be identified with the zero section $\theta=0$ of $\ff^*_\calH\cap M^*_{+,\calH}$, see \eqref{eqn:ffstar diffeo}. We denote these identifications by
\begin{equation}
\label{eqn:ybar notation}
\begin{array}{ccccc}
Y & \cong & \ff \cap M_+ & \hookrightarrow & \ff^*_\calH\cap M^*_{+,\calH} \\
y & \mapsto & \ybar & \mapsto & \ybar^* 
\end{array}
\end{equation}
so $\ybar=(E=0,z=0,y)$ and $\ybar^*=(E=0,z=0,y,\theta=0)$. Similarly, we have $\ybar_-$ and $\ybar_-^*$ for $\ff\cap M_-$.
In terms of the identification of $y$ and $\ybar$ we have functions
\[S^\pm \coloneqq S_{\vert \ff \cap M_{\pm}}\colon Y\to \R.\]
The functions $S^+$ and $S^-$  can be different. For simplicity, we will formulate most results for the upper portion of the front face, $\ff\setminus M_-$. The corresponding results for $\ff\setminus M_+$ are analogous, with $(E,S^+)$ replaced by $(E_-,S^-)$.

Outside the front face $\ff^*_\calH$ the vector field $V$ is simply the geodesic vector field for a smooth Riemannian metric, so it is non-zero. Therefore, all critical points of $\mathcal{V}=wV$ must lie on the front face. 
On the front face $\calV$ has no critical points in the $Z$-coordinate system, i.e., away from $M_{\pm,\calH}^*$, since the $Z$-component is $f(Z)\geq 1$ by  \eqref{eq:Z'ff}.
In the $E$-coordinate system, $\calV$ is given by \eqref{eq:y'ff}, \eqref{eq:E'ff}, \eqref{eq:Etheta'ff}. Its critical points are at the points $(E,y,\theta)$ satisfying
\[E=0,\quad \partial_yS^+(y)=0,\quad \theta=0.\]
Using the notation \eqref{eqn:ybar notation} we obtain:
\begin{Lem}
\label{Lem:crit pts}
The critical points of the rescaled geodesic vector field on $X^*_\calH$ are precisely the points $\ybar^*_c$ for critical points $y_c$ of $S^+$, and  $\ybar^*_{c,-}$ for critical points $y_c$ of $S^-$.
\end{Lem}
The linearised flow around such a critical point is not too complicated:
\begin{Lem}
\label{Lem:LinFlow}
Choose coordinates such that $y_c$ corresponds to $0$. Then the linearised flow around the critical point $\overline{y}_c^*$ reads
\begin{equation}
\begin{pmatrix} z'\\ E'\\ y'\\ \theta' 
\end{pmatrix}=\begin{pmatrix}
1 & 0 & 0 & 0\\ 0 & -1 & 0 & 0 \\ -b_+^\sharp &0 & 0& h^{-1} \\ Q & 0 & -\frac{1}{2}S_{yy}^+ & -(2k-1)\mathds{1}
\end{pmatrix}\begin{pmatrix}
z \\ E \\ y \\ \theta
\end{pmatrix}.
\label{eq:LinFlow}
\end{equation}
Here 
\[Q=\frac{1}{2}\partial_z \partial_y S(\overline{y}_c),  \quad  b_+^\sharp= b^\sharp(\overline{y}_c),\]
and $S_{yy}^+$ means the Hessian of $S^+$ computed with respect to the metric\footnote{Actually, the metric does not matter at a critical point. The reason being that $\mathrm{Hess}(S^+)_{ij}=S^+_{,ij}-\Gamma_{ij}^p S^+_{,p}=S^+_{,ij}$ at a critical point.} $h$.
\end{Lem}
\begin{proof}
This is mostly straight forward from the equations of motion \eqref{eq:z'}, \eqref{eq:y'}, and \eqref{eq:theta'}. Some parts deserve remarks. In $(z,E)$-coordinates, \eqref{eq:z'} reads
\[z'=zF(E)(1+z^{2k-1}F(E)^{2k-1})/D,\]
which linearises to $z F(0)=z$. This also gives the claimed form of the linearised $y$-equation. We are assuming that $\partial_E S=0$ on $\ff \cap M_{\pm}$, hence there is no $Q$-like term in the second column.
\end{proof}

\subsection{Dynamics on the front face near the critical points}
\label{Subsection:ffLin}
We now consider the restriction of $\calV$ to $\ff^*_\calH$.
Let $y_c$ be a critical point of $S_+$, and choose coordinates $y$ near $y_c\in Y$ in terms of which $y_c=0$. Then the linearisation of $\calV$ around $\ybar^*_c$ is given by \eqref{eq:LinFlow} with the $z$-rows and columns deleted, i.e.
\begin{equation}
\label{eq:ffLin}
\begin{pmatrix}
E'\\ y'\\ \theta'
\end{pmatrix}=\begin{pmatrix}
-1 & 0 & 0\\ 0& 0 &h^{-1} \\ 0& -\frac{1}{2} S_{yy}^+ & -(2k-1)\mathds{1} 
\end{pmatrix}\begin{pmatrix}
E \\ y \\ \theta
\end{pmatrix}.
\end{equation}

We note how the $E$-equation has completely decoupled.
We proceed by finding the eigenvalues of the linear  $(y,\theta)$-system,
\[\begin{pmatrix}
 y'\\ \theta'
\end{pmatrix}=\begin{pmatrix}
 0 &h^{-1} \\ -\frac{1}{2} S_{yy}^+ & -(2k-1)\mathds{1} 
\end{pmatrix}\begin{pmatrix}
 y \\ \theta´
\end{pmatrix}.\]
One readily checks that eigenvectors with eigenvalue $\mu$ can be written
\begin{equation}
\label{eqn:eigenvectors in ff}
\begin{pmatrix}
u\\ \mu h u
\end{pmatrix}
\end{equation}
where $u$ is an eigenvector of $h^{-1}S^+_{yy}$
with eigenvalue $-2\mu(\mu+2k-1)$.
If $a_j$ denote an eigenvalues for $h^{-1}S_{yy}^+$, then the corresponding eigenvalues for this linear system are 
\[\mu_j^{\pm}=\frac{-(2k-1) \pm \sqrt{(2k-1)^2-2a_j}}{2}.\]
For each $a_j<0$, there is one negative and one positive eigenvalue, $\mu_j^+$ and $\mu_j^-$.

Let us formulate these computations as a lemma before we study the stable and unstable manifolds.

\begin{Lem}
\label{Lem:CritPoints}
The linearisation around a critical point $\ybar^*_c$ as in Lemma \ref{Lem:CritPoints} is described by \eqref{eq:ffLin}. The eigenvalues are $-1$ with multiplicity at least\footnote{The multiplicity increases by $1$ for each $a_j$ with $a_j=4(k-1)$.} $1$ and 
\[\mu^{\pm}_j =\frac{-(2k-1) \pm \sqrt{(2k-1)^2-2a_j}}{2},\]
where $a_j$ is an eigenvalue of $h^{-1}S^+_{yy}$ at $y_c$, with eigenvectors given in \eqref{eqn:eigenvectors in ff}.
\end{Lem}

We now assume $S^+$ is a Morse function. This ensures that the above eigenvalues $a_j$ are all non-zero, so the eigenvalues $\mu_{j}^{\pm}$ all have non-vanishing real parts. The critical point $\ybar^*_c$ is then said to be \textbf{hyperbolic}. Recall that the \textbf{Morse index} of a Morse function $\mathrm{ind}_{y_c} S^+$ at a critical point $y_c$ is the dimension of the negative eigenspace of the Hessian at the point $y_c$. The point $y_c$ is a minimum if the index is $0$, a maximum if the index is $n-1=\dim Y$, and a saddle point otherwise. 

Let 
\[\Phi_\tau\colon \ff^*_{\calH} \to \ff^*_{\calH}\]
be the flow taking an initial point $(Z_0,y_0,\theta_0)$ to the unique solution of the front face dynamics \eqref{eq:Z'ff}-\eqref{eq:theta'ff} at time $\tau$. 
Since $Y$ is compact, $S^+$ has finitely many critical points. For any such critical point $y_c\in \mathrm{Crit}(S^+)$, we introduce the \textbf{stable manifold} $\mathcal{M}_{y_c}^s$ as the set of all points in $\ff^*_{\calH}\setminus M_{-,\calH}$ which flow into $\ybar^*_c$;
\[(Z_0,y_0,\theta_0)\in \mathcal{M}_{y_c}^s \iff \lim_{\tau \to \infty} \Phi_{\tau}(Z_0,y_0,\theta_0)=\ybar^*_c.\]

\begin{Prop}[The stable manifold theorem]
\label{Prop:StableMfd}
Assume $S^+$ is Morse and let $y_{c}$ be a critical point of $S^+$. Then the stable manifold $\mathcal{M}_{y_c}^s$ is an injectively immersed submanifold of codimension 
\[\mathrm{codim} \,\mathcal{M}_{y_c}^s =\mathrm{ind}_{y_c}(S^+).\]
Furthermore, there is a neighbourhood $U_{y_c}$ of $\ybar^*_c$ such that $\mathcal{M}_{y_c}^s\cap U_{y_c}$ is an embedded submanifold. Finally, the convergence as $\tau\to \infty$ is exponential.
\end{Prop}
\begin{proof}
The description locally near the critical point follows from a standard version of the stable manifold theorem, see for instance \cite[Theorem 9.4]{Teschl}. The tangent space of $\mathcal{M}_{y_c}^s$ at $\ybar^*_c$ is the span of the eigenvectors of Lemma \ref{Lem:CritPoints} whose eigenvalues have negative real part. The dimension is given as the number of eigenvalues with negative real part. The eigenvalues are $-1$ and the $2n-2$ numbers $\mu_j^{\pm}$. The $\mu_j^-$ always have negative real part, and there are $n-1$ of them. The $\mu_j^+$ have negative  real part precisely when the $a_j$ are positive, hence there are $n-1-\mathrm{ind}_{y_c}(S^+)$ of these, giving a total dimension of $2n-1- \mathrm{ind}_{y_c}(S^+)$. Since $\ff^*_{\calH}$ is $(2n-1)$-dimensional, the codimension is $\mathrm{ind}_{y_c}(S^+)$. The exponential convergence follows from hyperbolicity, see for instance \cite[Theorem 9.5]{Teschl}. 

For the global structure, see for instance \cite[9.1 Theorem]{Smale}.
\end{proof}

From this we obtain:
\begin{Prop}
\label{Prop:ffGeod}
Assume $S^+$ is a Morse function. Then:
\begin{itemize}
\item For any maximum $y_{\max}$ of $S^+ $, the geodesics in $\ff$ running into $\ybar_{\max}$ foliate a neighbourhood of $\ybar_{\max}$ in $\ff$. 
\item For any minimum $y_{\min}$ of $S^+$, every geodesic in $\ff$ which starts near $\ybar^*_{\min}$ runs into $y_{\min}$.
\end{itemize}

\end{Prop}
In other words:
\begin{itemize}
\item
For each point on $\ff$ near $\ybar_{\max}$ there is a \textit{unique} direction so that the geodesic starting at that point in this direction ends at $\ybar_{\max}$.
\item 
For each point on $\ff$ near $\ybar_{\min}$
and \textit{all} directions near $\theta=0$ (i.e.\ close to the $Z$-direction) the geodesic starting at that point in this direction ends at $\ybar_{\min}$.
\end{itemize}
\begin{proof}
At a maximum, the Morse index is $n-1$. By the stable manifold theorem, Proposition \ref{Prop:StableMfd}, the stable manifold $\mathcal{M}_{y_{\max}}^s\subset \ff^*_{\calH}$ is smooth and $n$-dimensional. The tangent space is spanned by the eigenvectors $\begin{pmatrix}
1\\0\\0
\end{pmatrix}$ and 
 $\begin{pmatrix}
0 \\ u_j\\ \mu_j^- h u_j
\end{pmatrix}$, where $u_j$ are the eigenvectors of $h^{-1}S^+_{yy}$. Projecting $\mathcal{M}^s_{y_{\max}}$ onto $\ff$ via $\pi\colon {}^{2k-1} X^*\to X$ is therefore a local diffeomorphism.
At a minimum, the Morse index is $0$. Hence the stable manifold has dimension $2n-1$. 

\end{proof}
For later use, we also note that the same proof also shows that for any critical point $y_c$ of $S^+$, the projection $\pi\colon \mathcal{M}_{y_c}^s \to \ff$ is a submersion near $\ybar^*_c$.

\subsection{Global dynamics on the front face}
\label{Subsection:Globalff}
We have seen in Proposition \ref{Prop:LimitingDynamicbeta=1} that sufficiently vertically starting geodesics have a front face limit with well-defined dynamics. The next result says that these limiting front face curves necessarily flow into critical points of $\mathcal{V}$.

\begin{Thm}
\label{Thm:Asymptotics}
Assume $S^+$ has isolated critical points\footnote{This is slightly weaker than being Morse.}.

Let $\sigma=(Z,y,\theta)$ be a solution of the front face dynamics \eqref{eq:Z'ff}, \eqref{eq:y'ff}, \eqref{eq:theta'ff} (and $\xi=1$) with arbitrary initial value in $\ff^*_{\calH}\setminus M^*_{-,\calH}$. 
 
 Then $\sigma$ converges to $(E=0,y_{\mathrm{crit}},0)$ as $\tau\to \infty$, for some critical point $y_{\mathrm{crit}}$ of $S^+$.
\end{Thm}
\begin{proof} Let us first assume $\sigma(0)\notin M^*_{+,\calH}$ - we will deal with this case in the end.
The $E\to 0\iff Z\to \infty$-convergence follows from $Z'=f(Z)\geq 1$. 
We first assume $\partial_Z S=0$, since the argument becomes more transparent then. We claim $\mathcal{G}=S+\vert \theta\vert^2$ is essentially a Lyapunov function for the non-autonomous $(y,\theta)$-system. The evolution of $\vert \theta\vert^2$ on the front face, \eqref{eq:ffNormthetaEvol}, says
\[\frac{d}{d\tau}\vert \theta\vert^2=-2(2k-1)f'(Z)\vert \theta\vert^2-\ip{\partial_yS}{\theta}.\]
 We then compute 
\[\frac{d}{d\tau} \mathcal{G} =\frac{d}{d\tau}\vert \theta\vert^2+\ip{\partial_y S}{y'}=-2(2k-1)f'(Z)\vert \theta\vert^2.\]
Since $Z\to \infty,$ $f'(Z)>\frac{1}{2}$ for all large enough $\tau$ (recall that $f'(Z)\xrightarrow{Z\to \infty}=1$). So $\mathcal{G}$ is decreasing and clearly bounded from below, hence it converges as $\tau\to \infty$. But then
\[\int_0^\infty f'(Z)\vert \theta\vert^2\, d\tau\]
exists. Since we know from Lemma \eqref{Lem:thetaBounded} that $\theta$ is uniformly bounded, we also know that $\frac{d}{d\tau} \vert \theta\vert^2$ is uniformly bounded. Hence $\vert \theta\vert \to 0$.
 
When $\partial_Z S\neq 0$ the above argument does not show that $\mathcal{G}$ is decreasing. But we still know that $\mathcal{G}$ is bounded, hence 
\[\mathcal{G}(T)-\mathcal{G}(\tau_0)=\int_{\tau_0}^T \frac{d}{d\tau} \mathcal{G}\, d\tau=-2(2k-1)\int_{\tau_0}^Tf'(Z)\vert \theta\vert^2\, d\tau +\int_{\tau_0}^T \partial_Z S Z'\, d\tau.\]
is uniformly bounded in $T$. We will argue that the second integral on the right hand side is uniformly bounded, hence the first integral also has to be bounded. One can then argue as when $\partial_ZS=0$. We know $Z'=f(Z)\sim Z$ for $Z\to \infty$. Hence $Z(\tau)\geq c e^{\tau}$ as $\tau\to \infty$ for some $c>0$. 
If $\partial_ZS=\mathcal{O}(Z^{-1-\delta})$ for some $\delta>0$, then 
\[\int_{\tau_0}^T \partial_ZS Z'\, d\tau\lesssim \int_{\tau_0}^T e^{-\delta \tau}\, d\tau\leq \frac{1}{\delta}\]
is uniformly bounded.  
 Since $S$ is assumed smooth on the front face, we get \[\partial_ZS=\frac{\partial E}{\partial Z} \partial_ES=-Z^{-2}\partial_ES =\er(Z^{-2}),\]
 i.e. we can choose $\delta=1$.

The above arguments show that $\theta$ converges to $0$ and that $\vert \theta\vert$ is in $L^2$. This turn ensures that $y'\to 0$, but not directly that $y$ also converges;  this would follow if $\vert \theta\vert$ were in $L^1$. The above arguments however show that $S(Z(\tau),y(\tau))$ converges. 
Since the flow is complete, this limit has to be a critical point of the rescaled geodesic vector field $\mathcal{V}$, i.e. $S$ converges to a critical value $S^+(y_{\mathrm{crit}})$. Per assumption, the critical points of $S^+$ are isolated, so $y(\tau)$ has to converge to a critical point of $S^+$.

When $\sigma(0)\in M_{+,\calH}^*$, the above argument simplifies, since $E(\tau)=0$ for all $\tau$ in that case, and consequently $S=S^+$. The change in $\vert \theta\vert^2$, \eqref{eq:NormthetaEvol}, reads
\[\frac{d}{d\tau}\vert \theta\vert^2=-2(2k-1)\vert \theta\vert^2-\ip{\partial_yS^+}{\theta},\]
making $\mathcal{G}=S^+ + \vert \theta\vert^2$ a Lyapunov function. The rest of the argument is the same.
 
\end{proof}

\begin{Rem}
\label{Rem:ReachingMax}
The proof shows one more thing, namely that $\vert \theta\vert^2+S$ is decreasing when $Z\geq 0$ and $\partial_ZS=0$. Hence a necessary condition for a curve to end up in $y_{\infty}$ is
\begin{equation}
S(y_{\infty})\leq S(y_0)+\vert \theta_0\vert^2.
\label{eq:Sineq}
\end{equation}
So if $y_{0}$ is close to a minimum and $y_{\infty}$ is a maximum, $\vert \theta_0\vert^2$ has to be large enough for this inequality to be fulfilled.
In particular, a geodesic with $S(y_0)<S^+_{\max}$ and $\theta_0=0$ can not land in a maximum. 

 The inequality \eqref{eq:Sineq} imposes no restriction if $y_{\infty}$ is a global minimum.
\end{Rem}

\begin{Rem}
In our guiding example \eqref{eqn:example}, $\partial_ZS\neq 0$, but it does satisfy $\partial_Z S=\mathcal{O}(Z^{-1-{2k}})$. See Lemma \ref{Lem:ExampleS}. Also for the family of examples discussed in Appendix \ref{Section:GeneratingExamples} it is generally the case that $\partial_Z S\neq 0$. 
\end{Rem}

\begin{Rem}
Theorem \ref{Thm:Asymptotics} is an analogue of \cite[Theorem 1]{GrGr15}. 
\end{Rem}

Combining this with our description of the stable manifolds of Proposition \ref{Prop:StableMfd}, we get that the every point lies in some stable manifold. To make this a bit more precise, we need some notation. Let $\mathrm{ind}(S^+)$ denote all the Morse indices,
\[\mathrm{ind}(S^+)=\left\{\mathrm{ind}_{y_c}(S^+) \, \colon\, y_c\in \mathrm{Crit}(S^+)\right\}.\] 
For any $\lambda\in \mathrm{ind}(S^+)$, let $\mathcal{M}_\lambda\subset \ff^*_{\calH} \setminus M_{-,\calH}^*$ denote the initial values of curves landing in a critical point with Morse index $\lambda$, 
\[\mathcal{M}_{\lambda}=\bigcup_{\substack{y_c\in \mathrm{Crit}(S^+) \\ \mathrm{ind}_{y_c}(S^+)=\lambda}} \mathcal{M}^s_{y_c}.\]
 
\begin{Cor}
\label{Cor:GeneralDirections}
Assume $S^+$ is Morse. Then 
\[\ff^*_{\calH}\setminus M^*_{-,\calH} =\bigcup_{\lambda\in \mathrm{ind}(S^+)} \mathcal{M}_\lambda,\]
where each $\mathcal{M}_\lambda$ is an injectively immersed submanifold of codimension $\lambda$.
\end{Cor}  
Note that $\mathcal{M}_{\min}=\mathcal{M}_0$ and $\mathcal{M}_{\max}=\mathcal{M}_{(n-1)}$ are always present in this union.
 
\begin{Rem}
\label{Rem:MorseTop}
The topology of $Y$ puts some restrictions on the possible Morse indices. For instance, the Poincar\'{e}-Hopf theorem says
\[\chi(Y)=\sum_{y_c\in \mathrm{Crit}(S^+)} (-1)^{\mathrm{ind}_{y_c}(S^+)},\]
where $\chi(Y)$ is the Euler characteristic. Another example is Reeb's sphere theorem, \cite[Theorem 4.1]{MilnorMorse}; if $S^+$ only has one maximum and minimum and no saddle points, then $Y$ is homeomorphic to $\S^{n-1}$.    
\end{Rem}

Another way of describing the front face dynamics is to look at directions eventually hitting a critical point of a given Morse index. For any $Z_0\in \R$, we introduce
\[\Upsilon_{Z_0}^\lambda\coloneqq \mathcal{M}^\lambda\cap \{Z=Z_0\}\]
as the sets of initial conditions $(y,\theta)$ such that $(Z_0,y,\theta)$ flows into a critical point of index $\lambda$.
\begin{Prop}
\label{Prop:Discrete}
$\Upsilon_{Z_0}^\lambda$  a codimension $(\lambda+1)$ submanifold of $\ff^*_{\calH}$.
\end{Prop}
\begin{proof}
By the proof of Proposition \ref{Prop:ffGeod}, $\begin{pmatrix}
1 \\ 0\\ 0
\end{pmatrix}$ is tangent to $\mathcal{M}_\lambda$, and $\{Z=Z_0\}$ clearly has the remaining tangent directions. $\mathcal{M}_\lambda$ and $\{Z=Z_0\}$  therefore intersect transversally. The codimension of a transverse intersection is the sum of the codimensions. By the stable manifold theorem, $\Upsilon_{Z_0}^\lambda$ is a submanifold (and not just injectively immersed) for $Z_0$ large enough. But the flow maps the $\Upsilon_{Z}^\lambda$ diffeomorphically unto each other for varying $Z$, hence they are all submanifolds.
\end{proof}
Using the identification \eqref{eqn:ffstar diffeo}, we may identify $\Upsilon_{Z_0}^\lambda$ with a submanifold of $T^*Y$. We will implicitly do so from now on.


\subsection{Dynamics on $M_+$ near the critical points}
\label{Subsection:eps=0Dyn}
Recall that $M_+^*$ is a subset of $(z,y,\xi,\theta)$-space, see \eqref{eqn:Mstar bundle}. The rescaled Hamilton vector field $\calV$ restricted to $M_+^*$ is given by the equations \eqref{eq:z'Leading}-\eqref{eq:theta'Leading}, where one takes $\eps=0$. This amounts to replacing $w$ by $z$ and $w_z$ by $1$.
Geometrically, this is the rescaled Hamilton vector field on a space with a cuspidal singularity. This was studied in \cite{GrGr15}. 
As shown in Lemma \ref{Lem:CritPoints} the critical points of $\calV$ on $M_+^*$ are given by 
\[z=0, \quad \theta=0=\partial_yS^+(y)\]
and $\xi=1$. Near such a point the equation $2\calH=1$ can  be solved for $\xi$ locally by \eqref{eqn:Ham 2k-2}, so $M_+^*$ is parametrised by $z,y,\theta$ there.
The linearisation of $\calV$ around such a critical point is computed in\footnote{The matrix here is slightly different than in \cite{GrGr15} since our definition of $S$ is different: our $S$ is $k(k-1)$ times the $S$ in \cite{GrGr15}.
} \cite[Equation 3.10]{GrGr15} and \eqref{eq:LinFlow} (by deleting the rows and columns with $E$): 
\[\begin{pmatrix}
z'\\ y'\\ \theta'
\end{pmatrix}=\begin{pmatrix}
1 & 0 & 0\\ -b^\sharp_+ & 0 &h^{-1} \\ Q& -\frac{1}{2} S_{yy}^+ & -(2k-1) 
\end{pmatrix}\begin{pmatrix}
z \\ y \\ \theta
\end{pmatrix}.\]
 One can find the eigenvalues of this matrix, and \cite[Lemma 4.1]{GrGr15} says they are almost as for the linearised front face dynamics (but the eigenvectors do not decouple between the $z$ and the $(y,\theta)$ subspaces). The difference is that the $z$-direction corresponds to a positive eigenvalue when $\xi=1$. This is in contrast to the negative eigenvalue, $E'=-  E$, of \eqref{eq:ffLin}.  

The stable manifold theorem can be applied here as well, and this was done in \cite{GrGr15}. One result is as follows.
\begin{Prop}[{\cite[Proposition 4.2]{GrGr15}}]
\label{Prop:z=0Geod}
Assume $S^+$ is Morse. 
\begin{itemize}
\item For any maximum $y_{\max}$ of $S^+$, there is a neighbourhood $\overline{y}_{\max}\in U\subset M_{+}$ such that  $U$
 is foliated by geodesics starting at $\overline{y}_{\max}$. 
\item For any minimum $y_{\min}$ of $S^+$, there is a unique geodesic starting at $\overline{y}_{\min}\in M_+$.
\end{itemize}

\end{Prop}

The combination of Proposition \ref{Prop:ffGeod} and Proposition \ref{Prop:z=0Geod} in the case of a minimum $y_{\min}$ is illustrated in Figure \ref{Fig:ffGeod}. Geodesics on the front face near $y_{\min}$ all flow into $y_{\min}$, whereas there is only one exiting direction.

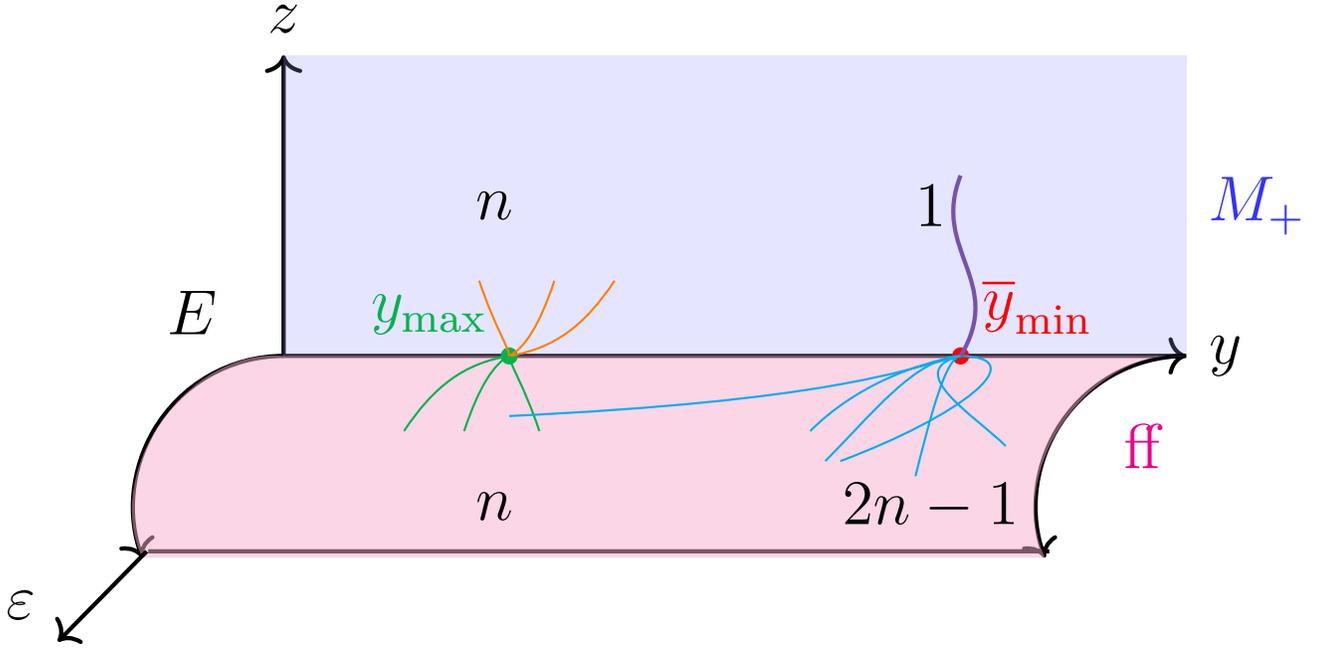
\begin{figure}
\scalebox{2}{
\begin{tikzpicture}
\draw[thick,->] (-3,0) coordinate (h1) -- (3,0) coordinate (h2) node[anchor=west] {$y$};
\draw[thick,->] (-3,0) -- (-3,2) node[anchor=south] {$z$};
\draw[thick,->] (-3,0) arc (90:200:1cm) coordinate (h3);
\fill[blue!40,nearly transparent] (-3,0) -- (3,0) -- (3,2) -- (-3,2) -- cycle;
\node[blue!80] at (3,1) [anchor=west] {$M_+$};
\node at (-3.3,0) [anchor=south east] {$E$};
\draw[thick,-] (-3.9,-1.3)--(2.09,-1.3);
\draw[thick,->] (3,0) arc (90:200:1cm) coordinate (h4);
\path[fill=magenta!40, opacity=.5] (h3) to  (h4) to [ bend left=60] (h2) to  (h1) to [ bend right=60] (h3); 
\node[magenta] at (3,-0.6) [anchor=east]  {$\ff$};
\draw[thick,->] (-3.91,-1.3)--(-4.5,-1.9) node[anchor=south east] {$\varepsilon$};
\filldraw[color=red](1.5,0) circle (0.05) node[anchor=south west] {$\overline{y}_{\min}$};
\draw[cyan] (0.5,-0.5) .. controls (0.8,-0.2) and (1.2,-0.05) .. (1.5,0);
\draw[cyan] (0.6,-0.7) .. controls (0.9,-0.4) and (1.2,-0.01) .. (1.5,0);
\draw[cyan] (0.7,-0.7) .. controls (1.5,-0.4) and (2.0,-0.01) .. (1.5,0);
\draw[cyan] (1.2,-0.8) .. controls (1.3,-0.4) and (1.4,-0.01) .. (1.5,0);
\draw[cyan] (1.8,-0.6) .. controls (1.6,-0.4) and (1.1,-0.1) .. (1.5,0);
\node at (1.3,-1) {$2n-1$};
\node at (-1.6,-1) {$n$};
\node at (1.3,1) {$1$};
\node at (-1.6,1) {$n$};
\draw[thick,royalpurple] (1.5,0) .. controls (1.8,0.5) and (1.3,0.7) .. (1.5,1.2);
\filldraw[color=green!70!blue](-1.5,0) circle (0.05) node[anchor=south east] {$y_{\max}$};
\draw[color=green!70!blue] (-2.2,-0.5) .. controls (-2.0,-0.2) and (-1.8,-0.05) .. (-1.5,0);
\draw[color=red!50!yellow] (-1.5,0) .. controls (-1.2,0.05) and (-1.0,0.2) .. (-0.8,0.5);
\draw[color=green!70!blue] (-1.8,-0.5) .. controls (-1.7,-0.2) and (-1.6,-0.05) .. (-1.5,0);
\draw[color=red!50!yellow] (-1.5,0) .. controls (-1.4,0.05)  and (-1.3,0.2) .. (-1.2,0.5);
\draw[color=green!70!blue] (-1.3,-0.5) .. controls (-1.4,-0.2) and (-1.5,-0.05) .. (-1.5,0);
\draw[color=red!50!yellow] (-1.5,0) .. controls (-1.5,0.05) and (-1.6,0.2)  .. (-1.7,0.5);
\draw[cyan] (-1.5,-0.4) .. controls (.6,-0.3) and (1.1,-0.1) .. (1.5,0);
\end{tikzpicture}
}
\caption{The blowup $X$. A maximum and a minimum of $S^+$ are shown.
Near the maximum $y_{\max}$, the geodesics flowing into/out of it foliate neighbourhoods in $\ff$ and $M_+$, respectively. Near the minimum $y_{\min}$, all nearly vertical geodesics on the front face flow into it, and there is a unique one leaving it in $M_+$ (the $yz$-plane in the picture). Note how these geodesics can intersect each other near $y_{\min}$, and do not yield a foliation. We also note how there are geodesics starting near the maximum which still flow into the minimum.
The lower numbers $n,2n-1$ refer to the dimensions of the corresponding stable manifolds in $\ff^*_{\calH}$, as computed in Proposition \ref{Prop:ffGeod}. The upper numbers $n,1$ refer to the dimensions of the unstable manifolds in $M_{+,\calH}^*$, as in Proposition \ref{Prop:z=0Geod}.}
\label{Fig:ffGeod}

\end{figure}

\subsection{Dynamics near the boundary}
\label{ssec:dyn near boundary}
\renewcommand{\calO}{\mathcal{U}}
We now use the results of the previous sections about the dynamics \textit{on} the boundary faces to draw conclusions about the dynamics \textit{near} the boundary, in the case where $S^+$ is a Morse function.
We will show that most geodesics which start at the waist $z=0$ almost vertically will be focussed for small $\eps$, meaning that they will run close to one of finitely many distinguished geodesics. We need some terminology to make this precise. Fix some height  $0<z_1\in I$ and consider the restrictions of $X^*_{\calH}$ to $\{z=z_1\}$ and the lift $\beta^*\{z=0\}$:
\[X^*_{\calH,z_1}\coloneqq X^*_{\calH} \cap \{ z=z_1\},\]
\[X^*_{\calH,0} \coloneqq  X^*_{\calH} \cap \{ Z=0\}.\]
See Figure \ref{Fig:PoincareMap}.
Points on $X^*_{\calH,0}$ may be written as triples $(\eps,y,\theta)$. 
Note that $X^*_{\calH,0}$ is a manifold with boundary, and the boundary is the \lq edge\rq\ 
$$\ff^*_{\calH,0} \coloneqq X^*_{\calH,0} \cap \ff^*_\calH$$ at $\eps=0$.
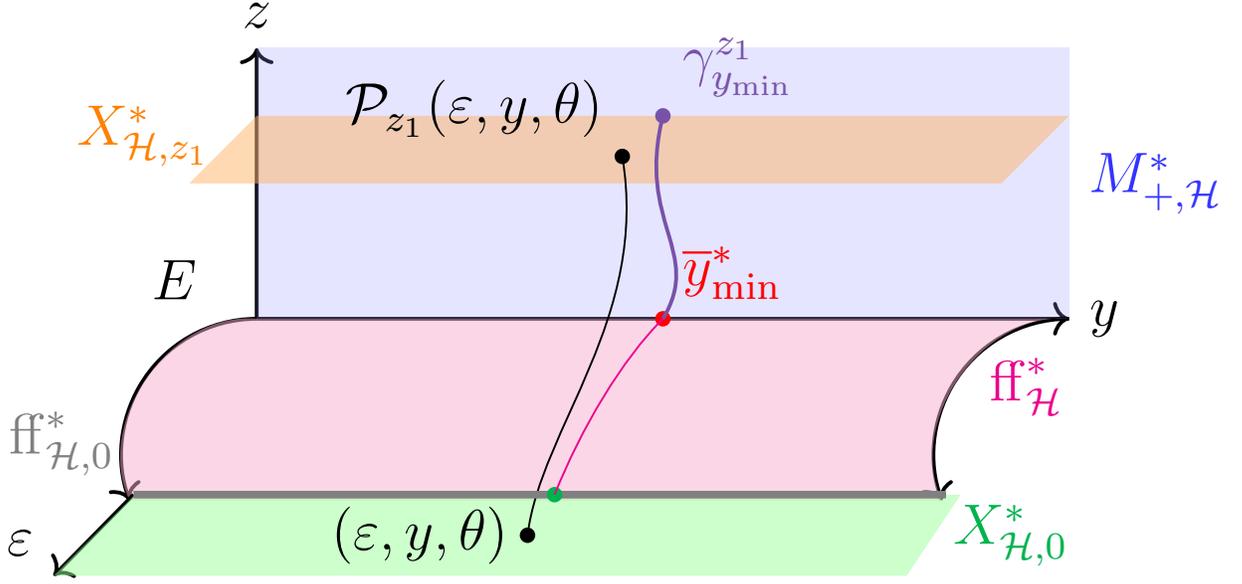
\begin{figure}
\scalebox{1.8}{
\begin{tikzpicture}
\draw[thick,->] (-3,0) coordinate (h1) -- (3,0) coordinate (h2) node[anchor=west] {$y$};
\draw[thick,->] (-3,0) -- (-3,2) node[anchor=south] {$z$};
\draw[thick,->] (-3,0) arc (90:200:1cm) coordinate (h3);
\fill[blue!40,nearly transparent] (-3,0) -- (3,0) -- (3,2) -- (-3,2) -- cycle;
\node[blue!80] at (3,1) [anchor=west] {$M_{+,\calH}^*$};
\node at (-3.3,0) [anchor=south east] {$E$};
\draw[thick,-] (-3.9,-1.3)--(2.09,-1.3);
\draw[thick,->] (3,0) arc (90:200:1cm) coordinate (h4);
\path[fill=magenta!40, opacity=.5] (h3) to  (h4) to [ bend left=60] (h2) to  (h1) to [ bend right=60] (h3); 
\node[magenta] at (3.1,-0.5) [anchor=east]  {$\ff^*_{\calH}$};
\draw[thick,->] (-3.91,-1.3)--(-4.5,-1.9) node[anchor=south east] {$\varepsilon$};
\filldraw[color=red](0,0) circle (0.05) node[anchor=south west] {$\overline{y}_{\min}^*$};
\fill[green!40,opacity=0.5] (-3.9,-1.3) -- (-4.5,-1.9) -- (1.8,-1.9) -- (2.2,-1.3) -- cycle;
\node[green!70!blue] at (2,-1.6) [anchor=west] {$X^*_{\calH,0}$};
\fill[orange!60,opacity=0.5] (-3,1.5) -- (3,1.5) -- (2.5,1.0) --  (-3.5,1.0) -- cycle;
\node[orange] at (-3.2,1.35) [anchor=east] {$X^*_{\calH,z_1}$};
\filldraw[color=black](-1,-1.6) circle (0.05) node[anchor=east] {$(\eps,y,\theta)$};
\draw[black] (-1,-1.6) .. controls (-0.9,-0.8) and (-0.1,0) .. (-0.3,1.2);
\filldraw[color=black](-0.3,1.2) circle (0.05) node[anchor=south east] {$\mathcal{P}_{z_1}(\eps,y,\theta)$};
\draw[thick,royalpurple] (0,0) .. controls (0.3,0.5) and (-0.2,0.7) .. (0,1.5);
\filldraw[color=royalpurple](0,1.5) circle (0.05) node[anchor=south west] {$\gamma_{y_{\min}}^{z_1}$};
\draw[ultra thick,gray] (-3.91,-1.3)--(2.09,-1.3);
\node[anchor=south east, gray] at (-3.91,-1.3) {$\ff^*_{\calH,0}$};
\filldraw[green!70!blue](-0.8,-1.3) circle (0.05);
\draw[magenta] (-0.8,-1.3) .. controls (-0.6,-0.8) and (-0.3,-0.3) .. (0,0);
\end{tikzpicture}
}
\caption{$X^*_{\calH}$ with the level hypersurfaces $X_{\calH,0}^*$, $X_{\calH,z_1}^*$. An $\eps>0$ its geodesic and its nearby boundary geodesics, the Poincar\'{e} map $\mathcal{P}_{z_1}$ and the point $\gamma_{y_{\min}}^{z_1}$ are also illustrated. The $\theta$-direction is suppressed.}
\label{Fig:PoincareMap}

\end{figure}

Consider the Poincar\'{e} map 
\[\mathcal{P}_{z_1} \colon X^*_{\calH,0} \setminus \ff^*_{\calH,0}\to X^*_{\calH,z_1},\]
sending an initial point $(\eps, y,\theta)$ to the intersection of the (image of the) associated geodesic with $X^*_{\calH,z_1}$. This intersection exists and is unique since $z$ is strictly increasing along the flow for all $z$ close to $0$; recall that we may assume the interval $I$ to be sufficiently small.
For any minimum $y_{\min}$ of $S^+$, we denote by $\gamma_{y_{\min}}$ the unique 
upward moving geodesic in $M_+^*$ starting in $\overline{y}_{\min}^*$ (see Proposition \ref{Prop:z=0Geod}), and by 
$\gamma_{y_{\min}}^{z_1}$ its intersection with $X^*_{\calH,z_1}$.  
  Let $d$ denote any smooth distance function\footnote{By a \textbf{smooth distance function $d$} we mean a Riemannian distance function of some smooth Riemannian metric.} on $X^*_{\calH,z_1}$ (they are all equivalent near $\eps=0$ by the compactness of $Y$). 

\begin{Thm}
\label{Thm:focussing}
There is $\varrho>0$ and for each minimum $\ymin$ of $S^+$ an open subset $\calO_\ymin$ of $\ff^*_{\calH,0}$ so that: 
\begin{enumerate}
 \item[(a)] For each compact subset $K\subset X^*_{\calH,0}$ satisfying $K\cap \, \ff^*_{\calH,0} \subset \calO_\ymin$ there is a constant $C$ so that
 for all $q=(\eps,y,\theta)\in K$
 \[ d(\calP_{z_1}(q), \gamma_\ymin^{z_1}) \leq C \eps^\varrho;\]
 \item[(b)] The union $\calO\coloneqq \displaystyle\bigcup_{\ymin} \calO_\ymin$ is dense in $\ff^*_{\calH,0}$.
\end{enumerate}
\end{Thm}
In fact, the proof shows that the set in (b) is the complement of a finite union of submanifolds of $\ff^*_{\calH,0}$ of codimension at least $1$. The codimensions occurring correspond to the positive Morse indices of $S^+$.  The constants $C,\varrho$ can be chosen to be the same for all $z_1$ in any compact subset of $I\cap (0,\infty)$. 

Before proving Theorem \ref{Thm:focussing}, we give a version in terms of sequences of geodesics on $(M,g_\eps)$.  Let $\gamma_{\eps}$ be a family of unit speed geodesics starting in $(z=0,y_{\eps}(0),\theta_{\eps}(0))$, where $(y_{\eps}(0),\theta_{\eps}(0)) \in T^*Y$. Assume the initial values converge,
\[\lim_{\eps\to 0} (y_{\eps}(0),\theta_{\eps}(0))=(y_0,\theta_0).\] 
Let 
\[\gamma_{\eps}^{z_1}\coloneqq \mathcal{P}_{z_1}(\eps,y_{\eps},\theta_{\eps})\]
denote the intersection of $\gamma_{\eps}$ with the level set $X^*_{\calH,z_1}$ as above.
\begin{Cor}
\label{Cor:Focussing}
There is an open, dense subset $\mathcal{U}\subset T^*Y$ and constants $C,\varrho>0$ such that if $(y_0,\theta_0)\in \mathcal{U}$, then there is a minimum $y_{\min}$ of $S^+$ such that, for any family $\gamma_\eps$ of geodesics as above,
\begin{equation}
d(\gamma_{\eps}^{z_1},\gamma_{y_{\min}}^{z_1})\leq C \eps^{\varrho}.
\end{equation}
\end{Cor}
Note that by \eqref{eq:y'Leading}, $\theta_\eps(0)$ is related to the initial horizontal velocity $\ydot_\eps(0)$ via $\eps \ydot_\eps(0) = \theta_\eps(0)^\sharp + \mathcal{O}(\eps)$. This yields the formulation of the Focussing Theorem in the introduction.

\begin{proof}[Proof of Theorem \ref{Thm:focussing}]
Consider the flow in $\ff^*_\calH$. 
We define $\calO_\ymin$ to be the intersection of the stable manifold of $\overline{y}_{\min}^*$ with $\ff^*_{\calH,0} = \ff^*_\calH\cap X^*_{\calH,0}$. By  Proposition \ref{Prop:Discrete}
(or its proof) this is open, and the complement of the union in (b) is a finite union of  submanifolds of codimension at least $1$. This proves (b).

To prove (a) we first consider a neighbourhood of $p^*\coloneqq \overline{y}_{\min}^*$. Since $\gamma_\ymin$ is transversal to $\ff^*_\calH$ we may choose coordinates $v$ replacing $(E,y,\theta)$ so that
$(z,v)$ are local coordinates centred at $p^*$ and so that $\gamma_\ymin$ is given by $v=0$ locally. Since the rescaled geodesic vector field $\mathcal{V}$ is tangential to $\gamma_\ymin$ this implies that the $v$-component of $\mathcal{V}$ vanishes at $v=0$. Also, since $\mathcal{V}$ is tangential to $\ff^*_{\calH}=\{z=0\}$ its $z$-component vanishes at $z=0$. Therefore, we may write the flow equations as
 $$ z' = A(z,v)\,z\,,\quad v' = B(z,v)\,v$$
where $A$ and $B$ are a scalar resp.\ $m\times m$-matrix-valued function depending smoothly on $(z,v)$ (where $m=\dim\ff^*_\calH=2n-1$). Also,
$A(0,0)=1$ by \eqref{eq:LinFlow}, and all eigenvalues of $B(0,0)$ have negative real parts by Proposition \ref{Prop:StableMfd}. 
This implies that there is an open neighbourhood $U_{p^*}$ of $p^*$ in $X^*_{\calH}$, a scalar product $\langle\cdot,\cdot\rangle$ in $\R^{m}$ with induced norm $\vert\cdot\vert$, and a constant $\varsigma>0$ 
so that  for all $(z,v)\in U_{p^*}$ we have
\begin{equation}
\label{eqn:AB estimates}
 A(z,v) \leq 2\,,\quad \ip{B(z,v)v}{v} \leq - \varsigma |v|^2\,.
\end{equation}

Now consider the function
\[ \chi(z,v) \coloneqq z^{\varsigma} |v|^2\,. \]
Along an integral curve we have
\begin{align*}
\chi' &= \varsigma z^{\varsigma-1} z' |v|^2 + 2 z^{\varsigma}\langle v',v\rangle \\
&= z^{\varsigma} \left( \varsigma A |v|^2 +  2\ip{Bv}{v} \right)
\end{align*}
and this is $\leq 0$ on $U_{p^*}$ by \eqref{eqn:AB estimates}.

Therefore, $\chi$ is decreasing along integral curves that run inside $U_{p^*}$.
Choosing $\varrho\coloneqq \varsigma/2$, this implies for such an integral curve $\gamma$ and for any fixed small $z_2>0$:

\begin{equation}
\parbox{\dimexpr\linewidth-4em}{%
    \strut
    If $\gamma$ starts within an $\mathcal{O}(\eps)$ neighborhood of $\ff^*_\calH$ (i.e.\ with $z=\mathcal{O}(\eps)$) and with $v$ bounded then it will arrive at $z=z_2$ within an $\mathcal{O}(\eps^\varrho)$ neighborhood of $\gamma_\ymin^{z_2}$.
    \strut
  }
\label{eq:FocussingStatement}
\end{equation}
 Here and below all statements about $\eps$ are meant for all sufficiently small $\eps$.
By shrinking $U_{p^*}$ we may assume that the following holds for all sufficiently small $E_0>0$ and some fixed $z_2>0$: 
all integral curves $\gamma$ starting in $U_{p^*}$ at $E=E_0$ remain in $U_{p^*}$ until they hit $z=z_2$.

\medskip

The idea of the proof of (a) now is to write the Poincaré map $\mathcal{P}_{z_1}$ as the composition of three maps (see Figure \ref{Fig:PoincareMapProof}), where we write $X^*_{\calH,E=E_0} \coloneqq X^*_\calH \cap \{E=E_0\}$:
\[
\mathcal{P}^\ff\colon X^*_{\calH,0} \to X^*_{\calH,E=E_0}\,,\quad \mathcal{P}_\trans\colon X^*_{\calH,E=E_0} \setminus \ff^*_\calH \to X^*_{\calH,z_2}\setminus M^*_+\,,\quad
\mathcal{P}_{z_2,z_1}\colon X^*_{\calH,z_2} \to X^*_{\calH,z_1} 
\]
suitably localized. 
Here $\mathcal{P}^{\ff}$, $\mathcal{P}_{z_2,z_1}$ are the  Poincar\'{e} maps near the boundary hypersurfaces $\ff^*_\calH$ and $M^*_{\calH,+}$ and away from their intersection, see Figure \ref{Fig:PoincareMapProof}. 
\begin{figure}
\scalebox{1.6}{
\begin{tikzpicture}
\draw[thick,->] (-3,0) coordinate (h1) -- (3,0) coordinate (h2) node[anchor=west] {$y$};
\draw[thick,->] (-3,0) -- (-3,2) node[anchor=south] {$z$};
\draw[thick,->] (-3,0) arc (90:200:1cm) coordinate (h3);
\fill[blue!40,nearly transparent] (-3,0) -- (3,0) -- (3,2) -- (-3,2) -- cycle;
\node[blue!80] at (3,1) [anchor=west] {$M_{+,\calH}^*$};
\draw[thick,-] (-3.9,-1.3)--(2.09,-1.3);
\draw[thick,->] (3,0) arc (90:200:1cm) coordinate (h4);
\path[fill=magenta!40, opacity=.5] (h3) to  (h4) to [ bend left=60] (h2) to  (h1) to [ bend right=60] (h3); 
\node[magenta] at (3,-0.6) [anchor=east]  {$\ff^*_{\calH}$};
\draw[thick,->] (-3.91,-1.3)--(-4.5,-1.9) node[anchor=south east] {$\varepsilon$};
\fill[green!40,opacity=0.5] (-3.9,-1.3) -- (-4.5,-1.9) -- (1.8,-1.9) -- (2.2,-1.3) -- cycle;
\fill[green!50,opacity=0.8] (-3.9,-0.5) -- (2.2,-0.5) -- (2,-0.1) -- (-4.3,-0.1) -- cycle;
\node[green!70!blue] at (2,-1.6) [anchor=west] {$X^*_{\calH,0}$};
\node[green!60!blue] at (-4.5,-0.1) [anchor=east] {$X^*_{\calH,E=E_0}$};
\filldraw[green!70!blue](-0.8,-1.3) circle (0.05);
\draw[magenta] (-0.8,-1.3) .. controls (-0.6,-0.8) and (-0.3,-0.3) .. (0,0);
\fill[orange!60,opacity=0.5] (-3,1.5) -- (3,1.5) -- (2.5,1.0) --  (-3.5,1.0) -- cycle;
\node[orange] at (-3.2,1.35) [anchor=east] {$X^*_{\calH,z_1}$};
\fill[yellow!60,opacity=0.8] (-3,0.6) -- (3,0.6) -- (2.5,0.1) --  (-3.5,0.1) -- cycle;
\node[yellow!75!black] at (-3.3,0.25) [anchor=east] {$X^*_{\calH,z_2}$};
\filldraw[color=black](-1,-1.6) circle (0.05) node[anchor=east] {$(\eps,y,\theta)$};
\draw[black,->] (-1,-1.6) to [out=100,in=-120] (-0.58,-0.43);
\draw[black,->] (-0.58,-0.43) to [out=60,in=-100] (-0.33,0.25);
\draw[black,->] (-0.33,0.25) to [out=100,in=-90] (-0.3,1.15);
\filldraw[color=black](-0.3,1.2) circle (0.05);
\node[anchor=south east]  at (-0.3,0.5) {$\mathcal{P}_{z_2,z_1}$};
\filldraw[color=black](-0.33,0.3) circle (0.05) node[anchor=north east] {$\mathcal{P}_{\mathrm{trans}}$};
\filldraw[color=black](-0.56,-0.4) circle (0.05);
\node[anchor=north east] at (-0.8,-0.5) {$\mathcal{P}^{\ff}$};
\draw[thick,royalpurple] (0,0) .. controls (0.3,0.5) and (-0.2,0.7) .. (0,1.5);
\filldraw[color=royalpurple](0,1.5) circle (0.05) node[anchor=south west] {$\gamma_{y_{\min}}^{z_1}$};
\filldraw[color=red](0,0) circle (0.05) node[anchor=south west] {$\overline{y}_{\min}^*$};
\end{tikzpicture}
}
\caption{The same setting as Figure \ref{Fig:PoincareMap}, with the intermediate level sets $X^*_{\calH,E=E_0}$ and $X^*_{\calH,z=z_2}$, and corresponding Poincar\'{e} maps added. The black dots along the geodesic starting in $(\varepsilon,y,\theta)$ are the successive applications of the Poincar\'{e} maps: $(\eps,y,\theta)$, $\mathcal{P}^{\ff}(\eps,y,\theta)$  and $\mathcal{P}_{\mathrm{trans}}\circ \mathcal{P}^{\ff}(\eps,y,\theta)$ respectively.} 
\label{Fig:PoincareMapProof}

\end{figure}
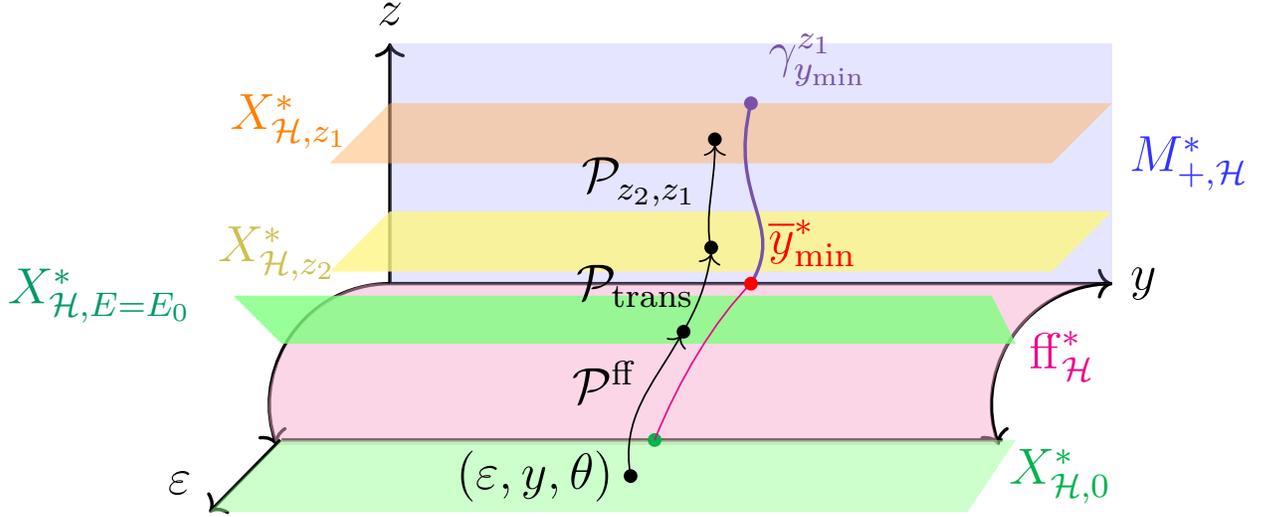
These maps are smooth\footnote{They are well-defined since $E$ is  decreasing (away from $M^*_{\calH,+}$) and $z$ is increasing (away from $\ff^*_\calH$) along the flow, with derivatives bounded away from zero.}, while $\mathcal{P}_\trans$ is the Poincaré map near the edge $\ff^*_\calH\cap M^*_{\calH,+}$, and when localised near $p^*$ is responsible for the focussing by
\eqref{eq:FocussingStatement}.\footnote{Note that $\mathcal{P}_\trans$ is not defined at the boundary $\ff^*_\calH$ since flow lines starting in $\ff^*_\calH$ remain there, hence never reach $z=z_2$. See also Remark \ref{Rem:Dulac} below.}
 Furthermore, $E_0$ and $z_2$ are chosen sufficiently small so the hypothesis of \eqref{eq:FocussingStatement} holds.

More precisely, (a) follows from the following observations. Let $\calO_\ymin$ be as defined above and $K$ as in the theorem. Let $U_{p^*,0} = U_{p^*}\cap \ff^*_\calH$.
\begin{itemize}
\item
 The map $\mathcal{P}^\ff$ sends $K_0\coloneqq K\cap \ff^*_\calH$ into $U_{p^*,0} \cap X^*_{\calH,E=E_0}$, for $E_0>0$ sufficiently small, because $K_0$ is contained in the stable manifold of $p^*$.
\item
Therefore, since $\mathcal{P}^\ff$ is smooth and $K_0$ is compact, it sends an $\eps$-neighbourhood of $K_0$ in $K$ to an $\mathcal{O}(\eps)$-neighbourhood $K'$ of $U_{p^*,0} \cap X^*_{\calH,E=E_0}$ contained in $U_{p^*}$, for all sufficiently small $\eps$.
\item
By \eqref{eq:FocussingStatement}, $\mathcal{P}_\trans$ sends $K'$ to an $\mathcal{O}(\eps^\varrho)$ neighbourhood of $\gamma_\ymin^{z_2}$.
\item
Since $\mathcal{P}_{z_2,z_1}$ is smooth it sends this $\mathcal{O}(\eps^\varrho)$ neighbourhood of $\gamma_\ymin^{z_2}$ to an $\mathcal{O}(\eps^\varrho)$ neighbourhood of $\gamma_\ymin^{z_1}$.
\end{itemize}
This proves (a) for some some $\eps_0$-neighbourhood $K'$ of $K_0$ in $K$, which of course suffices.
\end{proof}
\begin{Rem} \label{Rem:Dulac}
As we remarked above, the map $\mathcal{P}_\trans$ is not defined at the boundary $\ff^*_\calH$. However, near $p^*$ it can be extended  to $U_{p^*}\cap\ff^*_\calH$, by sending any point of this set to $\gamma_\ymin^{z_2}$, and this extension is continuous. The extended map is often called the \textbf{Dulac map} for the hyperbolic critical point $p^*$  in the dynamical systems literature. It is defined if the unstable manifold of $p^*$ has dimension one, as is the case here. It
seems to be folklore that the Dulac map is Hölder continuous, and this would be a stronger statement than our Theorem \ref{Thm:focussing}(a) (it would also give an estimate for the distance between two geodesics with small positive epsilons, rather than only comparing one with $\eps>0$ to the boundary map at $\eps=0$). However, we could not find a reference for such a result (but see \cite{Dulac} in 2 and 3 dimensions), so we decided to include our own proof for the weaker statement that suffices for our purposes.
\end{Rem}
We illustrate Theorem \ref{Thm:focussing} in the case of our Leitmotiv, the surfaces of \eqref{eqn:example}, 
\begin{equation}
\label{eqn:RepeatExample}
M_\eps=\left\{u^2 +\frac{v^2}{1-\delta^2} = z^{2k}+\eps^{2k}\right\}\subset \R^3.
\end{equation}
We show in Appendix \ref{Section:GeneratingExamples} that the induced metric from $\R^3$ fits into our framework and $S^\pm=k(k-1) \vert(u,v)\vert^2$, where $\vert(u,v)\vert$ is the distance from $(u,v,z)\in M_{\eps}$ to the $z$-axis. When the eccentricity does not vanish, $\delta\neq 0$, this is a Morse function with minima and maxima at the intersections of the ellipse with the minor and major axis respectively. Here is what we know and do not know about the 'focussing set' $\mathcal{U}$ in this case. By Theorem \ref{Thm:focussing}, $\mathcal{U}$ is a dense subset of $\ff^*_{\calH,0}$, and its complement is 1-dimensional. Let $\pm y_{\max}=(0,\pm 1)\in \S^1$. Then $(\pm y_{\max},\theta=0)\not \in \calO$, i.e. vertically starting geodesics starting at the major axis do not focus. These geodesics are shown as dashed lines in Figure \ref{fig:EllipsePlots}. We do not in general know if all other vertically starting geodesics focus:
\begin{Quest}
\label{FocussingQuestion}
For the example surface $M_{\eps}$, is $(y,\theta=0)\in \mathcal{U}$ for all $y\neq \pm y_{\max}$?
\end{Quest}
By Remark \ref{Rem:ReachingMax}, the answer would be 'yes' if $S_{\vert \ff}$ were $Z$-independent. We can compute this function explicitly for many examples, including $M_{\eps}$, see \eqref{eq:Sonff} below. The result is $S_{\vert\ff}(Z,y)=a(Z)S^{+}(y)$ for some explicit function $a$ depending on $k$ and $p$ (for $w=\sqrt[p]{\eps^p+\vert z\vert^p}$). The surface $M_{\eps}$ corresponds to $p=2k$, and \eqref{eq:Sonff} tells us that $a$ is not constant for this value. It is not even monotone. The function $a$ is constant for $p=k=2$, i.e. for the surface
\begin{equation}
M_\eps'\coloneqq \left\{u^2 +\frac{v^2}{1-\delta^2} = \left(z^{2}+\eps^{2}\right)^2\right\}\subset \R^3,
\label{eq:M'}
\end{equation}
so here the answer to the question is 'yes'. The argument of Theorem \ref{Thm:Asymptotics} which gives rise to Remark \ref{Rem:ReachingMax} relies on the fact that $\mathcal{G}=S_{\vert \ff}+\vert \theta\vert^2$ is a decreasing function when $a$ is constant. If $a$ were increasing, which it is for $p\leq 2k-2$ but not for $p=2k$, one could modify the arguments of Theorem \ref{Thm:Asymptotics} to show that the answer to Question \ref{FocussingQuestion} would be 'yes'. 
 We do not know if our inability to answer the above question is just a shortcoming of our methods, or if geodesics starting vertically very close to $\pm y_{\max}$ can indeed avoid being focussed.




Figure \ref{fig:EllipsePlots} shows \eqref{eqn:RepeatExample} for $\eps=1, \eps=0.2,$ and $\eps=0.1$. In each case, 10 geodesics starting at equidistributed points along $z=0$ are shown. All start with $\theta(0)=0$. One notices how they accumulate near two geodesics $v=0$  as $\eps\to 0$, corresponding to the minima of $S^+$. Figure \ref{fig:CirclePlot} shows how the same initial values do not produce any focussing when the eccentricity vanishes, $\delta=0$, i.e. when the cross section is a circle. In this case, $S^\pm$ is constant.

\begin{Rem}
\label{Rem:explan eps to 0}
Among the main results of \cite{GrGr15} are the following statements: for a cuspidal singularity (i.e.\ spaces and metrics as in \eqref{eq:MetricAnsatz} with $\eps=0$) geodesics emanating upwards from the singularity $z=0$ must do so from critical points of $S^+$; if $S^+$ is Morse then a unique geodesic emanates from each minimum from $S^+$, while an open dense set emanates from all other critical points combined; and if, in addition, $n=2$ then these geodesics combine to define an exponential map based at the singularity, which is a local homeomorphism if the Hessian of $S^+$ at each minimum has sufficiently small eigenvalues.

Some of these results may be understood in terms of our Focussing Theorem \ref{Thm:focussing}, as we now explain in the case of Example \eqref{eqn:example} with $\delta\neq0$, under the assumption that the answer to Question \ref{FocussingQuestion} is 'yes'.

Let $M^z_\eps$ be the cross section of $M_\eps$ at height $z$.
For each $\eps>0$ consider the projection map $P_\eps$ from $M_\eps^1$ to the waist $M_\eps^0$, defined by letting $P_\eps(p)$ be the point on $M_\eps^0$ closest to $p\in M_\eps^1$. 
This point is unique: if $q$ is any closest point then the minimising geodesic from $p$ to $q$ must hit the waist perpendicularly. So if $q$, $q'$ are two closest points then  the geodesic triangle formed by $p,q,q'$ (recall that the waist is totally geodesic) would have sum of interior angles at least $2\pi$, which is impossible since  $M_\eps$ has negative Gauß curvature everywhere (outside the waist), as one can see by a direct computation. 

Parametrising $M_\eps$ by $\R\times \S^1$ as before, we can write $P_\eps(1,y) = (0,\Ptilde_\eps(y))$ for a map $\Ptilde_\eps:\S^1\to \S^1$. Clearly, $\Ptilde_\eps$ has the two minima of $S^+$ as fixed points. 
Now the Focussing Theorem \ref{Thm:focussing}, combined with the assumption\footnote{Which is true at least for $M_{\eps}'$ defined in \eqref{eq:M'}.} $(y,\theta=0)$ lies in the focussing set for all non-maxima points $y$ of $S^+$,
implies that $\Ptilde_\eps(y)$ converges to one of the maxima of $S^+$ as $\eps\to0$, for each non-minimum $y$ of $S^+$.

Therefore, $\Ptilde_\eps$ converges pointwise to the discontinuous map $\Ptilde_0$ which sends the minima to themselves, the left half to the left maximum and the right half to the right maximum. 

In summary and put differently, the exponential map for $M_0$ based at the singularity  is the limit of the exponential maps for $M_\eps$ based at the waist (with geodesics leaving the waist orthogonally), as $\eps\to0$, and the focussing implies its peculiar behaviour with respect to maxima and minima of $S^+$.
\end{Rem}

\subsection{The case of constant $S^\pm$}
We have mainly considered the case of $S^\pm$ being Morse. One can also say something when $S^\pm$ is constant.

Here is the analogue of Theorem \ref{Thm:Asymptotics}. 
\begin{Thm}
\label{Thm:SconstAsymptotics}
Assume $S^+$ is constant.

Let $\sigma=(Z,y,\theta)$ be a solution of the front face dynamics \eqref{eq:Z'ff}, \eqref{eq:y'ff}, \eqref{eq:theta'ff} (and $\xi=1$) with arbitrary initial value in $\ff^*_{\calH}\setminus M^*_{\calH,-}$. 
 
 Then $\sigma$ converges to $(E=0,y_0,0)$ as $\tau\to \infty$ for some $y_0$. The convergence is exponential.

\end{Thm}  
\begin{proof}
From \eqref{eq:Z'ff} we know $Z'=f(Z)\sim Z\implies Z(\tau)\geq ce^\tau$ for some constant $c>0$ and all sufficiently large $\tau$. Since $S$ is smooth on the front face, we can Taylor expand around $E=0$ and find $S=S^++\er(E)=S^+ + \er(Z^{-1})$. Since $\partial_y S^+=0$, we get $\partial_y S=\er(Z^{-1})$. 
Since $f'(Z)\to 1$, we can choose $\tau$ large enough that $f'(Z)\geq \frac{1}{2}$. The evolution of $\vert \theta\vert^2$ \eqref{eq:ffNormthetaEvol} can then be estimated as
\[ \frac{d}{d\tau} \vert \theta\vert^2 \leq -(2k-1)\vert \theta\vert^2 +2Ce^{-\tau} \vert \theta\vert .\]
Dividing by $\vert \theta\vert$ and introducing $\nu=\frac{2k-1}{2}$, we therefore find
\[\frac{d}{d\tau} \left( e^{\nu\tau} \vert \theta\vert\right)\leq Ce^{(\nu-1)\tau},\]
which integrates to 
\[\vert \theta\vert(\tau)\leq  e^{-\nu(\tau-\tau_0)}\vert\theta\vert(\tau_0)+\frac{C}{\nu-1} e^{-\tau}.\]
This shows that $\theta\to 0$ exponentially. Hence $y$ also converges exponentially by \eqref{eq:y'ff}.  
\end{proof}

\subsection{Hamiltonian for the front face dynamics}

We note as a curiosity that the limiting boundary dynamics can be written as time-dependent Hamiltonian system for $y,\theta$. Indeed, we first observe that the radial motion completely decouples and one can solve for $Z(\tau)$ or $E(\tau)$. We next observe that choosing $\psi(\tau)$ to satisfy 
\[\psi'(\tau)=(2k-1)f'(Z), \quad \psi(0)=0,\]
we achieve
\[e^{\psi}(\theta'+(2k-1)f'(Z) \theta)=(e^{\psi}\theta)'.\]
We can in fact find $\psi$ by using the $Z'$-equation to write 
\[\psi'=(2k-1)f'(Z)=(2k-1)\frac{f'(Z)}{f(Z)} Z'=(2k-1)\frac{d}{d\tau} \ln f(Z(\tau)),\]
hence
\[e^{\psi}=f(Z)^{2k-1}.\]
Rescale the momentum and time once more, setting 
\[\Theta\coloneqq e^{\psi} \theta\]
and
\[\frac{ds}{d\tau}\coloneqq e^{-\psi(\tau)},\quad s(0)=0.\]
We will use a dot for the $s$-derivative\footnote{Since we have not used the original $t$-derivative in a while, this should not cause any confusion.} $\dot{y}\coloneqq \frac{dy}{ds}$ and so on. We then find
\[\dot{y}=\Theta^\#=\mathcal{F}_{\Theta}\]
\[\dot{\Theta}=-\mathcal{F}_{y},\]
where 
\[\mathcal{F}=\mathcal{F}(s,y,\Theta)\coloneqq \frac{1}{2}\left(e^{2\psi} S(y)+\vert \Theta\vert^2\right)=\frac{1}{2}\left(\dot{\tau}(s)^2 S(y)+\vert \Theta\vert^2\right).\]
This is the time-dependent Hamiltonian announced at the outset. 
This is analogous to the way that in the singular case $\eps=0$ the front face dynamics is given by a damped Hamiltonian system with potential $S$, see \cite[Eq. (1.7)]{GrGr15}.

As a concrete example, if $w=\sqrt{z^2+\eps^2}$, then $Z(\tau)=\sinh(\tau)$, $f(Z)=\cosh(\tau)$ and  $f'(Z(\tau))=\tanh(\tau)$.

\appendix
\section{Examples}
\label{Section:GeneratingExamples}
Our general setting arises in embedded examples generalizing \eqref{eqn:example}, as follows. Let $Y\subset\R^{N-1}$ be a smooth closed submanifold of dimension $n-1$. For $w=w(\eps,z)$ as in the main text, let
\[M_{\eps} = \Psi(I\times Y) 
\]
where 
\begin{alignat}{3}
\label{eqn:def beta}
\Psi: \R \times & \R^{N-1} \to \R  \times && \R^{N-1},\quad (z,y) \mapsto (z,w(\eps,z)^k y)
\\
\cup &   && \!\!\!\!\cup  \notag
\\
I \times & Y &&  \!\!\!\!\! M_\eps \notag
\end{alignat}
We recover Example \eqref{eqn:example} by choosing $n=2$, $N=3$, $Y=\left\{(u,v)\,\colon\,  u^2+\frac{v^2}{1-\delta^2}=1\right\}$ an ellipse, and $w^{2k}=\eps^{2k}+z^{2k}$.

Since $\Psi$ is regular except at $(\eps,z)=(0,0)$, the set $M_\eps$ is a smooth submanifold if $\eps>0$. Also, $M_0$ is smooth except for a cuspidal singularity at the origin.

\newcommand{\eucl}{{\mathrm{eucl}}}

On $M_\eps$ we consider the metric induced by the Euclidean metric $g_\eucl$ on $\R^N$. 
We denote its pull-back to $\R\times Y$ by $g_\eps$. 
Since $g_\eps$ is also the restriction of $\Psi^*g_\eucl$ to $I\times Y$, it is in fact a two-tensor on all of $I\times Y$ for all $\eps\in[0,1)$, including at $z=0$, $\eps=0$.\footnote{This tensor has some degree of differentiability, but in general is not smooth, i.e.\ $C^\infty$, since $w$ is not smooth at $(0,0)$. The smoothness can be described most precisely in terms of the blown-up space $X$ introduced in Section \ref{sec:setting}.}

Note that $I\times Y$ is independent of $\eps$, hence smooth even at $\eps=0$.
The fact that $M_0$ has a singularity at the origin is reflected by the fact that $g_0$ is degenerate (not positive definite) at $z=0$.

\begin{Prop}
Let $M_{\eps}\subset \R^{N}$ be the family of manifolds defined above and let $g_{\eps}$ be the metric induced from the Euclidean metric on $\R^{N}$. Then there is a change of $z$-coordinate such that $g_{\eps}$ has the form \eqref{eq:MetricAnsatz} with $\kappa=2k-2$. The function $S$ is
\[S(z,y,\eps) = a(\eps,z)\, \vert y\vert^2 +\mathcal{O}(w^{k-1})\,,\ a = k\left( w w_{zz} + (k-1) w_z^2\right)\]  
where $|y|$ is the Euclidean norm of $y\in Y\subset \R^{N-1}$. Furthermore, the mixed term $w^{2k} b\, dz$ of \eqref{eq:MetricAnsatz} vanishes on the waist $z=0$.
\end{Prop}
Note that $Y$ can be considered as the cross-section of $M_0$ at $z=1$, and then $|y|$ is the distance from $y\in Y$ to the $z$-axis.

\begin{proof}
Fix $\eps>0$ throughout. 
Denote coordinates on the ambient space of $M_\eps$, i.e.\ the image space $\R\times\R^{N-1}$ in \eqref{eqn:def beta}, by $(z,u)$. We first calculate the pull-back of the Euclidean metric $dz^2+|du|^2$ under the map \eqref{eqn:def beta}. We denote this metric on $\R\times\R^{N-1}$ also by $g_\eps$. In the end we will restrict $g_\eps$ to $\R\times Y$.
We write $W=w^k$.
From $\Psi^*du = W_z y\,dz + W\,dy$ we get
$$ g_\eps \coloneqq \Psi^*(dz^2 + |du|^2) = A\,dz^2 + 2 dz\,B_y\cdot dy + W^2 \,|dy|^2 $$
where
$$ A = 1+ W_z^2|y|^2,\quad B =  W W_z\, |y|^2.$$
This has almost the desired form except that $B_y$ only vanishes like $w^{2k-1}$ as $w\to0$ instead of $w^{2k}$. We will see that we can remedy this by replacing the $z$-coordinate by
$$ \zeta = z + B.$$
We define $\Xtilde$ as the blow-up of $[0,1)_\eps\times \R_z\times\R^{N-1}$ in $\{z=\eps=0\}=\{(0,0)\}\times\R^{N-1}$, in analogy to the space $X$ defined in Section \ref{sec:setting}.
Note that $B$ is homogeneous of degree $2k-1$ in $(\eps,z)$ and therefore smooth on  $\Xtilde$, vanishing to order $2k-1$ on the front face.
Using $2k-1\geq2$ one sees easily (for example by a calculation in projective coordinates) that the coordinate change from $(\eps,z,y)$ to $(\eps,\zeta,y)$ induces a diffeomorphism of $\Xtilde$%
, with inverse
$$ z = \zeta - B + O_{4k-3}$$
where we denote by $O_l$ a smooth function or tensor (in terms of the basis $dz$ resp. $d\zeta$ and $dy$) on $\Xtilde$ vanishing to order $l$ at the front face. This is an identity between functions on $\Xtilde$.


\newcommand{\Atilde}{\tilde{A}}
\newcommand{\Btilde}{\tilde{B}}

Inserting $dz = (1-B_\zeta)\,d\zeta - B_y\cdot dy + O_{4k-4}$ into $g_\eps$ 
(with partial derivatives taken in the $(\eps,\zeta,y)$ coordinate system on $\Xtilde$)
we get
after a short calculation that 
$g_\eps = \Atilde \,d\zeta^2 + 2 d\zeta \Btilde\cdot dy + W^2 h$ where
\begin{align*}
\Atilde &= A(1-B_\zeta)^2 + O_{4k-4}
\\
\Btilde &= \left[(1-A)(1-B_\zeta) + O_{4k-4}\right] B_y
\\
h & =  |dy|^2 + (A-2) (B_y\cdot dy)^2 + O_{6k-5} 
\end{align*}
Here we used that $B$ and $B_y$ are $O_{2k-1}$. 
From $1-A=-W_z^2|y|^2 = O_{2k-2}$ we get $\Btilde = O_{4k-3}$, which is even better than the desired order $2k$ of vanishing.
To calculate $\Atilde$, first note that $W_z=\frac{\partial\zeta}{\partial z}W_\zeta= (1+O_{2k-2})W_\zeta$, which yields
$$ \Atilde = 1 - W W_{\zeta\zeta} |y|^2 + O_{3k-3}.$$
Finally, we insert $W=w^k$. Note that the function $w$, when written in $(\eps,\zeta,y)$-coordinates, is not 1-homogeneous in $(\eps,\zeta)$, and also depends on $y$. However, this is irrelevant since 
$$ w(\eps,z) = w(\eps,\zeta- B + O_{4k-3}) = w(\eps,\zeta) + O_{2k-1} $$
where $(\zeta,\eps)\mapsto w(\zeta,\eps)$ is 1-homogeneous and independent of $y$.
To get the statement of the proposition, we change notation from $\zeta$ to $z$ and restrict $\Atilde$, $\Btilde$ and $h$ to $\R\times Y$.

Before changing coordinates, the mixed term reads $2 dz B_y\cdot dy$, where
\[B=W_zW  \vert y\vert^2.\]
Since $z=0$ is a minimum for $W$, $W_z(z=0)=0$. Changing $z$-coordinate to 
\[\zeta=z+B,\]
the new mixed term has the form $2d\zeta \tilde{B}_y\cdot dy$, with
\[\tilde{B}=\left[(1-A)(1-B_{\zeta})+O_{4k-4}\right]B_{y}.\]
Since $\zeta(z=0)=0$ and $B_y(z=0)=0$, also $\tilde{B}$ vanishes at $\zeta=z=0$.
\end{proof}

\begin{Rem}
It is natural to ask if the result remains true if one uses a different metric on $\R^N$. 
As shown in \cite{BeyGri:IGIC} (for the singular space at $\eps=0$), this is the case if the $z$-axis $\R\times \{0\}^{N-1}\subset \R^{N}$ is a geodesic. 
With other metrics on $\R^{N}$, one can get other values of $\kappa$ in the range $k\leq \kappa\leq 2k-2$. We expect that the corresponding statement still holds in the $\eps$-dependent case.
\end{Rem}
We now check assumption \eqref{eqn:S assumption}.
\begin{Lem}
\label{Lem:ExampleS}
When $w=\sqrt[p]{\eps^p+\vert z\vert^p}$, $S_{\vert \ff}$ satisfies 
\[\partial_Z S_{\vert \ff}=\mathcal{O}(Z^{-1-p}),\quad \vert Z\vert \to \infty,\]
and
\[S_{\vert \ff\cap M_{\pm}}(y)=k(k-1)\vert y\vert^2,\]
where $\vert y\vert^2$ is the squared distance from $Y$ to the $z$-axis in $\R^{N}$ as above. 

\end{Lem}
In terms of the coordinate $E$ we have $Z=E^{-1}$, so $\partial_Z=-E^{2}\partial_E$, so $\partial_E S_{|\ff} = \mathcal{O}(E^{p-1})$. This implies that assumption \eqref{eqn:S assumption} is satisfied if $p>1$. 
\begin{proof}
To avoid a lot of signum functions, we consider $z>0$. Then
\[w_z=\left(\frac{z}{w}\right)^{p-1},\quad  ww_{zz}=(p-1)\frac{\eps^p z^{p-2}}{w^{2p-2}}.\] 
In terms of the coordinate $Z$, we therefore get 
\begin{equation}
S_{\vert \ff}=k\vert y\vert^2 \frac{Z^{p-2}}{f(Z)^{2p-2}}\left((k-1)Z^p+(p-1)\right)
\label{eq:Sonff}
\end{equation}

for the restriction to the front face. The $Z$-derivative of this is $\mathcal{O}(Z^{-1-p})$, as one readily checks. Restricting to $\ff\cap M_{\pm}$ corresponds to sending $\vert Z\vert\to \infty$. Since $f(Z)\sim \vert  Z\vert$, the claim follows.

\end{proof}
\begin{Rem}
The assumptions of Theorem \ref{Thm:Asymptotics} are satisfied for all these families of examples. Whether $S_{\vert \ff \cap M_{\pm}}$ is Morse or not depends on the cross section $Y$. When the cross section is an ellipse, $S$ is Morse with minima and maxima corresponding to the intersections with the minor and major axes respectively.
\end{Rem}

\subsection{Ellipse as cross section}
\label{Section:Elliptic}
Here we gather a few formulas used for the numerical calculations of our main example \eqref{eqn:example}. Let $W=w^k$ (one could allow for more general functions). Let 

\begin{equation}
M_{\eps}\coloneqq \left\{(u,v,z)\in \R^3 \, \colon\, u^2+\frac{v^2}{1-\delta^2}=W(\eps,z)^2\right\}
\label{eq:Ellipsoidal}
\end{equation}
as before.

We introduce adapted polar coordinates
\[u=W\cos\phi,\quad v=\sqrt{1-\delta^2}W\sin \phi.\]
Then 
\[du^2+dv^2=dW^2+W^2 d\phi^2-\delta^2(\sin\phi \,dW +W \cos \phi \,d\phi)^2,\]
hence
\[g=(1+W_z^2(1-\delta^2\sin^2\phi))dz^2   -2\delta^2W W_z\sin \phi \cos\phi \,dz d\phi +W^2(1-\delta^2 \cos^2\phi) d\phi^2.\]

\begin{Lem}
The Hamiltonian for $g$ reads
\begin{equation}
2\calH=\frac{\eta^2}{W^2(1-\delta^2\cos^2\phi)}+\frac{\left(\xi+\delta^2 \frac{W_z}{W}\frac{\sin\phi \cos\phi}{(1-\delta^2\cos^2\phi)}\eta\right)^2}{1+\frac{W_z^2(1-\delta^2)}{1-\delta^2\cos^2\phi}}.
\end{equation}
\end{Lem}
\begin{proof}
Same proof as for \eqref{eq:ExactHamiltonianEta}.
\end{proof}

Introduce the auxiliary function $\psi$ via
\[\exp(2\psi)=1-\delta^2 \cos^2\phi.\]
This will help simplify some equations.
The Hamiltonian equations of motion then read 
\[\dot{z}=\frac{\xi+\delta^2 \frac{W_z}{W}\frac{\sin\phi \cos\phi}{(1-\delta^2\cos^2\phi)}}{1+\frac{W_z^2(1-\delta^2)}{1-\delta^2\cos^2\phi}}=\frac{\xi e^{2\psi}+\delta^2 \frac{W_z}{W}\sin\phi\cos\phi}{e^{2\psi}+W_z^2(1-\delta^2)},\]
\[\dot{\phi}=\frac{1}{W(1-\delta^2\cos^2\phi)}\left(\frac{\eta}{W}+\delta^2 W_z \sin\phi\cos\phi \dot{z}\right)=\frac{\eta}{W^2} e^{-2\psi}+\frac{W_z}{W}\dot{z} \psi_{\phi},\]
\[\dot{\xi}= \frac{e^{-2\psi}\eta^2}{W^2} \frac{W_z}{W}-\psi_{\phi}\eta \left(\frac{W_z}{W}\right)_z \dot{z}+W_z W_{zz} (1-\delta^2)e^{-2\psi}\dot{z}^2,\]
\begin{align*}\dot{\eta}&=\psi_{\phi}\frac{e^{-2\psi}\eta^2}{W^2}-\frac{W_z}{W}\psi_{\phi\phi}\eta \dot{z}-W_z^2(1-\delta^2)e^{-2\psi}\psi_{\phi} \dot{z}^2\\
& =\psi_{\phi}\frac{e^{-2\psi}}{W^2(1+W_z^2(1-\delta^2)e^{-2\psi})}\left(\eta^2-1\right)-\frac{W_z}{W}\psi_{\phi\phi}\eta \dot{z}.
\end{align*}

If one prefers the Lagrangian picture, the equations are
\[\ddot{z}+\frac{(1-\delta^2)W_z}{e^{2\psi}+W_z^2(1-\delta^2)}\left(W_{zz}\dot{z}^2-W\dot{\phi}^2\right)=0\]
\[\ddot{\phi}+ \frac{\psi_{\phi}}{e^{2\psi}+W_z^2(1-\delta^2)}\left(\dot{\phi}^2-\frac{W_{zz}}{W} \dot{z}^2\right) +2\frac{W_z}{W}\dot{z}\dot{\phi}=0.\]

The plots shown in Figure \ref{fig:WindingExamples} and \ref{fig:EllipsePlots} were produced by numerically solving this coupled ODE using the Euler method and MATLAB.


\begin{thebibliography}{99}

\bibitem[\textsc{ARS21}]{ARS} Pierre Albin, Fr\'{e}d\'{e}ric Rochon,
and David Sher, \emph{Resolvent, heat kernel and torsion under 
degeneration to fibered cusps},  Memoirs of the American Mathematical Society
Volume: 269, (2021). 


\bibitem[\textsc{BeLy07}]{BerLyt:TSGHLSS}
Andreas Bernig and Alexander Lytchak.
\newblock Tangent spaces and {G}romov-{H}ausdorff limits of subanalytic spaces.
\newblock {\em J. Reine Angew. Math.}, 608, pp. 1--15, (2007).

\bibitem[\textsc{BeGr25}]{BeyGri:IGIC}
Antje Beyer and Daniel Grieser, \emph{The inner geometry of incomplete cusps},
manuscript in preparation.


\bibitem[\textsc{BoNa01}]{Dulac}
Patrick Bonckaert and Vincent Naudot, \emph{Asymptotic properties of the Dulac map near a
hyperbolic saddle in dimension three},
Ann. Fac. Sci. Toulouse, 6e série, tome 10,
no 4, pp. 595-617, (2001).


\bibitem[\textsc{BePo82}]{DWarped} John K. Beem and T. G. Powell, \emph{Geodesic Completeness and Maximality in Lorentzian
Warped Products}, Tensor (N. S.), 39, pp. 31--36 (1982).


\bibitem[\textsc{Che87}]{CheegerEta}
Jeff Cheeger, \emph{Eta invariants, the adiabatic approximation and conical singularities.}
J. Differential Geom., 26, pp. 175--211, (1987).

\bibitem[\textsc{Dai06}]{DaiEta}
Xianzhe Dai, \emph{Eta invariants for manifold with boundary}, Analysis, Geometry and Topology of Elliptic Operators, pp. 141--172, (2006).

\bibitem[\textsc{DaMe12}]{DaiMel}
Xianzhe Dai and Richard ~B. Melrose, \emph{Adiabatic Limit, Heat Kernel and Analytic Torsion}, In: Dai, X., Rong, X. (eds) Metric and Differential Geometry. Progress in Mathematics, vol 297, pp. 233--298 Birkhäuser, (2012).

\bibitem[\textsc{DoMcgo95}]{hyperbolic}
Jozef Dodziuk and Jeffrey Mcgowan, \emph{The Spectrum of the Hodge Laplacian for a Degenerating Family of Hyperbolic Three Manifolds}, Trans. Amer. Math. Soc. 347(6), pp. 1981--1995, (1995).

\bibitem[\textsc{GHL04}]{GHL}
Sylvestre Gallot, Dominique Hulin and Jacques Lafontaine,
\textit{Riemannian Geometry (Third Edition)},
Springer, (2004).

\bibitem[\textsc{GrGr15}]{GrGr15}
Vincent Grandjean and Daniel Grieser,
\textit{The exponential map at a cuspidal singularity},
    Journal f\"ur die reine und angewandte Mathematik, Volume 2018, Issue 736,  pp. 33--67. (2015).
    
    \bibitem[Gri01]{Gri:BBC}
Daniel ~Grieser, \emph{Basics of the {$b$}-calculus}, In J.~Gil, D.~Grieser, and M.~Lesch, editors, {\em Approaches to
  Singular Analysis}, Advances in Partial Differential Equations, pp. 30--84,  Birkh\"auser, (2001).



\bibitem[\textsc{Gri11}]{Gri:NDOCS}
Daniel Grieser, \emph{A natural differential operator on conic spaces},
{\em Discrete Contin. Dyn. Syst.}, Dynamical systems, differential
  equations and applications. 8th AIMS Conference. Suppl. Vol. I, pp. 568--577,
  (2011).
  
\bibitem[\textsc{Gri17}]{GriQuasimodes}
  Daniel Grieser, \emph{Scales, blow-up and quasimode constructions}, in Geometric and Computational Spectral Theory, Contemporary Mathematics, AMS, (2017). 
  
\bibitem[\textsc{GrJe07}]{GriJer}
Daniel Grieser and David Jerison, \emph{Asymptotics of eigenfunctions on plane domains},
    Pacific Journal of Mathematics 240(1), pp. 109--133, (2007).  
    
\bibitem[\textsc{GrLy24}]{GriLye}
Daniel Grieser and Jørgen Olsen Lye,
\textit{Geodesics orbiting a singularity},
J. Geom. 115, 1, (2024).


\bibitem[\textsc{Gro92}]{Gromov}
Mikhael Gromov, \emph{Spectral geometry of semi-algebraic sets}, Ann. Inst. Fourier, 42, pp. 249--274, (1992).



\bibitem[\textsc{MaMe90}]{MazMel}
Rafe R. Mazzeo and  Richard B. Melrose, \emph{The adiabatic limit, Hodge cohomology and Leray's spectral sequence for a fibration},
J. Differential Geom. 31(1): pp. 185--213, (1990).


\bibitem[Mel96]{Mel:DAMWC}
Richard ~B. Melrose, \emph{Differential analysis on manifolds with corners}, Book in preparation. http://www-math.mit.edu/$\sim$rbm/book.html, (1996).

\bibitem[\textsc{Mel08}]{Mel:RBIASS}
Richard ~B. Melrose, \emph{Real blow ups: Introduction to analysis on singular spaces}, Notes for lectures at MSRI, http://www-math.mit.edu/$\sim$rbm/InSisp/InSiSp.html, (2008).

\bibitem[\textsc{MeWu04}]{MeWu:Geo}
Richard B. Melrose and Jared Wunsch, \emph{Propagation of singularities for the wave equation on conic manifolds}, Inventiones mathematicae, volume 156, pp. 235–299, (2004).

\bibitem[\textsc{Mil63}]{MilnorMorse}
John Milnor, \emph{Morse Theory}, Princeton University Press, (1963).

\bibitem[\textsc{Sam79}]{Sam79} V. S. Samovol, \emph{Linearization of systems of differential equations in the neighborhood of invariant toroidal
manifolds}, Proc. Moscow Math. Soc. 38, pp. 187--219, (1979).

\bibitem[\textsc{See92}]{SeeleyConic} Robert Seeley, \emph{Conic degeneration of the Gauss-Bonnet operator}, J. Anal. Math. 59, pp. 205--215, (1992).

\bibitem[\textsc{Sma63}]{Smale} Stephen Smale, \emph{Stable manifolds for differential equations and diffeomorphisms}, Annali della Scuola Normale Superiore di Pisa, Classe di Scienze 3e série, tome 17,
no 1-2, pp. 97--116, (1963).

\bibitem[\textsc{Sto82}]{Sto:EMISP}
David~A. Stone.
\newblock {\em The exponential map at an isolated singular point}.
\newblock Number 256 in Memoirs of the American Mathematical Society. American
  Mathematical Society, (1982).
  
\bibitem[\textsc{Tes11}]{Teschl}
Gerald Teschl, \emph{Ordinary Differential Equations and Dynamical Systems}, AMS, (2011).

\bibitem[\textsc{Won43}]{Won43} Yung-Chow Wong, \emph{Some Einstein Spaces with Conformally Separable Fundamental Tensors}, Trans. Am. Math. Soc., Vol. 53, No. 2, pp. 157--194 (1943).

\bibitem[\textsc{Yan40}]{Yan40} Kentaro Yano, \emph{Conformally separable quadratic differential forms}, Proc. Imp. Acad. 16(3): pp. 83--86 (1940).


  

\bibitem[\textsc{Yos97}]{Yoshikawa}
Ken-Ichi Yoshikawa, \emph{Degeneration of Algebraic Manifolds and the Spectrum of Laplacian}, Nagoya Mathematical Journal 146, pp. 83--129, (1997). 




\end{thebibliography}
\end{document}